%% file: main.tex
\documentclass[11pt]{article}
\usepackage{graphicx} 
\usepackage[margin=1in]{geometry}
\usepackage[utf8]{inputenc}
\usepackage[english]{babel}
\usepackage{amssymb}
\usepackage{amsmath}
\usepackage{amsfonts}
\usepackage{bbm}
\usepackage{booktabs}
\usepackage{array}
\usepackage{verbatim}
\usepackage{xcolor}
\usepackage{geometry}
\usepackage{setspace}
\usepackage{amsthm}
\graphicspath{ {./images/} }
\usepackage{graphicx}
\usepackage{float} 
\usepackage{natbib}
\usepackage{enumitem}
\usepackage{float}
\usepackage{csquotes}

\usepackage{setspace}
\usepackage{authblk}
\usepackage{orcidlink}

\usepackage{bm}
\usepackage{arydshln}
\usepackage{multirow}
\usepackage{subcaption}
\usepackage[normalem]{ulem}
\definecolor{mygreen}{rgb}{0.0,0.55,0.0}
\definecolor{myorange}{rgb}{0.95,0.5,0.0}

\usepackage{longtable}
\usepackage{booktabs}

\usepackage{nicefrac}       

\usepackage{algorithm}
\usepackage{algorithmic}
\usepackage{enumitem}
\usepackage{csquotes}
\usepackage{tikz}
\usetikzlibrary{
  shapes.geometric,
  arrows.meta,
  positioning,
  decorations.pathreplacing,
  decorations.markings,
  calc,
  fit,
  backgrounds,
  fadings,
  patterns,
  shadows
}

\usepackage{array}

\DeclareMathOperator*{\argmin}{arg\,min}

\theoremstyle{definition}
\newtheorem{theorem}{Theorem}[section]
\newtheorem{definition}[theorem]{Definition}
\newtheorem{assumption}[theorem]{Assumption}
\newtheorem{lemma}[theorem]{Lemma}
\newtheorem{remark}[theorem]{Remark}
\newtheorem{proposition}[theorem]{Proposition}
\newtheorem{corollary}[theorem]{Corollary}

\numberwithin{equation}{section}

\title{Generative Transfer for Entropic Optimal Transport with Unknown Costs}

\author{
Antoine Debouchage$^{1,2,*}$,
Xiaozhen Wang$^{3,*}$,
Zhenjie Ren$^{2,\dagger}$,
Francois Buet-Golfouse$^{1}$ \\
{\small
$^{1}$ AIML Global Markets, Barclays \\
$^{2}$ LaMME, Université Évry Paris-Saclay, Évry-Courcouronnes, France \\
$^{3}$ CEREMADE, Université Paris-Dauphine PSL, Paris, France
}
}
\date{}

\begin{document}

\maketitle

\begingroup
\renewcommand{\thefootnote}{\fnsymbol{footnote}}
\footnotetext[1]{These authors contributed equally as co-first authors.}
\footnotetext[2]{Supported by the Finance For Energy Market Research Centre, the France 2030 grant (ANR-21-EXES-0003), and the PEPR PDE-AI project.}
\endgroup

\begin{abstract}
    This paper addresses the practical challenge in Entropic Optimal Transport (EOT) where the underlying ground cost function is typically latent and unobserved.
    Rather than assuming a fixed geometric cost, we adopt a data-driven approach where a shared cost is revealed only through samples from a reference optimal coupling.
    The question is then: given samples from a reference optimal coupling, can we recover the optimal coupling for new marginals under the same latent cost?
    We introduce a generative transfer framework that recovers the optimal coupling for new marginals by utilizing an iterative path-wise tilting algorithm.
    Unlike static importance reweighting, this method evolves the coupling jointly with a marginal transport path, allowing mass to move beyond the reference support.
    We derive sample-level learning rules for these infinitesimal updates, which yield covariance-type evolution equations for the associated transport vector fields.
    By integrating this dynamics with Conditional Flow Matching (CFM), we produce a practical sampler for paired data.
    Finally, we provide theoretical guarantees establishing a global convergence rate of $\mathcal{O}(\delta)$, ensuring the generated coupling converges to the target EOT plan in $W_1$ distance.
\end{abstract}

\section{Introduction}

Generative models built from dynamical transport, notably diffusion and flow-based models, have become a standard way to map a simple base distribution to complex data \citep{ho2020denoising,song2020score,rezende2015variational}.
In many applications, the supervision consists of unpaired samples from a source and a target distribution, and learning amounts to constructing a time-dependent map that pushes one distribution to the other.
Score-based and flow-matching objectives make this practical, providing scalable samplers and learned vector fields that connect the two marginals \citep{lipman2022flow,tong2023improving}.

Optimal transport (OT) plans form a canonical notion of \textit{alignment} between probability distributions and have become a useful primitive in generative modeling \citep{villani2021topics,peyre2019computational}.
Most scalable neural approaches to compute OT rely on dual formulations \citep{peyre2019computational}, which have been successfully integrated into generative pipelines.
Entropic optimal transport (EOT), together with its dynamic formulation as a Schr\"odinger bridge (SB), yields cost-aware entropic couplings whose conditional laws can exhibit one-to-many behavior, with diversity controlled by the regularization level, while remaining computationally tractable and differentiable via Sinkhorn-type solvers \citep{cuturi2013sinkhorn,conforti2019second}.

A practical obstacle in EOT is that the underlying cost function is typically latent.
For convenience, quadratic costs are often used even when they distort the task geometry.
We therefore adopt a data-driven view: the cost is unobserved and only revealed through the optimal couplings it induces.

Unlike most EOT works that focus on varying marginals, we consider a shared but unknown cost function across tasks.
The cost is never observed directly; its structure is entirely encoded in a reference optimal coupling.
To recover the optimal coupling for new marginals, we introduce path-wise tilting, an iterative algorithm that evolves a coupling jointly with the marginals along a path.
Because mass can move beyond the reference support, this avoids the failures of static importance reweighting. We refer to our method as \textbf{TACO} (\textbf{T}ilting  \textbf{A}daptation of \textbf{C}ouplings for entropic \textbf{Optimal} Transport).

\paragraph{Contributions.}
\begin{itemize}[leftmargin=1.3em,itemsep=0pt,topsep=2pt]
  \item \textbf{Cost-transfer setting.} We formalise \textit{data-based cost-transfer} for EOT, where a shared, unobserved cost is accessed solely via samples from a reference coupling.
  \item \textbf{Path-wise tilting.} We propose an iterative algorithm that avoids the support-mismatch of static reweighting by evolving the coupling jointly with a marginal transport path.
  \item \textbf{Generative sampler.} TACO is the first generative framework to output a continuous coupling sampler (paired samples) for EOT with a fully latent cost.
  \item \textbf{Convergence guarantee.} We prove a global $\mathcal{O}(\delta)$ convergence rate in $W_1$ distance (Theorem~\ref{thm:global_tv_convergence}-\ref{thm:main_convergence}).
\end{itemize}

\section{Related Work}
\label{sec:related}

\paragraph{Entropic OT and Schrödinger bridges.}
Sinkhorn algorithms \citep{cuturi2013sinkhorn} and neural dual solvers \citep{seguy2018largescale,makkuva2020optimal} scale EOT to high dimensions. Recent Schrödinger Bridge (SB) methods, including DSB \citep{debortoli2021diffusion}, DSBM \citep{shi2023diffusion}, LightSB \citep{gazdieva2024light}, and SB-Flow \citep{debortoli2024sbflow}, parameterize diffusion paths, while others utilize score matching \citep{gushchin2023entropic} or progressive $\varepsilon$-scheduling \citep{kassraie2024progot}. \textbf{All these methods require a known cost and must be retrained for every new marginal pair.}

\paragraph{Flow matching and generative EOT.}
Flow matching (FM) \citep{lipman2022flow,liu2022flow,albergo2022building} enables simulation-free training; OT-CFM \citep{tong2023improving} improves this with OT-reindexed priors but still requires a specified cost. \textbf{GENOT} \citep{klein2024genot} is the closest generative EOT framework, supporting general costs and Gromov-Wasserstein, but it handles only fixed marginal pairs and lacks transferability. While Meta OT \citep{amos2023meta} amortizes iterations and CondOT \citep{bunne2022condot} handles context, they generally assume known costs. Identifiability results \citep{gonzalez2024identifiability} provide the theoretical basis for our latent-cost approach.

\paragraph{Tilting, change of measure, and applications.}
Tilting is classical in sequential importance reweighting \citep{moral2004feynman} and has recently been adapted for normalizing flows \citep{potaptchik2025tilt} and SB posteriors \citep{pachebat2025iterative}. Our path-wise tilting extends these to a continuous-time Euler scheme with convergence guarantees. In single-cell biology, TACO amortizes via cost transfer, avoiding the per-condition retraining required by CellOT \citep{bunne2022supervised}. Furthermore, TACO replaces the adversarial objectives of Cycle-GAN \citep{zhu2017unpaired} and UNSB \citep{kim2022maximum} with a principled dual tilting objective.

\section{Preliminaries}

\subsection{Entropic Optimal Transport}
\label{sec:prelim-eot-sb}
Entropic optimal transport (EOT), also known as the static Schr\"odinger problem, is a scalable entropy-regularized variant of optimal transport \citep{cuturi2013sinkhorn,peyre2019computational}.
A substantial theory has been developed for OT and its constrained variants, including martingale and Markov formulations relevant to finance and stochastic control \citep{carlier2017convergence,nutz2021introduction}.
In recent years, EOT has also been extended beyond the classical two-marginal setting, including entropy-regularized Martingale optimal transport \citep{henry2019martingale,chen2026convergence} and convex variants of entropic OT \citep{kazeykina2025entropic}.

Let $\mathcal{X}\subset\mathbb{R}^d$ be compact and convex, $\mu\in\mathcal{P}(\mathcal{X})$ and $\nu\in\mathcal{P}(\mathcal{X})$ and let $c:\mathcal{X}\times\mathcal{X}\to\mathbb{R}$ be a ground cost. For $\varepsilon>0$, the entropic OT problem is
\begin{equation}
\pi^\star \in \argmin_{\pi\in\Pi(\mu,\nu)}
\Big\{\langle c,\pi\rangle+\varepsilon\,\mathrm{H}\big(\pi\big)\Big\},
\label{eq:eot-primal}
\end{equation}
where $\Pi(\mu,\nu)$ denotes the set of couplings with marginals $(\mu,\nu)$ and $H$ is the entropy defined as
\[
H(\pi)
:=
\begin{cases}
\displaystyle \int_{\mathcal{X}\times\mathcal{X}}
\pi(x,y)\,\log \pi(x,y)\,dx\,dy,
& \text{if }\pi\ll \mathcal{L}^{2d},\\[1.2ex]
+\infty, & \text{otherwise.}
\end{cases}
\]
The dual problem of \eqref{eq:eot-primal} can be written in terms of potentials $(f,g)$ as
\begin{equation}
\sup_{f,g}\;
\Big\{\int_{\mathcal{X}} f\,\mathrm{d}\mu + \int_{\mathcal{X}} g\,\mathrm{d}\nu
-\varepsilon \int_{\mathcal{X}\times\mathcal{X}}
\exp\!\Big(\tfrac{f(x)+g(y)-c(x,y)}{\varepsilon}\Big)\,\mathrm{d}x\mathrm{d}y\Big\}.
\label{eq:eot-dual}
\end{equation}
First-order optimality yields the log-form condition for the optimal coupling $\pi^*$ and the optimal dual $f^*, g^*$
\begin{equation}
\log \pi^*(x,y) + \frac{1}{\varepsilon}\Big(c(x,y)-f^\star(x)-g^\star(y)\Big) = C.
\label{eq:eot-foc-log}
\end{equation}
where $C$ is a normalisation constant.




\subsection{Conditional Flow Matching (CFM)}
Flow Matching (FM) \citep{lipman2023flowmatchinggenerativemodeling} learns a velocity field $\beta_t$ such that the flow $\partial_t\rho_t+\nabla\cdot(\rho_t\beta_t)=0$ transports $\rho_0=\mu$ to $\rho_1=\nu$.
While FM matches \textit{marginals}, it typically fails to preserve a specific reference coupling $\pi \in \Pi(\mu,\nu)$.

\noindent\textbf{CFM and Lifting.} To sample from a target coupling $\pi$ directly, CFM lifts the state to $z=(z^{(1)},z^{(2)})\in\mathbb{R}^{2d}$ and freezes the second coordinate.
For $(X,Y)\sim\pi$, define the lifted endpoints $Z_0:=(X,X)$ and $Z_1:=(Y,X)$. For $t\in[0,1]$, we define the Gaussian interpolation $I_t := (1-t)Z_0+tZ_1+\sqrt{t(1-t)}\,\Xi$ and conditional velocity field $\beta_t(z) := \mathbb{E}[\dot I_t \mid I_t=z]$, where $\Xi\sim\mathcal N(0,\sigma^2 I_{2d}) \perp (Z_0,Z_1)$.
We then simulate the "frozen" ODE $\frac{d}{dt}Z_t=(\beta_t(Z_t),0)$ with $Z_0 \sim (X,X)_{\#}\pi$.

\noindent\textbf{Implementation.} In practice, given samples $(x_i)_i \sim \mu$, we initialize $Z_0$ from the empirical diagonal measure $\frac{1}{n}\sum_{i=1}^n\delta_{(x_i,x_i)}$.
Because $Z_t^{(2)}$ remains constant, the terminal pair provides a direct sampler for the coupling: $(Z_1^{(2)}, Z_1^{(1)}) \sim (X, Y) \sim \pi$.

\section{Method: Transfer by tilting}
\subsection{Problem Setup and Tilt Relation}
We consider two entropic OT problems that share the same unknown cost function $c$. The reference task is $\mathrm{EOT}(\mu',\nu')$, and the target task is $\mathrm{EOT}(\mu,\nu)$. The cost $c$ is never observed; instead, we assume access to samples from an optimal coupling $\pi'\in\Pi(\mu',\nu')$ of the reference problem, and our goal is to construct a sampler for the target optimal coupling $\pi^\star\in\Pi(\mu,\nu)$ of the target problem under the same hidden cost.

Let $(f',g')$ and $(f^\star,g^\star)$ denote optimal dual potentials for the reference and target problems, respectively for the same regularization $\varepsilon>0$. The log-form first-order conditions read
\begin{align}
\log \pi'(x,y) + \frac{1}{\varepsilon}\big(c(x,y)-f'(x)-g'(y)\big) &= C', \label{eq:foc-ref}\\
\log \pi^\star(x,y) + \frac{1}{\varepsilon}\big(c(x,y)-f^\star(x)-g^\star(y)\big) &= C, \label{eq:foc-tgt}
\end{align}
for constants $C,C'$. Subtracting \eqref{eq:foc-ref} from \eqref{eq:foc-tgt}, the shared cost $c$ cancels and we obtain
\[
\log\frac{\pi^\star(x,y)}{\pi'(x,y)}
=\frac{1}{\varepsilon}\Big((f^\star-f')(x)+(g^\star-g')(y)\Big) + (C'-C).
\]
Equivalently, there exist functions $\varphi,\psi:\mathcal X\to\mathbb{R}$ such that
\begin{equation}
\label{eq:tilt-relation}
\pi^\star(x,y)\ \propto\ \pi'(x,y)\,\exp\!\big(\varphi(x)+\psi(y)\big),
\qquad
\varphi=\tfrac{1}{\varepsilon}(f^\star-f'),\ \ \psi=\tfrac{1}{\varepsilon}(g^\star-g').
\end{equation}
where $\propto$ denotes equality up to a normalizing constant ensuring $\pi^\star$ integrates to one. Thus, transferring from the reference to the target task amounts to finding a tilt (a pair of additive potentials) that modifies $\pi'$ into a coupling with marginals $(\mu,\nu)$ while remaining consistent with the same hidden cost.

\begin{remark}[Why one-shot reweighting fails]
\label{rem:one-shot-fails}
A natural attempt is to directly estimate $(\varphi,\psi)$ by maximising the entropy-regularised dual \eqref{eq:eot-dual} using only samples from $\pi'$ as the reference kernel.
This is well-defined \textit{only} if the support of $\pi'$ covers that of $\pi^\star$.
When the new marginals $(\mu,\nu)$ differ substantially from $(\mu',\nu')$, the support of $\pi^\star$ may extend beyond that of $\pi'$, making one-shot reweighting unreliable.
Our path-wise tilting strategy explicitly addresses this by growing the support gradually along a connecting path.
\end{remark}

\begin{figure}[htbp]
    \centering
    \input{misc/tikz_transposed}
    \caption{\textbf{Cost transfer via tilting along a marginal path.}
    We observe a reference optimal coupling $\pi'$ for $(\mu',\nu')$ and seek the target optimal coupling $\pi^\star$
    for $(\mu,\nu)$ under the same unknown cost. We bridge the two tasks through intermediate marginals
    $(\mu_t,\nu_t)$ and couplings $\pi_t$.}
    \label{fig:pathwise-tilting-schematic}
\end{figure}
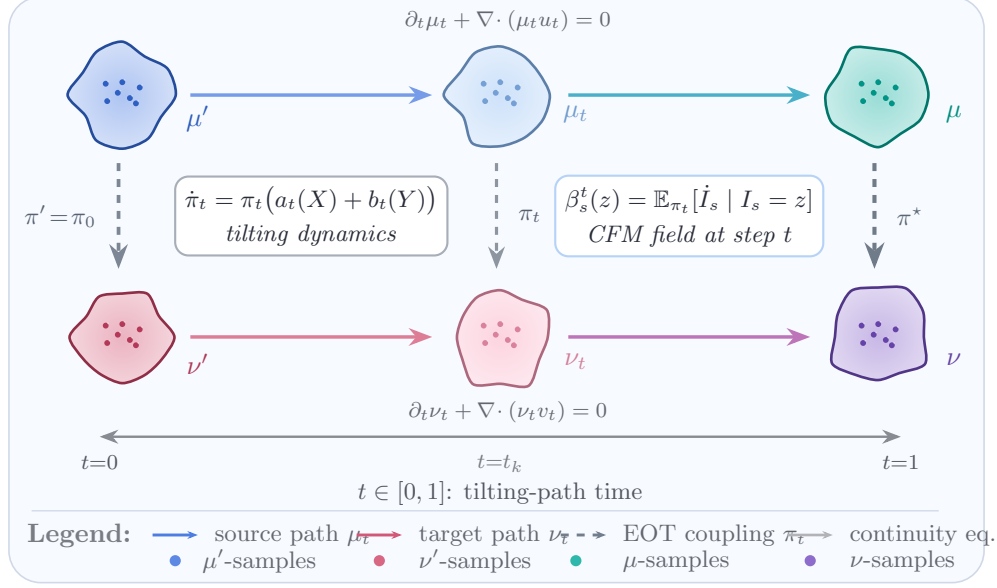

\subsection{Infinitesimal Tilting Dynamics along a Marginal Path}
\paragraph{Marginal path driven by flow matching.}
We connect the reference marginals $(\mu',\nu') = (\mu_0,\nu_0)$ to the target marginals $(\mu,\nu) = (\mu_1,\nu_1)$ through a path $(\mu_t,\nu_t)_{t\in[0,1]}$ generated by vector fields $(u_t,v_t)$:
\begin{equation}
\label{eq:cont-eq-uv}
\partial_t \mu_t + \nabla\!\cdot(\mu_t u_t)=0,
\qquad
\partial_t \nu_t + \nabla\!\cdot(\nu_t v_t)=0,
\end{equation}
In practice, $u_t$ and $v_t$ can be learned by CFM between $(\mu',\mu)$ and $(\nu',\nu)$, yielding an explicit way to evolve the marginals along $t$.

\paragraph{EOT couplings along the path and infinitesimal tilts.}
Let $\pi_t$ be the entropic OT coupling between $\mu_t$ and $\nu_t$ for the \textit{same} unknown cost $c$ and fixed $\varepsilon>0$. By \eqref{eq:tilt-relation} there exist potentials $(\varphi_t,\psi_t)$ such that
\begin{equation}
\label{eq:schrod-form-path}
\pi_t(x,y)= \pi'(x,y) \exp\Big(\varphi_t(x)+\psi_t(y)\Big).
\end{equation}
Define the time-derivatives of the potentials as infinitesimal tilting rates
\begin{equation}
\label{eq:def-a-b-path}
a_t(x):=\dot\varphi_t(x),
\qquad
b_t(y):=\dot\psi_t(y).
\end{equation}
Differentiating \eqref{eq:schrod-form-path} yields the tilting form of the coupling evolution
\begin{equation}
\label{eq:pi-dot}
\dot\pi_t(x,y)=\pi_t(x,y)\big(a_t(x)+b_t(y)\big).
\end{equation}
Hence, $\pi_{t+\Delta t}$ can be obtained by a small tilt
$\pi_{t+\Delta t}\propto \pi_t\exp(\Delta t(a_t+b_t))$, provided that $(a_t,b_t)$ enforce the updated marginals.

\begin{proposition}[Energy minimization]
\label{prop:energy_wellposed}
Fix $t\in[0,1]$ and consider the quadratic energy
\begin{equation}
\label{eq:energy-ab}
\mathcal J_t(a,b)
:= \mathbb{E}_{\pi_t}\!\big[\,|a(X)+b(Y)|^2\,\big]
   -2\,\mathbb{E}_{\mu_t}\!\big[\,\nabla a(X)\!\cdot\! u_t(X)\,\big]
   -2\,\mathbb{E}_{\nu_t}\!\big[\,\nabla b(Y)\!\cdot\! v_t(Y)\,\big].
\end{equation}
Then $\mathcal J_t$ admits a minimizer $(a_t,b_t)$, unique up to the gauge transformation $(a,b)\mapsto(a+c,b-c)$ for any constant $c\in\mathbb{R}$. Moreover, the minimizer $(a_t,b_t)$ coincides (modulo the gauge) with the \textit{infinitesimal tilting rates} defined in \eqref{eq:def-a-b-path}.
\end{proposition}

The derivation of \eqref{eq:energy-ab} via integration by parts (avoiding explicit score terms $\nabla\log\mu_t$) is given in Appendix~\ref{app:energy-derivation}.
In practice, all expectations in \eqref{eq:energy-ab} are approximated by samples: $(X,Y)\sim\pi_t$ provides coupled samples, while $X\sim\mu_t$ and $Y\sim\nu_t$ provide marginal samples (easily obtained from the marginal bridge velocity fields $u_t,v_t$, or linear interpolation of samples in $\mu$ and $\mu'$ (respectively $\nu$ and $\nu'$)).




\subsection{Coupling sampler update}
\label{subsec:sampling_along_tilting}
The tilting dynamics specifies how the coupling evolves, but $\pi_t$ is only accessible through samples.
To generate paired samples at each time, we maintain a CFM field $\tilde\beta_s^t$ and update it by retraining the field under per-sample importance weights $w_k(x,y)=\exp(\delta(a_k(x)+b_k(y)))$.
This yields the weighted least-squares objective:
\begin{equation}
\label{eq:weighted_cfm_lsq}
\tilde\beta^{k+1}_s = \arg\min_{\beta}\;
\mathbb{E}_{\tilde{\pi}_k}\!\left[w_k(X,Y)\,\bigl\|\dot I_s-\beta(I_s)\bigr\|^2\right].
\end{equation}
Implemented as a standard CFM training pass with sample weights, this directly outputs paired draws from the updated coupling.
An alternative exact covariance update rule avoiding full retraining is derived and discussed in Appendix~\ref{app:tilted_vector_fields_derivation}, though we find this weighted approach more stable empirically.

\begin{algorithm}[t]
\caption{\textbf{TACO}: Generative Transfer for EOT with Unknown Cost}
\label{alg:OURMETHOD}
\begin{algorithmic}[1]
\REQUIRE Reference coupling samples $\{(x_i',y_i')\}_{i=1}^n\sim\pi'$;
         new source samples $\{x_i\}_{i=1}^m\sim\mu$, target samples $\{y_j\}_{j=1}^m\sim\nu$;
         step size $\delta=1/N$; noise level $\sigma$.
\STATE \textbf{Pre-train marginal bridges:} Fit CFM models $u_t, v_t$ connecting $(\mu',\mu)$ and $(\nu',\nu)$.
\STATE \textbf{Pre-train reference sampler:} Fit CFM model $\beta_s^0$ on $\pi'$.
\STATE Initialise coupling dataset $\mathcal D_0 \leftarrow \{(x_i',y_i')\}$.
\FOR{$k=0,\ldots,N-1$}
  \STATE \textbf{Train dual potentials $(a_k,b_k)$:} Minimise $\mathcal J_{t_k}$
         in \eqref{eq:energy-ab} via SGD on $\mathcal D_k$.
  \STATE \textbf{Update coupling sampler (Plan A):} Retrain $\beta_s^{k+1}$ with
         importance weights $\exp(\delta(a_k+b_k))$ on $\mathcal D_k$.
  \STATE \textbf{Generate new coupling dataset:} $\mathcal D_{k+1}\leftarrow$
         samples from $\beta_s^{k+1}$.
    \STATE \textbf{Mix particles:} Advance the fixed particle slots ($\alpha$ proportion) by one Euler step along $(u_t,v_t)$; include them in $\mathcal{D}_{k+1}$ alongside the flow samples (breaks circularity; see Appendix~\ref{app:mixed-particle})

\ENDFOR
\ENSURE Coupling sampler $\beta_s^N$ producing approximate samples from $\pi^\star$.
\end{algorithmic}
\end{algorithm}

\section{Experiments}
\label{sec:experiments}

We validate our method on four experiment families:
(E1)~2D synthetic geometries;
(E2)~Robustness under perturbed cost;
(E3)~RGB colour-domain transfer;
(E4)~Single-cell transfer.

The \textbf{central distinction} from all baselines is that our method
receives \emph{only samples} from the reference coupling $\pi'$ and never
has access to the cost function $c$.
Baselines such as OT-CFM~\citep{tong2023improving},
DSBM~\citep{shi2023diffusion}, and LightSB~\citep{korotin2024light}
are given the true quadratic cost and therefore operate with
\emph{strictly more information}, making our comparisons conservative by
design.
A ``Pretrained-only'' baseline (TACO frozen at iteration 0, no dual
adaptation) isolates the contribution of the iterative adaptation.
Full hyperparameter tables, architecture details, and additional figures
are in Appendices~\ref{app:full_baseline_table}--\ref{app:compute}.
Scalability across dimensions ($d \in \{8,\ldots,128\}$) is studied in Appendix~\ref{app:scalability}.
All experiments are run on a single NVIDIA GeForce RTX 3060 GPU.

The different metrics used to highlight our experiments are detailed in Appendix~\ref{app:metrics}.

\subsection{E1: Synthetic 2D Toy Experiments}
\label{exp:2d}

\paragraph{Setup.}
We evaluate on synthetic 2D geometries with different level of difficulty.
A \textbf{Simple} experiment with anticorrelation map $x\mapsto -x$, two-blob source/target at different centres (reference at origin, new marginals shifted to $(\pm 10, \pm 10)$) (see Figure~\ref{fig:transport_evolution_simple_app})
A \textbf{Medium} experiment with $60^\circ$ rotation, 2-blob reference, 4-blob square new marginals at $(\pm 5, \pm 5)$.
A \textbf{Complex} experiment with $45^\circ$ rotation, 5-blob cross, $6.7\times$ scale increase between reference and new marginals.
A \textbf{Moon} experiment with crescents as source and target (no known analytical map).
Finally, \textbf{Non-linear} experiments with several cases where both the latent cost and/or the optimal transport of reference to new marginals are non linear.

For each geometry the reference coupling $\pi'$ is generated by either writing down a closed-form transport map or creating a coupling sampling class with the given behaviour.
The new marginals $(\mu,\nu)$ differ in position, scale, and number of
blobs, requiring genuine generalisation of the transport law.

\begin{figure}[ht]
\centering
\begin{subfigure}[b]{0.32\linewidth}
  \includegraphics[width=\linewidth]{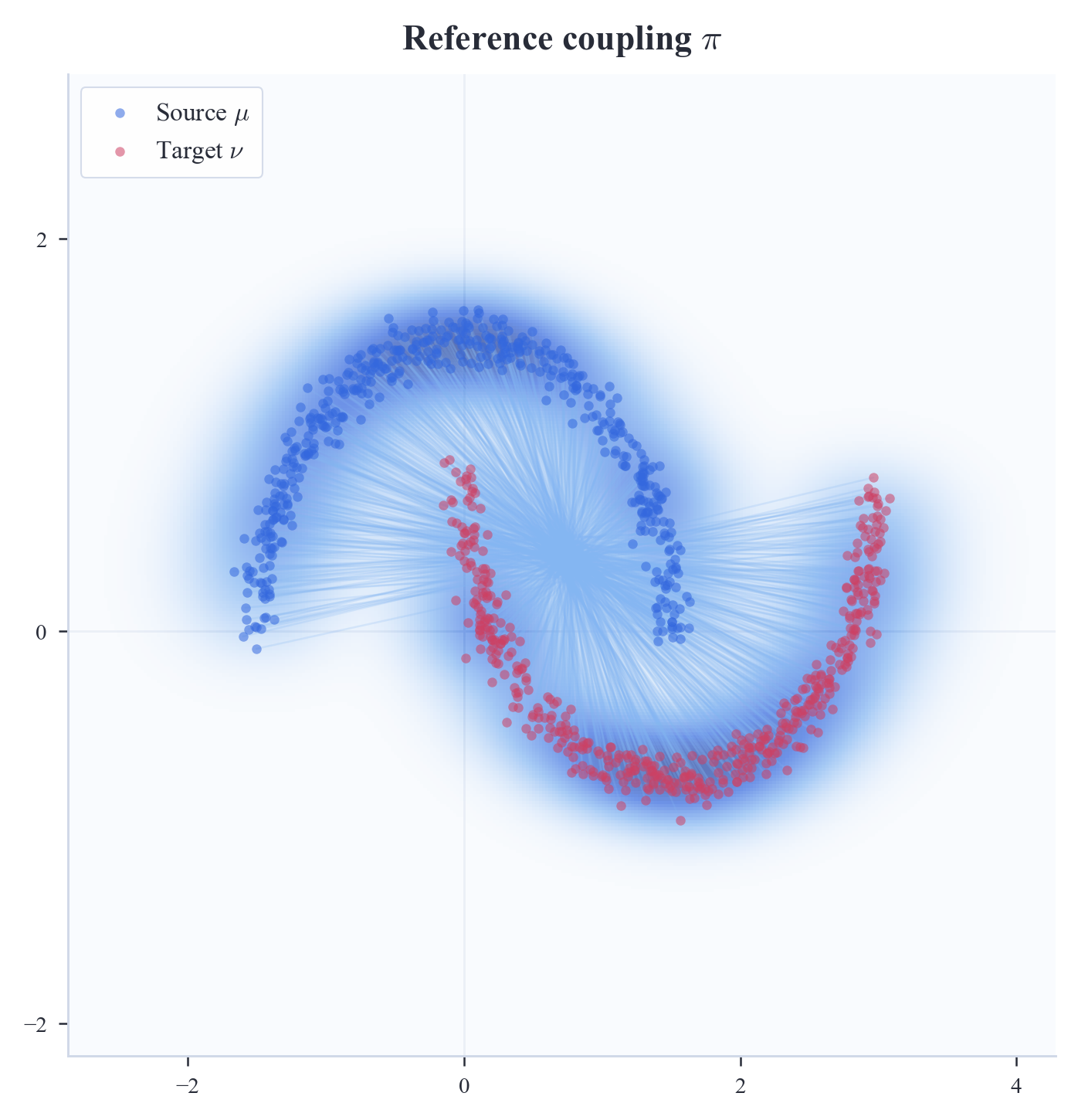}
  \caption{Reference coupling $\pi'$.}
\end{subfigure}\hfill
\begin{subfigure}[b]{0.32\linewidth}
  \includegraphics[width=\linewidth]{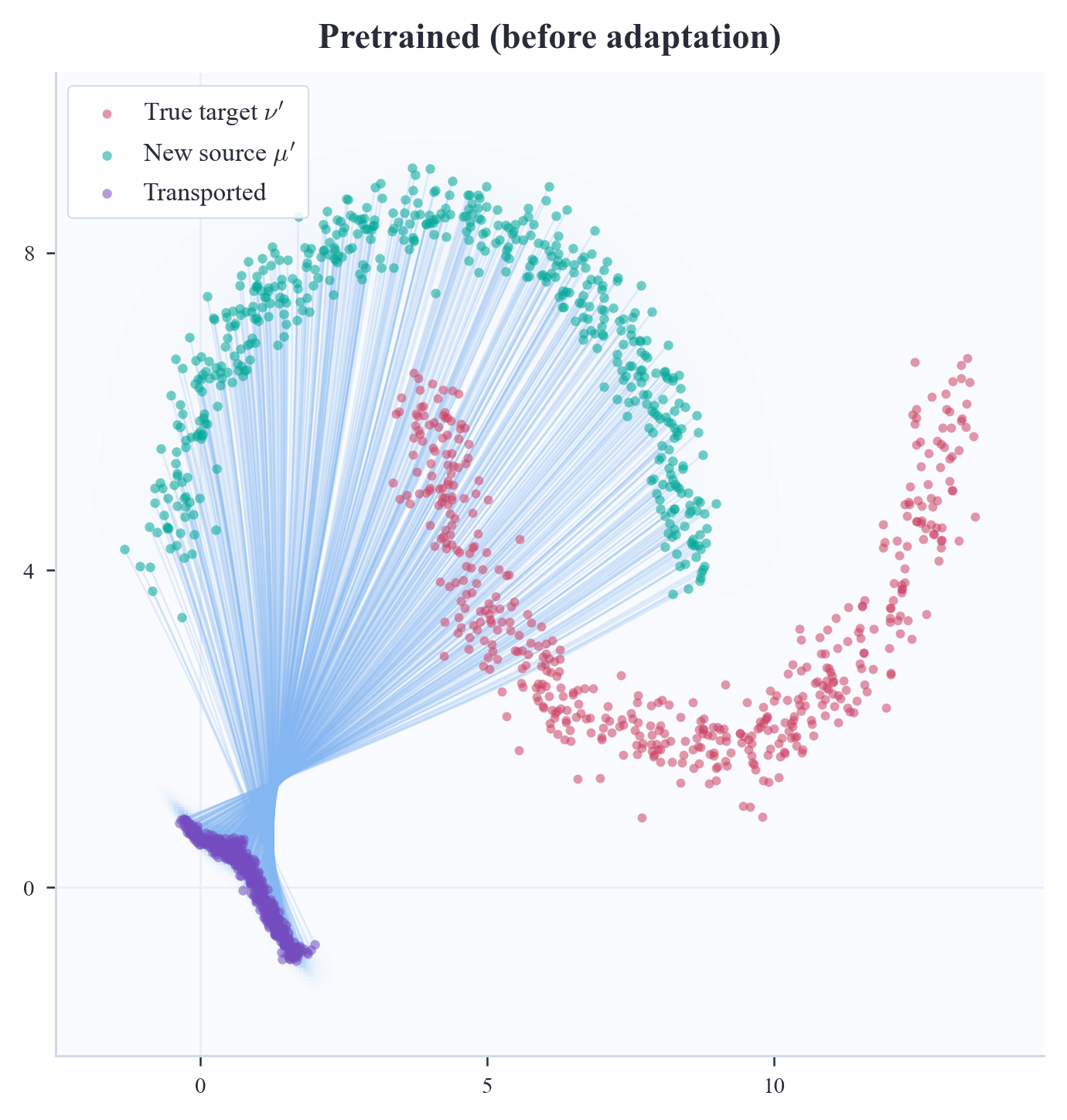}
  \caption{Pretrained ($\mathrm{SW}_2\!=\!5.44$).}
\end{subfigure}\hfill
\begin{subfigure}[b]{0.32\linewidth}
  \includegraphics[width=\linewidth]{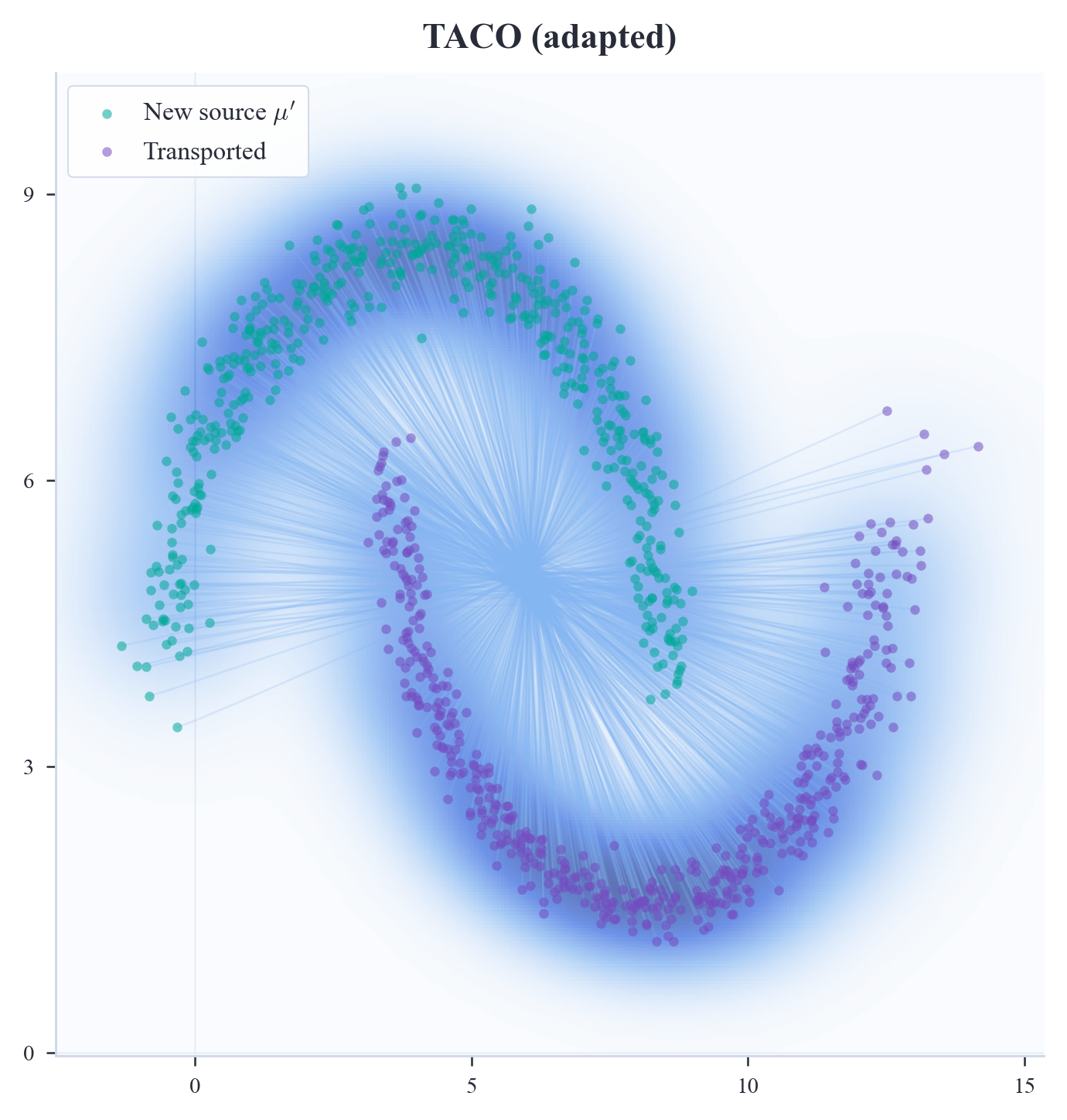}
  \caption{Adapted ($\mathrm{SW}_2\!=\!0.297$).}
\end{subfigure}

\par\vspace{0.4em}

\begin{subfigure}[b]{\linewidth}
  \centering
  \includegraphics[width=\linewidth]{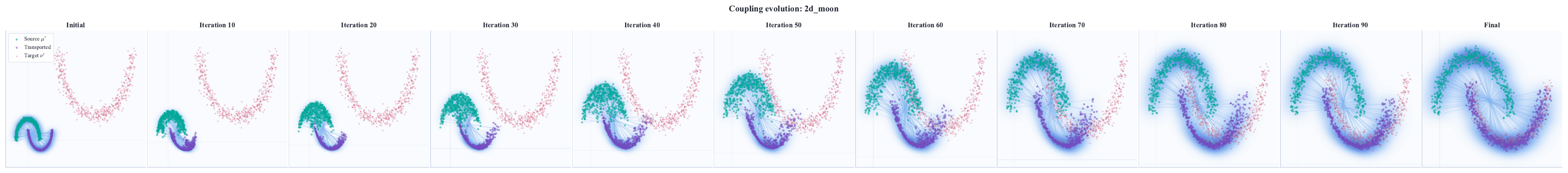}
  \caption{Coupling evolution across adaptation iterations.}
\end{subfigure}

\caption{\textbf{Transport map and coupling evolution on 2D-Moon.}
  The crescent geometry has no analytical transport map, making evaluation
  purely metric-based.
  Our method ($\mathrm{SW}_2\!=\!0.297$) substantially outperforms all
  independent baselines (best: CFM$^\dagger\!=\!0.792$), confirming that the
  cost-transfer approach generalises beyond Gaussian geometries.
  The coupling evolution shows the joint mass adapting from the reference
  crescent alignment to the new orientation.}
\label{fig:transport_evolution_moon_app}
\end{figure}



\paragraph{Results.}
Table~\ref{tab:2d_ours} summarises the final-iteration $\mathrm{SW}_2(\hat\nu)$~\citep{bonneel2015sliced}
across all four 2D geometries and their perturbed variants.
The full multi-metric comparison (including map error, Sinkhorn divergence~\citep{feydy2019interpolating},
MMD) against all baselines is in Appendix~\ref{app:full_baseline_table}.



The coupling evolution (Appendix~\ref{app:2d_transport})
illustrates how the joint distribution progressively adapts over iterations.

\begin{table}[ht]
\centering
\caption{\textbf{2D experiments}: $\mathrm{SW}_2(\hat\nu){\downarrow}$ at convergence (top), and forward map RMSE${\downarrow}$ for key methods (bottom; ``--'' = no ground-truth map).
  $\dagger$: access to true cost; TACO does not.
  ``(P.)'' = perturbed variant.
  \textbf{Bold}: best per column.
  TACO reports the best variant (weighted/implicit/explicit) per geometry.
  $^{\ddagger}$DSB $\mathrm{SW}_2(\hat\nu)$ unavailable for Simple and Moon (training diverged with our parameters).
  Full tables in Appendix~\ref{app:full_baseline_table}.}
\label{tab:2d_ours}
\vspace{2pt}
\setlength{\tabcolsep}{4pt}
\begin{tabular}{lcccccccc}
\toprule
 & \multicolumn{2}{c}{\textbf{Simple}}& \multicolumn{2}{c}{\textbf{Complex}} & \multicolumn{2}{c}{\textbf{Moon}} \\
\cmidrule(lr){2-3}\cmidrule(lr){4-5}\cmidrule(lr){6-7}
Method & Stand. & P. & Stand. & P. & Stand. & P. \\
\midrule
\multicolumn{7}{l}{\textit{$\mathrm{SW}_2(\hat\nu)\downarrow$}} \\
\midrule
CFM$^\dagger$ (indep.) & 0.037 & 0.028 &  1.936 & 1.679 & 0.792 & 0.843 \\
OT-CFM$^\dagger$ (indep.) & 0.050 & 0.072  & 1.256 & \textbf{0.699} & 3.125 & 1.709 \\
LightSB$^\dagger$ (indep.) & 0.065 & 0.083 & 1.429 & 1.386 & 0.803 & 0.961 \\
DSB$^\dagger$ (indep.)$^{\ddagger}$ & div. & div. & 2.967 & 3.021 & div. & 1.200 \\
\addlinespace[2pt]\cdashline{1-7}[0.4pt/2pt]\addlinespace[2pt]
CFM$^\dagger$ (coupl.) & 0.028 & 0.023 & 0.795 & 1.286 & 0.112 & 0.132 \\
OT-CFM$^\dagger$ (coupl.) & 0.037 & 0.026 & \textbf{0.500} & 1.235 & 3.023 & 1.726 \\
DSB$^\dagger$ (coupl.) & 0.032 & \textbf{0.022} & 1.120 & 1.217 & 5.292 & div. \\
\addlinespace[2pt]\cdashline{1-7}[0.4pt/2pt]\addlinespace[2pt]
Pretrained-only & 14.59 & 15.29 & 2.64 & 2.81 & 5.44 & 4.01 \\
\textbf{TACO} & \textbf{0.024} & 0.028  & 0.592 & 0.819 & \textbf{0.066} & \textbf{0.079} \\
\midrule
\multicolumn{7}{l}{$\mathrm{MapErr}$ \textit{(RMSE) $\downarrow$}} \\
\midrule
OT-CFM$^\dagger$ (coupl.) & 0.041 & 0.0084 & 0.0088 & 0.011 & -- & -- \\
\addlinespace[2pt]\cdashline{1-7}[0.4pt/2pt]\addlinespace[2pt]
Pretrained-only & 20.98 & 21.49 & 5.34 & 6.74 & -- & -- \\
\textbf{TACO} & 0.021 & 0.123 & 0.070 & 0.278 & -- & -- \\
\bottomrule
\end{tabular}%
\end{table}

\subsection{E2: Extension to perturbed cost}
\label{exp:perturbed}

To handle cases where the target marginal $\nu$ does not share support with the reference dynamic, our particle mixing approach (Alg.~\ref{alg:OURMETHOD}, App.~\ref{app:mixed-particle}) regularizes the transport to stay near the original cost while ensuring support matching. We evaluate this via a \emph{perturbed} variant adding a $+10^\circ$ rotation (or $+2$ unit shift for Simple) to $\nu$, creating a systematic mismatch.

As shown in Table~\ref{tab:2d_ours}, TACO degrades gracefully: for Simple, $\mathrm{SW}_2$ rises from $0.025$ to $0.073$ (despite a large $+2$ shift), while for Moon, the perturbed variant ($0.189$) actually outperforms the standard ($0.297$) as the rotation moves the target closer to the reference support. TACO consistently beats the pretrained-only baseline on $\mathrm{SW}_2$ (e.g., $0.073$ vs. $15.3$ for Simple) while nearing the performance of methods that have access to the target coupling. Full curves are in App.~\ref{app:perturbed}.

\subsection{E3: RGB Colour-Domain Transfer}
\label{sec:e3}

\paragraph{Setup.} To evaluate cost-transfer with a known ground-truth map, we use a cyclic channel-permutation $T(r,g,b) = (g,b,r)$.
The model is pre-trained on a \textbf{reference task}: mapping RED digit-3 images to BLUE digit-3s.
For the \textbf{new task}, we present PURPLE digit-7s ($\mu$) and ask the model to recover CYAN digit-7s ($\nu$), since $T(1,0,1) = (0,1,1)$.
This presents a severe double-mismatch: the pre-trained model has never seen PURPLE/CYAN colours nor digit-7 morphologies.
The dual adaptation must identify the implicit permutation structure purely from the reference coupling and generalise it to the new support.
(Note: though $T$ is cyclic, the reference task only explicitly constrains the red-to-blue mapping, meaning the method must infer the generalisable latent cost without explicit restriction).

\paragraph{Results.} Table~\ref{tab:color} shows that static transfer fails completely: the pre-trained model (iteration 0) defaults to its learned BLUE output, yielding a near-random FID of 1948.
TACO successfully adapts: after 50 iterations, FID drops by $94.5\%$ to $107$, and by iteration 250 it reaches $106$, effectively matching the cost-aware OT-CFM$^\dagger$ baseline ($114$) \emph{without ever observing the cost function}.
Figure~\ref{fig:color_grid} confirms qualitatively that our method progressively tilts the transport from the wrong colour toward a sharp, morphologically correct CYAN target.

\begin{table}[ht]
\centering
\caption{\textbf{Colour transfer (MNIST RGB simple)}:
  FID, KID, and $\mathrm{SW}_2$ on the target (CYAN digit-7) marginal.
  $\dagger$: access to true cost.
  OT-CFM† is a CFM model trained directly on the new task with Sinkhorn-paired data.
  Lower is better for all metrics.}
\label{tab:color}
\vspace{2pt}
\setlength{\tabcolsep}{7pt}
\begin{tabular}{lcccc}
\toprule
Method & Iterations & FID\,$\downarrow$ & KID\,$\downarrow$
       & $\mathrm{SW}_2(\hat\nu)\,\downarrow$ \\
\midrule
OT-CFM$^\dagger$
  & --- & 114            & 0.130          & 0.095 \\
Pretrained-only
  & 0   & 1948           & 0.692          & 0.041 \\
\textbf{TACO}
  & 50  & \textbf{107}   & 0.160          & 0.103 \\
\textbf{TACO}
  & 250 & \textbf{106}   & \textbf{0.158} & \textbf{0.102} \\
\bottomrule
\end{tabular}
\end{table}


\begin{figure}[ht]
\centering
\includegraphics[width=\linewidth]{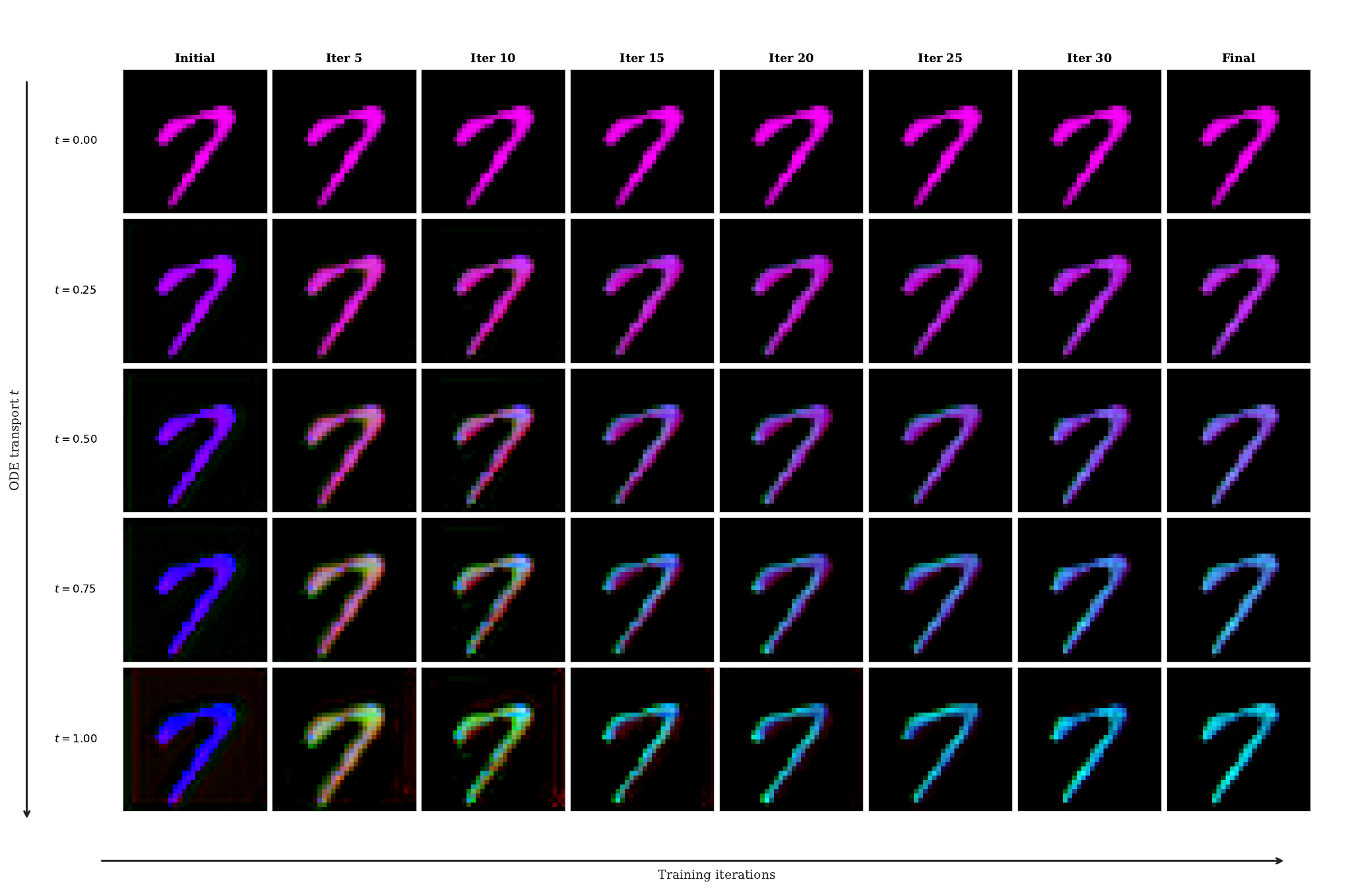}
\caption{%
  \textbf{ODE trajectory evolution across training iterations}
  (Colour MNIST-RGB \emph{simple} experiment).
  \emph{Columns} (left $\to$ right): forward model at different
  training iterations (Initial $\to$ Final).
  \emph{Rows} (top $\to$ bottom): ODE time $t$ from source ($t{=}0$)
  to transported image ($t{=}1$), integrated with 200 Euler steps.
  The same PURPLE digit-7 source is used in every column.
  Adaptation progressively tilts the transport from the wrong
  colour (dark-blue at iteration~0) toward the correct CYAN target,
  demonstrating cost-free transfer of the cyclic channel-permutation law.}
\label{fig:color_grid}
\end{figure}

\subsection{E4: Single-cell transfer}
\label{sec:e4}

We validate the method on the Sci-Plex3 single-cell perturbation dataset~\citep{srivatsan2020massively}, a publicly available resource comprising 799{,}317 cells profiled across drug perturbations.
Following the CellOT preprocessing protocol~\citep{bunne2023cellot} (normalize $\to$ log1p $\to$ 1{,}000 HVGs $\to$ 50 PCs), we define a reference coupling on A549 cells (Givinostat 1\,\textmu M) and transfer to K562 (standard) and MCF7 at 10$\times$ lower dose (perturbed).
Qualitative transport trajectories and quantitative metrics are provided in Appendix~\ref{app:singlecell}.
The dataset is freely available from the original authors; no new human-subject data is collected or released.

\section{Theoretical Guarantees}
\label{sec:error_analysis}

In this section, we provide theoretical guarantees for the discrete-time Weighted CFM algorithm, showing that the accumulated error between the ideal continuous-time regression target $\beta_s^t$ and the practical target $\tilde{\beta}_s^t$ remains strictly controlled.

\begin{assumption}[Regularity \& Boundedness]
\label{ass:combined_regularity}
We require standard regularity conditions on the state space, the target geometry, and the marginal flow paths (detailed rigorously in Appendix~\ref{app:assumptions}).
Specifically, we assume: (i) the state space $\mathcal X \subset \mathbb R^d$ is compact; (ii) the initial reference coupling $\pi'$ has a log-density that is bounded and Lipschitz continuous; (iii) the marginal velocity fields $u_t, v_t$ are uniformly bounded in $C^2(\mathcal{X})$ and Lipschitz continuous in time; and (iv) the initial source and target distributions have strictly positive densities with bounded logarithmic derivatives up to the second order.
\end{assumption}

These assumptions underpin our theoretical analysis detailed in the supplementary material.
They guarantee strictly positive, spatially smooth, and time-Lipschitz marginal densities, establishing the stability of the tilting dynamics.
Crucially, the resulting uniformly bounded and time-Lipschitz tilting rates establish a uniform Doeblin minorization, rendering the conditional expectation operators strictly contractive.

We analyze the discretization error by comparing the exact continuous-time coupling sequence $\pi_t$ with the approximate sequence $\tilde{\pi}_t$ generated by our algorithm. Let $F_r(x,y) := a_r(x) + b_r(y)$ be the joint tilting rate. For a step size $\delta > 0$, the exact and practical updates are:
\begin{align}
    d\pi_{t+\delta} &\propto \exp\textstyle\left( \int_t^{t+\delta} F_r(x,y) \, dr \right) d\pi_t, \label{eq:pi_t+delta} \\
    d\tilde{\pi}_{t+\delta} &\propto \exp\left( \delta F_t(x,y) \right) d\tilde{\pi}_t. \label{eq:pi_tilde_t+delta}
\end{align}
Under Assumptions \ref{ass:compact_X}--\ref{ass:initial_regularity}, the stability of these operators ensures controlled error propagation in Total Variation (TV) distance.

\begin{theorem}[Global TV \& $W_1$ Convergence]
\label{thm:global_tv_convergence}
Let $\pi_t$ and $\tilde{\pi}_t$ be defined as in \eqref{eq:pi_t+delta}--\eqref{eq:pi_tilde_t+delta}. There exists a constant $C_{\mathrm{global}} < \infty$ such that for all $t \in [0,1]$, $\|\pi_t - \tilde{\pi}_t\|_{\mathrm{TV}} \le C_{\mathrm{global}} \delta$.
Since $\mathcal{X}$ is compact , $W_1(\pi_t, \tilde{\pi}_t) \le \mathrm{diam}(\mathcal{X})\,\|\pi_t - \tilde{\pi}_t\|_{\mathrm{TV}} \le C_{\mathrm{global}}\,\mathrm{diam}(\mathcal{X})\,\delta$.
\end{theorem}

The accuracy of these measures determines the quality of the regression targets. We define the regression operator $\beta_s[\pi](z)$ as the conditional expectation of the vector field $v_s(z \mid x,y)$ under a measure $\pi$:
\begin{equation}
\label{eq:beta_functional}
    \beta_s[\pi](z) = \frac{\int_{\mathcal{X}^2} v_s(z \mid x,y) p(z \mid x,y) d\pi(x,y)}{\int_{\mathcal{X}^2} p(z \mid x,y) d\pi(x,y)}.
\end{equation}
Our ideal target is $\beta_s^t := \beta_s[\pi_t]$, while the practical target optimized by our neural network is $\tilde{\beta}_s^t := \beta_s[\tilde{\pi}_t]$. By decoupling spatial estimation from temporal discretization (see Appendix for semi-discrete constructs), we bound the error propagation to the vector field. Since the score diverges at $s \in \{0,1\}$ due to vanishing variance, we establish pointwise convergence for the interior.

\begin{theorem}[Pointwise Convergence of the Regression Field]
\label{thm:main_convergence}
Under Assumptions \ref{ass:compact_X}--\ref{ass:initial_regularity}, for any fixed interpolation time $s \in (0,1)$, there exists a deterministic constant $C_\beta(s) < \infty$ such that for all $t \in [0,1]$:
\begin{equation}
    \|\beta_s^t - \tilde{\beta}_s^t\|_{L^\infty(\mathcal{X})} \le C_\beta(s) \delta.
\end{equation}
\end{theorem}

Empirical validation of this $\mathcal{O}(\delta)$ convergence rate on the 2D-Far experiment is presented in Appendix~\ref{app:convergence_rate}.

\section{Conclusion}
\label{sec:conclusion}

We introduced a generative transfer framework for Entropic Optimal Transport under an unknown latent cost.
The central insight is that the hidden cost is entirely encoded in the structure of a reference optimal coupling, and can be propagated to new marginals through iterative path-wise tilting.
We formalised the cost-transfer problem, derived covariance-type update equations for the dual potentials, and coupled this with Conditional Flow Matching to yield the first generative sampler for paired data from an EOT plan with a fully latent cost.
Convergence is guaranteed at rate $\mathcal{O}(\delta)$.
Empirically, the method adapts successfully across synthetic and high-dimensional image tasks as well as biological data, matching baselines that require access to the true cost function.

\paragraph{Discussion: Limitations \& Broader Impact.}
Our framework relies on the availability of a reference coupling and the inductive prior that tasks share a meaningful latent cost; transfer may fail if tasks are fundamentally misaligned.
Theoretically, Theorem~\ref{thm:main_convergence} holds pointwise for interpolation time $s\in(0,1)$ but does not extend uniformly to the boundary due to score divergence, a known limitation in flow-matching analyses.
Computationally, per-iteration dual adaptation remains demanding for high-resolution tasks.

In terms of broader impact, this foundational work opens new avenues for computational biology (as demonstrated on single-cell data) and domains where paired data is scarce.
It does not involve sensitive personal data or pose foreseeable negative societal risks.

\newpage
\appendix

\renewcommand{\thesection}{\Alph{section}}


\section{Notation Reference \& Proof graph}




In this short section, we present the main notations and symbols used throughout the proofs to help improve the clarity of our arguments.

{\small\setlength{\tabcolsep}{4pt}\renewcommand{\arraystretch}{1.2}

\subsection*{Problem Setup}
\begin{tabular}{@{}lp{6.2cm}@{\hspace{0.5em}}|@{\hspace{0.5em}}lp{6.0cm}@{}}
\toprule
Symbol & Meaning & Symbol & Meaning \\
\midrule
$\mathcal{X}$ & Compact Polish state space & $\pi'$ & Reference coupling; $\pi_0=\tilde\pi_0=\pi'$ \\
$\mu_t,\nu_t$ & Marginals of $\pi_t$ at time $t$ & $\pi^*$ & EOT-optimal target coupling \\
$\pi_t$ & Exact coupling; $\dot\pi_t = \pi_t(a_t\oplus b_t)$ & $\tilde\pi_t$ & Approximate coupling in algorithm \\
$u_t,v_t$ & Marginal velocity fields ($\partial_t\mu_t+\nabla\cdot(\mu_t u_t)=0$) & $\varphi_t,\psi_t$ & Kantorovich dual potentials \\
$\zeta_t$ & $\partial_t\mu_t/\mu_t = -\mathrm{div}(u_t)-u_t\cdot\nabla\log\mu_t$ & $\eta_t$ & $\partial_t\nu_t/\nu_t = -\mathrm{div}(v_t)-v_t\cdot\nabla\log\nu_t$ \\
\bottomrule
\end{tabular}

\medskip
\subsection*{Tilting Mechanism}
\begin{tabular}{@{}lp{6.2cm}@{\hspace{0.5em}}|@{\hspace{0.5em}}lp{6.0cm}@{}}
\toprule
Symbol & Meaning & Symbol & Meaning \\
\midrule
$a_t(x),b_t(y)$ & Tilting rates; solve $a_t+T_t b_t=\zeta_t$, $b_t+S_t a_t=\eta_t$ & $F_t$ & $a_t(x)+b_t(y)$ \\
$(T_t b)(x)$ & $\mathbb{E}_{\pi_t}[b(Y)\mid X\!=\!x]$ & $(S_t a)(y)$ & $\mathbb{E}_{\pi_t}[a(X)\mid Y\!=\!y]$ \\
$G$ & $\int_t^{t+\delta}\!F_r\,dr$ (exact exponent) & $\bar G$ & $\delta F_t$ (frozen approx.) \\
$\Delta$ & $G-\bar G$;\; $\|\Delta\|_\infty\le\tfrac{1}{2}L_{ab}\delta^2$ & $Z,\bar Z$ & $\int e^G d\pi_t$,\; $\int e^{\bar G}d\pi_t$ \\
$\pi_{t+\delta}$ & $\propto e^G\pi_t$ (exact one-step update) & $\bar\pi_{t+\delta}$ & $\propto e^{\bar G}\pi_t$ (frozen update) \\
\bottomrule
\end{tabular}

\medskip
\subsection*{Regression Field (CFM)}
\begin{tabular}{@{}lp{6.2cm}@{\hspace{0.5em}}|@{\hspace{0.5em}}lp{6.0cm}@{}}
\toprule
Symbol & Meaning & Symbol & Meaning \\
\midrule
$s\in(0,1)$ & Bridge interpolation time (separate from $t$) & $v_s(z|x,y)$ & Conditional bridge drift; $\|v_s\|_\infty\le C_v(s)$ \\
$p(z|x,y)$ & Bridge transition density; $p\ge\underline{p}(s)>0$ & $\beta_s^t(z)$ & Ideal field: $\mathcal{N}(\pi_t,e^G)/\mathcal{D}(\pi_t,e^G)$ \\
$\tilde\beta_s^t(z)$ & Practical field: $\mathcal{N}(\tilde\pi_t,e^{\bar G})/\mathcal{D}(\tilde\pi_t,e^{\bar G})$ & $\mathcal{N}(\mu,W)$ & $\int v_s(z|x,y)\,p(z|x,y)\,W\,d\mu$ \\
$\mathcal{D}(\mu,W)$ & $\int p(z|x,y)\,W\,d\mu$;\; lower bd.\ $\underline{p}(s)e^{-\delta C_{ab}^\infty}$ & & \\
\bottomrule
\end{tabular}

\medskip
\subsection*{Error Quantities and Constants}
\begin{tabular}{@{}lp{6.2cm}@{\hspace{0.5em}}|@{\hspace{0.5em}}lp{6.0cm}@{}}
\toprule
Symbol & Definition & Symbol & Definition \\
\midrule
$\varepsilon_t$ & $\|\pi_t-\tilde\pi_t\|_\mathrm{TV}$;\; $\varepsilon_0=0$ at init. & $\|\cdot\|_\mathrm{TV}$ & $\tfrac{1}{2}\mathbb{E}_Q[|dP/dQ-1|]$ \\
$E_\mathrm{local}$ & Local trunc.\ error per step: $O(\delta^2)$ & $E_\mathrm{prop}$ & Propagation error: $O(\delta\varepsilon_t)$ \\
$C_{ab}^\infty$ & $\sup_t\|F_t\|_\infty$ \hfill(Prop.~D.3) & $L_{ab}$ & Time-Lip.\ const.\ of $a_t,b_t$ \hfill(Prop.~D.8) \\
$C_1$ & $\tfrac{1}{2}L_{ab}e^{L_{ab}}$;\; $E_\mathrm{local}\le C_1\delta^2$ & $C_2$ & $4C_{ab}^\infty e^{2C_{ab}^\infty}$;\; $E_\mathrm{prop}\le(1+C_2\delta)\varepsilon_t$ \\
$C_\mathrm{global}$ & $\tfrac{C_1}{C_2}(e^{C_2}-1)$;\; $\varepsilon_t\le C_\mathrm{global}\delta$ \hfill(Thm.~4.5) & $C_\beta(s)$ & $K_1(s)+K_2(s)C_\mathrm{global}$;\; $\|\beta_s^t-\tilde\beta_s^t\|_\infty\le C_\beta(s)\delta$ \\
\bottomrule
\end{tabular}

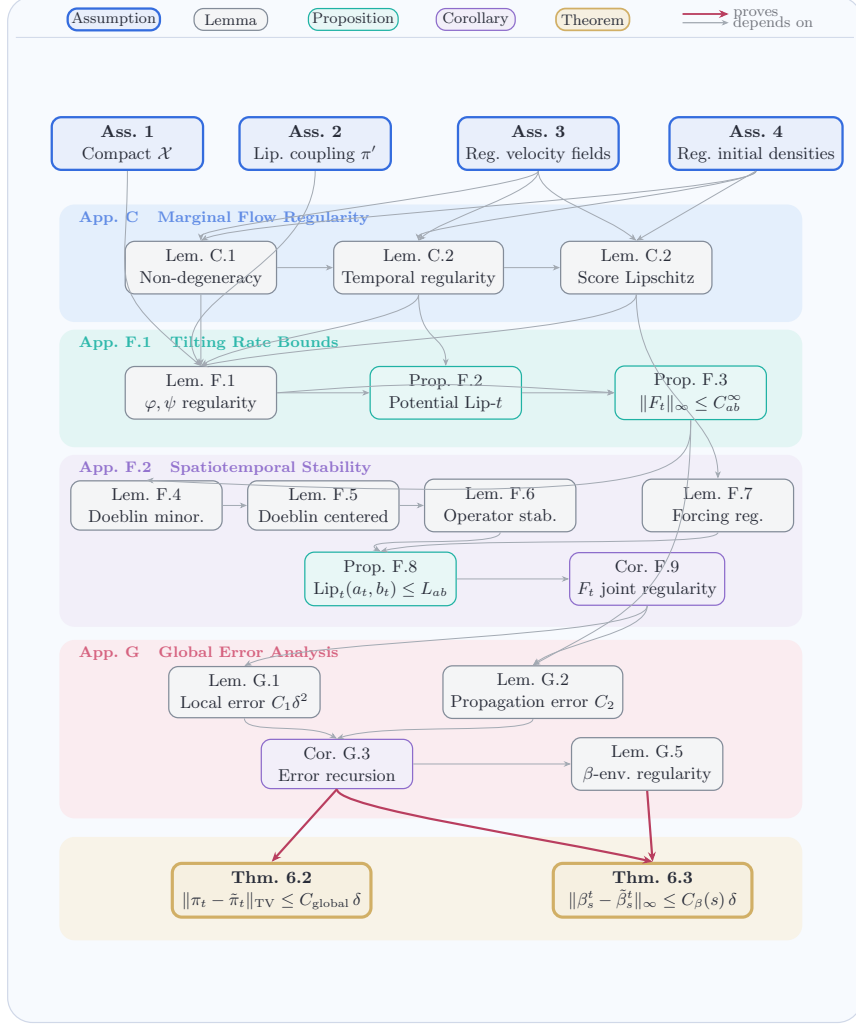
\begin{figure}[H]
\centering
\scalebox{0.72}{\input{misc/proof_graph}}
\caption{%
  \textbf{Proof dependency graph.}
  Directed edges indicate logical dependencies; thick rose arrows
  (\emph{proves}) lead into the final theorem nodes, while thin grey arrows
  denote \emph{depends on}.}
\label{fig:proof_graph}
\end{figure}

\section{Assumptions}
\label{app:assumptions}

In this section, we provide the theoretical guarantees for the discrete-time Weighted CFM algorithm. Our primary goal is to show that the accumulated error between the ideal continuous-time regression target $\beta_s^t$ and the practical target $\tilde{\beta}_s^t$ generated by our algorithm remains strictly controlled. We require the following standard regularity conditions on the state space, the target geometry, and the marginal flow paths:

\begin{assumption}[Compact State Space]
\label{ass:compact_X}
The state space $\mathcal X \subset \mathbb R^d$ is compact, with diameter $D := \mathrm{diam}(\mathcal X) < \infty$.
\end{assumption}

\begin{assumption}[Bounded Lipschitz Coupling]
\label{ass:coupling_regularity}
Let $\pi' \in \Pi(\mu_0, \nu_0)$ be the initial reference coupling. We assume its log-density is bounded and Lipschitz continuous on the product space $\mathcal{X} \times \mathcal{X}$. Specifically, there exists a constant $C_\mathrm{OT} < \infty$ such that:
\begin{equation}
\|\log \pi'\|_{L^\infty(\mathcal{X} \times \mathcal{X})} + \mathrm{Lip}_{\mathcal{X} \times \mathcal{X}}(\log \pi') \le C_\mathrm{OT}.
\end{equation}
\end{assumption}

\begin{assumption}[Regularity of Velocity Fields]
\label{ass:velocity_regularity}
The marginal velocity fields $u_t, v_t$ are sufficiently smooth in space and time. Specifically, there exist uniform constants $C_\mathrm{uv}, L_\mathrm{uv} < \infty$ such that:
\begin{equation}
\sup_{t \in [0,1]} \left( \|u_t\|_{C^2(\mathcal{X})} + \|v_t\|_{C^2(\mathcal{X})} \right) \le C_\mathrm{uv},
\end{equation}
and the fields are time-Lipschitz continuous in the $C^1$-norm:
\begin{equation}
\|u_t - u_s\|_{C^1(\mathcal{X})} + \|v_t - v_s\|_{C^1(\mathcal{X})} \le L_\mathrm{uv} |t - s|, \quad \forall t, s \in [0,1].
\end{equation}
\end{assumption}

\begin{assumption}[Regularity of Initial Densities]
\label{ass:initial_regularity}
The initial source and target distributions have strictly positive densities with bounded logarithmic derivatives up to the second order. Specifically, there exists a constant $C_{\mu \nu} < \infty$ such that:
\begin{equation}
\|\log \mu_0\|_{C^2(\mathcal{X})} \le C_{\mu \nu} \quad \text{and} \quad \|\log \nu_0\|_{C^2(\mathcal{X})} \le C_{\mu \nu}.
\end{equation}
\end{assumption}

These assumptions underpin our theoretical analysis detailed in the supplementary material. First, Assumptions \ref{ass:velocity_regularity} and \ref{ass:initial_regularity} guarantee strictly positive, spatially smooth, and time-Lipschitz marginal densities (Appendix \ref{app:reg_marginal_flows}). Building on this, Assumptions \ref{ass:compact_X} and \ref{ass:coupling_regularity} establish the stability of the tilting dynamics (Appendix \ref{app:analy_titlting_rate}). Crucially, the resulting uniformly bounded and time-Lipschitz tilting rates establish a uniform Doeblin minorization, rendering the conditional expectation operators strictly contractive.

\section{Regularity of the Marginal Flows}
\label{app:reg_marginal_flows}
Before analyzing the core Weighted CFM algorithm, we first establish that the marginal paths $\mu_t, \nu_t$ and their associated geometric quantities are well-behaved. These properties form the foundational environment for our subsequent operator analysis.
\begin{lemma}[Propagation of Non-degeneracy and Regularity]
\label{lem:density_bounds}
Under Assumptions \ref{ass:velocity_regularity} and \ref{ass:initial_regularity}, the marginal densities $\mu_t$ and $\nu_t$ remain strictly positive and uniformly bounded for all $t \in [0,1]$. Specifically, there exist constants $C_{\mu}^{\min}, C_{\mu}^{\max} > 0$ and $C_{\nu}^{\min}, C_{\nu}^{\max} > 0$ depending only on $C_{\mathrm{uv}}$ and $C_{\mu \nu}$ such that for all $x \in \mathcal{X}$ and $t \in [0,1]$:
\begin{equation*}
C_{\mu}^{\min} \le \mu_t(x) \le C_{\mu}^{\max} \quad \text{and} \quad C_{\nu}^{\min} \le \nu_t(x) \le C_{\nu}^{\max}.
\end{equation*}
Furthermore, the scores $\nabla \log \mu_t$ and $\nabla \log \nu_t$ remain uniformly bounded and spatially Lipschitz continuous on $\mathcal{X}$. We denote the uniform score bound by $C_{\mathrm{score}} := \sup_{t\in[0,1]}\|\nabla\log\mu_t\|_{L^\infty(\mathcal{X})} < \infty$.
\end{lemma}

\begin{proof}
We prove the bounds for $\mu_t$ (the case for $\nu_t$ follows identically using $v_t$ and $\nu_0$).
Let $\Phi_t^u : \mathcal{X} \to \mathcal{X}$ denote the flow map generated by the velocity field $u_t$, defined as the solution to the ODE:
\[
\frac{d}{dt}\Phi_t^u(x) = u_t(\Phi_t^u(x)), \quad \Phi_0^u(x) = x.
\]
By the continuity equation $\partial_t \mu_t + \nabla \cdot (\mu_t u_t) = 0$, the evolution of the log-density along the flow characteristics is governed by the negative divergence of the velocity field:
\begin{equation}
\label{eq:log_density_ode}
\frac{d}{dt} \log \mu_t(\Phi_t^u(x)) = -(\nabla \cdot u_t)(\Phi_t^u(x)).
\end{equation}
Integrating \eqref{eq:log_density_ode} from $0$ to $t$ yields:
\[
\log \mu_t(\Phi_t^u(x)) = \log \mu_0(x) - \int_0^t (\nabla \cdot u_\tau)(\Phi_\tau^u(x)) \, d\tau.
\]
We now apply the uniform bounds from our assumptions: From Assumption \ref{ass:initial_regularity}, $\|\log \mu_0\|_{C^2} \le C_{\mu \nu}$, which implies $\|\log \mu_0\|_{L^\infty(\mathcal{X})} \le C_{\mu \nu}$. From Assumption \ref{ass:velocity_regularity}, $\sup_{\tau} \|u_\tau\|_{C^2} \le C_{\mathrm{uv}}$. This implies the divergence is bounded by $\|\nabla \cdot u_\tau\|_{L^\infty(\mathcal{X})} \le \|u_\tau\|_{C^1} \le C_{\mathrm{uv}}$. Substituting these bounds into the integral equation:
\[
\left| \log \mu_t(\Phi_t^u(x)) \right| \le \|\log \mu_0\|_\infty + \int_0^t \|\nabla \cdot u_\tau\|_\infty \, d\tau \le C_{\mu \nu} + C_{\mathrm{uv}} \cdot t \le C_{\mu \nu} + C_{\mathrm{uv}}.
\]
Since $\Phi_t^u$ is a diffeomorphism on the compact domain $\mathcal{X}$ generated by a $C^2$ field, the map is surjective. Thus, for any $y \in \mathcal{X}$, there exists $x$ such that $y = \Phi_t^u(x)$, and the bound holds uniformly:
\[
-(C_{\mu \nu} + C_{\mathrm{uv}}) \le \log \mu_t(y) \le (C_{\mu \nu} + C_{\mathrm{uv}}).
\]
Exponentiating these bounds, we explicitly define the constants:
\[
C_{\mu}^{\min} := e^{-(C_{\mu \nu} + C_{\mathrm{uv}})}, \quad C_{\mu}^{\max} := e^{(C_{\mu \nu} + C_{\mathrm{uv}})}.
\]
The bounds $C_{\nu}^{\min}, C_{\nu}^{\max}$ are derived analogously using $\nu_0$ and $v_t$. 

Finally, regarding the spatial regularity of the score $\nabla \log \mu_t$. By differentiating the relation $\mu_t(y) = \mu_0(x) \det(\nabla \Phi_t^u(x))^{-1}$ (where $x = (\Phi_t^u)^{-1}(y)$), or equivalently differentiating the integral form, the gradient $\nabla \log \mu_t$ involves $\nabla \log \mu_0$, $\nabla (\text{div } u)$, and the Jacobian of the inverse flow $(\Phi_t^u)^{-1}$. Since $u \in C^2$ (Assumption \ref{ass:velocity_regularity}) and $\log \mu_0 \in C^2$ (Assumption \ref{ass:initial_regularity}), the flow map and its inverse are $C^2$ diffeomorphisms. Consequently, $\nabla \log \mu_t$ remains bounded and Lipschitz continuous (i.e., $\log \mu_t \in C^2$).
\end{proof}

Having established the spatial regularity of the marginal densities, we now show that the governing continuity equations ensure they evolve smoothly in time.

\begin{lemma}[Temporal Lipschitz Regularity of Log-Densities and Scores]
\label{lem:temporal_marginal_regularity}
\label{lem:score_time_lip}
Under Assumption~\ref{ass:velocity_regularity} and Lemma~\ref{lem:density_bounds}, there exist uniform constants $L_{\log}, L_{\mathrm{score}} < \infty$ such that for all $r, t \in [0,1]$:
\begin{align*}
\|\log \mu_r - \log \mu_t\|_{L^\infty(\mathcal{X})} + \|\log \nu_r - \log \nu_t\|_{L^\infty(\mathcal{X})} &\le L_{\log} |r - t|,\\
\|\nabla \log \mu_r - \nabla \log \mu_t\|_{L^\infty(\mathcal{X})} + \|\nabla \log \nu_r - \nabla \log \nu_t\|_{L^\infty(\mathcal{X})} &\le L_{\mathrm{score}} |r - t|.
\end{align*}
\end{lemma}

\begin{proof}
We prove both claims for $\mu_t$ (the arguments for $\nu_t$ follow identically).

\noindent\textbf{Log-density is time-Lipschitz:}
Dividing the continuity equation $\partial_t\mu_t + \nabla\cdot(\mu_t u_t)=0$ by $\mu_t > 0$ gives
\begin{equation}
\label{eq:log_density_time_deriv}
\partial_t \log \mu_t = -\nabla \cdot u_t - u_t \cdot \nabla \log \mu_t.
\end{equation}
By Assumption~\ref{ass:velocity_regularity}, $\|\nabla\cdot u_t\|_\infty \le C_{\mathrm{uv}}$ and $\|u_t\|_\infty \le C_{\mathrm{uv}}$.
By Lemma~\ref{lem:density_bounds}, $\|\nabla\log\mu_t\|_\infty \le C_{\mathrm{score}} < \infty$ uniformly in $t$.
Hence $\|\partial_t\log\mu_t\|_\infty \le C_{\mathrm{uv}}(1 + C_{\mathrm{score}}) =: L_{\log}$,
and integrating over time yields the Lipschitz bound.

\noindent\textbf{Score is time-Lipschitz:}
Let $s_t := \nabla\log\mu_t$. Taking the spatial gradient of \eqref{eq:log_density_time_deriv}:
\begin{equation}
\label{eq:score_pde}
\partial_t s_t = -\nabla(\nabla\cdot u_t) - J_{u_t}^T s_t - (\nabla s_t)\,u_t,
\end{equation}
where $J_{u_t}$ is the Jacobian of $u_t$ and $\nabla s_t = \nabla^2\log\mu_t$.
By Assumption~\ref{ass:velocity_regularity} ($u_t \in C^2$ uniformly) and Lemma~\ref{lem:density_bounds} ($\log\mu_t \in C^2$ uniformly), all three terms on the right are bounded, giving $\|\partial_t s_t\|_\infty \le C$ for some uniform $C$. Integrating over time yields the second Lipschitz bound, defining $L_{\mathrm{score}} := \operatorname{ess\,sup}_{t\in[0,1]}\|\partial_t s_t\|_\infty$.
\end{proof}

\section{Variational Characterization of the Tilting Dynamics}
\label{app:energy-derivation}
The regularity of the marginal paths established above guarantees that the infinitesimal tilting rates required to track these marginals are theoretically well-defined. We now derive the linear system and the energy formulation governing these rates $(a_t, b_t)$.

\smallskip
\noindent
\textbf{From marginal feasibility to a linear system.} Recall $\pi_t\in\Pi(\mu_t,\nu_t)$ evolve according to the infinitesimal tilting dynamics $\dot\pi_t(x,y)=\pi_t(x,y)\big(a_t(x)+b_t(y)\big)$, where $(a_t,b_t)$ are the infinitesimal tilting rates defined in \eqref{eq:def-a-b-path}.
The marginal constraints $\pi_t\in\Pi(\mu_t,\nu_t)$ read
\[
\int_{\mathcal X} \pi_t(x,y)\,dy= \mu_t(x),
\qquad
\int_{\mathcal X} \pi_t(x,y)\,dx= \nu_t(y),
\]
Differentiating in $t$ and using
\eqref{eq:pi-dot} yields
\[
\int \pi_t(x,y)\big(a_t(x)+b_t(y)\big)\,dy= \partial_t \mu_t(x),
\qquad
\int \pi_t(x,y)\big(a_t(x)+b_t(y)\big)\,dx= \partial_t \nu_t(y).
\]
Dividing by $\mu_t(x)$ and $\nu_t(y)$, which are strictly positive by Lemma~\ref{lem:density_bounds}, and rewriting as conditional expectations under $\pi_t$,
we obtain the linear system
\begin{equation}
\label{eq:app_linear_system_ab}
\boxed{
\begin{aligned}
a_t(x)+\mathbb{E}_{\pi_t}\!\left[b_t(Y)\mid X=x\right] &= \zeta_t(x),\\
b_t(y)+\mathbb{E}_{\pi_t}\!\left[a_t(X)\mid Y=y\right] &= \eta_t(y),
\end{aligned}}
\qquad
\zeta_t(x):=\frac{\partial_t \mu_t(x)}{\mu_t(x)},\ \ \eta_t(y):=\frac{\partial_t \nu_t(y)}{\nu_t(y)}.
\end{equation}

\begin{definition}[Tilting operators]
\label{def:tilting_operators}
Fix $t\in[0,1]$ and let $\pi_t\in\Pi(\mu_t,\nu_t)$ be a coupling on $\mathcal X\times\mathcal X$.
Define the conditional expectation operators
\[
(T_t b)(x):=\mathbb E_{\pi_t}[b(Y)\mid X=x],\qquad
(S_t a)(y):=\mathbb E_{\pi_t}[a(X)\mid Y=y].
\]
With this notation, the linear system \eqref{eq:app_linear_system_ab} reads
\begin{equation}
\label{eq:app_linear_system_ab_TS}
a_t+T_t b_t=\zeta_t,\qquad b_t+S_t a_t=\eta_t,
\end{equation}
where $\zeta_t:=\partial_t\mu_t/\mu_t$ and $\eta_t:=\partial_t\nu_t/\nu_t$.
\end{definition}

\begin{proposition}{(Alternating regressions and normal equations)}
\label{prop:alt_regression}
The system \eqref{eq:app_linear_system_ab} is equivalent to the normal equations of the least-squares problems:
\begin{equation}
\label{eq:alt-updates}
    \begin{aligned}
    a_t &= \arg\min_{a}\ \mathbb{E}_{\pi_t}\!\Big[\big(a(X)+b_t(Y)-\zeta_t(X)\big)^2\Big],\\
    b_t &= \arg\min_{b}\ \mathbb{E}_{\pi_t}\!\Big[\big(a_t(X)+b(Y)-\eta_t(Y)\big)^2\Big].
    \end{aligned}
\end{equation}
\end{proposition}

\begin{proof}
We verify the equivalence for the $a$-update; the $b$-update follows symmetrically.
Fix $b = b_t$ and consider the quadratic objective
\[
\mathcal{L}(a) := \mathbb{E}_{\pi_t}\!\Big[\big(a(X)+b_t(Y)-\zeta_t(X)\big)^2\Big].
\]
This is a convex quadratic functional in $a \in L^2(\mu_t)$.
Its Fr\'echet derivative in the direction of any test function $p \in L^2(\mu_t)$ is
\[
\frac{d}{d\xi}\mathcal{L}(a+\xi p)\Big|_{\xi=0}
= 2\,\mathbb{E}_{\pi_t}\!\Big[p(X)\big(a(X)+b_t(Y)-\zeta_t(X)\big)\Big].
\]
Setting this to zero and using the tower property
$\mathbb{E}_{\pi_t}[b_t(Y)p(X)] = \mathbb{E}_{\mu_t}[p(X)(T_t b_t)(X)]$, we obtain
$\mathbb{E}_{\mu_t}[p(X)(a(X)+(T_t b_t)(X)-\zeta_t(X))]=0$ for all $p$,
hence $a + T_t b_t = \zeta_t$ $\mu_t$-a.s., which is the first equation of
\eqref{eq:app_linear_system_ab}.
An identical computation on $\mathcal{L}(b) := \mathbb{E}_{\pi_t}[(a_t(X)+b(Y)-\eta_t(Y))^2]$
yields $b + S_t a_t = \eta_t$ $\nu_t$-a.s.
\end{proof}

In particular, when the marginal path is generated by flow fields $(u_t,v_t)$ through the continuity equations
$\partial_t \mu_t+\nabla\!\cdot(\mu_t u_t)=0$ and $\partial_t \nu_t+\nabla\!\cdot(\nu_t v_t)=0$, the right-hand sides admit the explicit forms
\begin{equation}
\label{eq:zeta-eta}
    \begin{aligned}
    \zeta_t(x)
    = -\frac{\nabla\!\cdot(\mu_t u_t)}{\mu_t}(x)
    & = -\mathrm{div}(u_t)(x)-u_t(x)\cdot\nabla\log \mu_t(x),\\
    \eta_t(y)
    = -\frac{\nabla\!\cdot(\nu_t v_t)}{\nu_t}(y)
    & = -\mathrm{div}(v_t)(y)-v_t(y)\cdot\nabla\log \nu_t(y).
    \end{aligned}
\end{equation}
Hence, for any prescribed marginal path $(\mu_t,\nu_t)$, the infinitesimal tilt $(a_t,b_t)$ is fully characterized by
\eqref{eq:app_linear_system_ab}. 
In principle, the right-hand sides $\zeta_t=\partial_t \mu_t/\mu_t$ and $\eta_t=\partial_t \nu_t/\nu_t$ can be written explicitly from the continuity equations, but doing so involves score terms $\nabla\log \mu_t$ and $\nabla\log \nu_t$, which are typically inaccessible in our sample-based setting. To avoid estimating these scores, we rewrite the constraints in an equivalent weak form. 

\begin{lemma}[Weak Form via Integration by Parts]
\label{lem:ipp}
Using the continuity equations $\partial_t \mu_t=-\nabla\!\cdot(\mu_t u_t)$ and $\partial_t \nu_t=-\nabla\!\cdot(\nu_t v_t)$, for all sufficiently smooth test functions $a, b$, integration by parts yields:
\begin{equation*}
    \begin{aligned}
        \int_{\mathcal{X}} a(x)\,\partial_t \mu_t(x)\,dx
        &= -\int_{\mathcal{X}} a(x)\,\nabla\!\cdot(\mu_t u_t)(x)\,dx  = \int_{\mathcal{X}} \nabla a(x)\cdot u_t(x)\,\mu_t(x)\,dx = \mathbb{E}_{\mu_t}\!\big[\nabla a(X)\cdot u_t(X)\big], \\[0.3em]
        \int_{\mathcal{X}} b(y)\,\partial_t \nu_t(y)\,dy \,
        &= -\int_{\mathcal{X}} b(y)\,\nabla\!\cdot(\nu_t v_t)(y)\,dy \,  = \,  \int_{\mathcal{X}} \nabla b(y)\cdot v_t(y)\,\nu_t(y)\,dy \,\,= \mathbb{E}_{\nu_t}\!\big[\nabla b(Y)\cdot v_t(Y)\big].
    \end{aligned}
\end{equation*}
Note that the boundary terms over $\partial\mathcal{X}$ vanish naturally. Since probability mass is conserved within the compact domain $\mathcal{X}$, the velocity fields must be tangent to the boundary, implying the no-flux conditions $(\mu_tu_t)\cdot n = 0$ and $(\nu_tv_t)\cdot n = 0$.

Plugging these weak forms into the normal equations associated with formulation \eqref{eq:alt-updates} leads to an equivalent, fully sample-based objective that successfully avoids explicit density scores:
\begin{equation}
\label{eq:weak-alt-updates}
\begin{aligned}
a_t & =  \arg\min_{a}\ 
\mathbb{E}_{\pi_t}\!\Big[a(X)^2 +2a(X)b_t(Y)\Big]
-2\,\mathbb{E}_{\mu_t}\!\Big[\nabla a(X)\cdot u_t(X)\Big],\\
b_t & =  \arg\min_{b}\ 
\mathbb{E}_{\pi_t}\!\Big[b(Y)^2 +2a_t(X)b(Y)\Big]
-2\,\mathbb{E}_{\nu_t}\!\Big[\nabla b(Y)\cdot v_t(Y)\Big].
\end{aligned}
\end{equation}
\end{lemma}

Now to the proof of Proposition~\ref{prop:energy_wellposed} on the well-posedness of the energy functional.

\begin{proof}[Proof of Proposition~\ref{prop:energy_wellposed}]
Recall that the energy \eqref{eq:energy-ab} is
\[
\mathcal J_t(a,b)
:= \mathbb{E}_{\pi_t}\!\big[\,|a(X)+b(Y)|^2\,\big]
   -2\,\mathbb{E}_{\mu_t}\!\big[\,\nabla a(X)\!\cdot\! u_t(X)\,\big]
   -2\,\mathbb{E}_{\nu_t}\!\big[\,\nabla b(Y)\!\cdot\! v_t(Y)\,\big].
\]
Fix $t\in[0,1]$. By the tower property and Fubini's theorem, for any square-integrable $a,b$,
\begin{equation}
\label{eq:app_dual_T_S}
\mathbb E_{\mu_t}\!\big[a(X)\,(T_t b)(X)\big]
=
\mathbb E_{\pi_t}[a(X)b(Y)]
=
\mathbb E_{\nu_t}\!\big[(S_t a)(Y)\,b(Y)\big].
\end{equation}

\noindent\textbf{Optimality conditions $\Leftrightarrow$ linear system \eqref{eq:app_linear_system_ab}}:
Consider the quadratic energy $\mathcal J_t$ defined above. Since $\mathcal J_t$ is a convex quadratic functional on $L^2(\mu_t)\times L^2(\nu_t)$,  its minimizer is characterized by the vanishing of its
Fr\'echet derivatives. Let $p\in C^1(\mathcal X)$ be any arbitrary test function.
Expanding $\mathcal J_t(a+\xi p,b)$ and differentiating at $\xi=0$ yields
\[
\mathbb E_{\pi_t}\!\big[(a(X)+b(Y))\,p(X)\big]
=
\mathbb E_{\mu_t}\!\big[\nabla p(X)\cdot u_t(X)\big].
\]
By Lemma~\ref{lem:ipp}, the right-hand side equals $\mathbb E_{\mu_t}[p(X)\,\zeta_t(X)]$ in weak form,
where $\zeta_t=\partial_t\mu_t/\mu_t$. Using
$\mathbb E_{\pi_t}[b(Y)p(X)]=\mathbb E_{\mu_t}[p(X)\,(T_t b)(X)]$, for all $p$, 
we obtain
\[
\mathbb E_{\mu_t}\!\big[p(X)\,(a(X)+(T_t b)(X)-\zeta_t(X))\big]=0,
\]
hence
\[
a+T_t b=\zeta_t\qquad \mu_t\text{-a.s.}
\]
An analogous argument with test functions depending only on $y$ gives, with $\eta_t=\partial_t\nu_t/\nu_t$,
\[
b+S_t a=\eta_t\qquad \nu_t\text{-a.s.}
\]
Therefore, the critical points of $\mathcal J_t$ are exactly the solutions of the linear system
\[
a+T_t b=\zeta_t,\qquad b+S_t a=\eta_t,
\]
which is equivalent to \eqref{eq:app_linear_system_ab}.

\noindent\textbf{Existence and uniqueness modulo gauge:}
The quadratic part of $\mathcal J_t$ is $\mathbb E_{\pi_t}[(a(X)+b(Y))^2]$, whose kernel consists precisely of constant gauges $(a,b)=(c,-c)$. Hence $\mathcal J_t$ is strictly convex on the quotient space modulo the gauge and admits a unique minimizer in each gauge class. Since $\mathcal J_t$ is a convex quadratic functional, this minimizer is characterized by the above linear system.

\noindent\textbf{Block coordinate descent and convergence:}
The associated block coordinate descent reads
\[
a^{k+1}=\zeta_t-T_t b^{k},
\qquad
b^{k+1}=\eta_t-S_t a^{k+1}.
\]
This scheme corresponds to a Gauss--Seidel iteration for the linear system \eqref{eq:app_linear_system_ab_TS}. To apply standard convergence results one works on the quotient space with the gauge fixed (e.g.\ $\int a\,d\mu_t=0$), on which $\mathcal{J}_t$ is strictly convex and the linear system is non-singular.
Since $\mathcal J_t$ is a strictly convex quadratic functional (modulo the gauge), classical results on block coordinate descent for convex quadratic problems, or equivalently Gauss--Seidel methods for linear systems, ensure that the iterates decrease $\mathcal J_t$ monotonically and converge (modulo the gauge) to the unique minimizer; see, e.g., \cite{bertsekas1997nonlinear,tseng2001convergence}.
\end{proof}

\begin{remark}
In practice, $a_t$ and $b_t$ are parameterized by neural networks and all expectations are approximated by minibatches of paired samples $(X,Y)\sim\pi_t$ and marginal samples $X\sim\mu_t$, $Y\sim\nu_t$; the gradients $\nabla a$ and
$\nabla b$ are obtained efficiently via automatic differentiation. Once $(a_t,b_t)$ are estimated, we update the coupling by the infinitesimal tilting rule  \eqref{eq:pi-dot}.
\end{remark}

\section{Derivations for the Tilted Conditional Vector Fields}
\label{app:tilted_vector_fields_derivation}
We provide the formal derivations for the vector field updates presented in Section \ref{subsec:sampling_along_tilting}.

\subsection{Derivation of Algorithm A: Weighted CFM} 
For a discrete step $\delta > 0$, the practical tilting update is defined by the density ratio $d\tilde{\pi}_{t+\delta}(x,y) \propto \exp(\delta(a_t(x)+b_t(y))) d\tilde{\pi}_t(x,y)$. By the definition of the conditional expectation, the updated vector field evaluates directly via integration:
\begin{align*}
    \tilde{\beta}_s^{t+\delta}(z) = \mathbb{E}_{\tilde{\pi}_{t+\delta}}[\dot{I}_s \mid I_s = z] 
    &= \frac{\int \dot{I}_s \exp(\delta(a_t(x)+b_t(y))) \delta(I_s - z) d\tilde{\pi}_t(x,y)}{\int \exp(\delta(a_t(x)+b_t(y))) \delta(I_s - z) d\tilde{\pi}_t(x,y)} \\
    &= \frac{\mathbb{E}_{\tilde{\pi}_t}\!\left[ \dot{I}_s \exp(\delta(a_t(X)+b_t(Y))) \mid I_s=z \right]}{\mathbb{E}_{\tilde{\pi}_t}\!\left[ \exp(\delta(a_t(X)+b_t(Y))) \mid I_s=z \right]}.
\end{align*}
Furthermore, since the conditional expectation $\mathbb{E}_{P}[Y \mid X]$ uniquely minimizes the mean squared error $\mathbb{E}_{P}[\|Y - f(X)\|^2]$, we can express the target as a regression objective. Applying the exact same change of measure, where the global normalization constant vanishes in the $\arg\min$:
\begin{align*}
    \tilde{\beta}_s^{t+\delta} 
    &= \arg\min_{\beta} \mathbb{E}_{\tilde{\pi}_{t+\delta}}\!\left[\|\dot{I}_s - \beta(I_s)\|^2\right] \\
    &= \arg\min_{\beta} \mathbb{E}_{\tilde{\pi}_t}\!\left[\exp(\delta(a_t(X)+b_t(Y))) \|\dot{I}_s - \beta(I_s)\|^2\right],
\end{align*}
which rigorously yields the Weighted CFM objective \eqref{eq:weighted_cfm_lsq}.

\subsection{Derivation of Algorithm B: Covariance ODE}
Under continuous tilting, the measure evolves via the differential equation $\partial_t(d\pi_t) = (a_t(x)+b_t(y)) d\pi_t$. We express the exact conditional vector field as the ratio of two marginal integrals:
\begin{equation*}
    \beta_s^t(z) = \frac{\int \dot{I}_s \delta(I_s - z) d\pi_t(x,y)}{\int \delta(I_s - z) d\pi_t(x,y)}.
\end{equation*}
Taking the time derivative $\frac{d}{dt}$ via the quotient rule, the operator $\partial_t$ acts linearly on $d\pi_t$, bringing down the multiplier $(a_t(x)+b_t(y))$:
\begin{align*}
    \frac{d}{dt}\beta_s^t(z) 
    &= \frac{\int \dot{I}_s (a_t(x) + b_t(y)) \delta(I_s - z) d\pi_t}{\int \delta(I_s - z) d\pi_t} - \left( \frac{\int \dot{I}_s \delta(I_s - z) d\pi_t}{\int \delta(I_s - z) d\pi_t} \right) \frac{\int (a_t(x) + b_t(y)) \delta(I_s - z) d\pi_t}{\int \delta(I_s - z) d\pi_t} \\
    &= \mathbb{E}_{\pi_t}\!\left[\dot{I}_s (a_t(X) + b_t(Y)) \mid I_s=z\right] - \mathbb{E}_{\pi_t}\!\left[\dot{I}_s \mid I_s=z\right] \mathbb{E}_{\pi_t}\!\left[a_t(X) + b_t(Y) \mid I_s=z\right].
\end{align*}
By the algebraic definition of covariance, this expansion collapses strictly to:
\begin{equation}
\label{eq:beta_cov_update}
    \frac{d}{dt}\beta_s^t(z) = \mathrm{Cov}_{\pi_t}\!\left(\dot{I}_s, a_t(X) + b_t(Y) \mid I_s=z\right),
\end{equation}
recovering the exact covariance evolution in \eqref{eq:beta_cov_update}.

\section{Analytical Properties of the Continuous Tilting Rates}
\label{app:analy_titlting_rate}
With the linear system for the tilting rates $(a_t, b_t)$ established, our next goal is to prove that these rates are both uniformly bounded and time-Lipschitz. We achieve this in two steps: first by bounding the corresponding relative dual potentials via entropic optimal transport theory, and then by proving the strict contraction of the conditional expectation operators.

\subsection{Uniform Boundedness via Relative Potentials}
In this subsection, we rely on the structural assumptions concerning the state space and the initial ground cost. By \eqref{eq:tilt-relation}, the log-form first-order conditions for the reference and target problems yield:
\begin{equation*}
\pi^\star(x,y) = \pi'(x,y) \exp\big(\varphi(x) + \psi(y)\big),
\end{equation*}
where $\varphi, \psi : \mathcal{X} \to \mathbb{R}$ act as the relative dual potentials ensuring the new marginal constraints $\mu^\star$ and $\nu^\star$ are met. From this relative entropy projection perspective, the reliance on an explicit ground cost $c$ is entirely circumvented; the geometry of the transport is fully dictated by the reference coupling $\pi'$.

\begin{lemma}[Uniform Boundedness and Lipschitz Regularity of Relative Potentials]
\label{lem:phi_psi_regularity}
Under Assumptions \ref{ass:compact_X}, \ref{ass:coupling_regularity}, and Lemma \ref{lem:density_bounds}, the relative potentials $\varphi_t$ and $\psi_t$ (under the gauge $\int e^{\psi_t} d\nu_t = 1$) are uniformly bounded and spatially Lipschitz on $\mathcal{X}$. Specifically, there exist uniform constants $C_{\infty}, C_{\mathrm{Lip}} < \infty$ such that for all $t \in [0,1]$:
\[
\|\varphi_t\|_{L^\infty(\mathcal{X})} + \|\psi_t\|_{L^\infty(\mathcal{X})} \le C_{\infty},
\]
and
\[
\mathrm{Lip}(\varphi_t) \le C_{\mathrm{Lip}}, \qquad \mathrm{Lip}(\psi_t) \le C_{\mathrm{Lip}}.
\]
\end{lemma}

\begin{proof}
Fix $t \in [0,1]$. Marginalizing the relation $\pi_t(x,y) = \pi'(x,y) e^{\varphi_t(x) + \psi_t(y)}$ yields the fixed-point equations:
\begin{equation}
\label{eq:phipsi_sinkhorn}
\begin{aligned}
\varphi_t(x) &= \log \mu_t(x) - \log \int_{\mathcal{X}} \pi'(x,y) e^{\psi_t(y)} dy, \\
\psi_t(y) &= \log \nu_t(y) - \log \int_{\mathcal{X}} \pi'(x,y) e^{\varphi_t(x)} dx.
\end{aligned}
\end{equation}

By Lemma \ref{lem:density_bounds}, $\log \mu_t$ and $\log \nu_t$ are uniformly bounded. By Assumption \ref{ass:coupling_regularity}, the effective kernel $-\log \pi'$ is bounded. Viewing \eqref{eq:phipsi_sinkhorn} as standard EOT dual equations, these bounds directly guarantee uniform $L^\infty$ bounds for $\varphi_t$ and $\psi_t$ (e.g., Lemma 4.9 in \cite{nutz2021introduction}), establishing $C_{\infty}$.

For spatial regularity, differencing \eqref{eq:phipsi_sinkhorn} for any $x_1, x_2 \in \mathcal{X}$ gives:
\begin{align*}
|\varphi_t(x_1) - \varphi_t(x_2)| &\le |\log \mu_t(x_1) - \log \mu_t(x_2)| + \left| \log \frac{\int \pi'(x_1,y) e^{\psi_t(y)} dy}{\int \pi'(x_2,y) e^{\psi_t(y)} dy} \right|.
\end{align*}
By the 1-Lipschitz property of the Log-Sum-Exp operator, the second term is bounded by the term $\sup_y |\log \pi'(x_1,y) - \log \pi'(x_2,y)|$. 
Hence:
\[
\mathrm{Lip}(\varphi_t) \le \mathrm{Lip}(\log \mu_t) + \mathrm{Lip}_{\mathcal{X}}(\log \pi').
\]
From Lemma \ref{lem:density_bounds}, $\sup_{t \in [0,1]} \mathrm{Lip}(\log \mu_t) < \infty$ due to bounded scores. From Assumption \ref{ass:coupling_regularity}, $\mathrm{Lip}_{\mathcal{X}}(\log \pi') \le C_\mathrm{OT}$. Summing these provides the uniform Lipschitz constant $C_{\mathrm{Lip}}$ for $\varphi_t$. A symmetric argument applies to $\psi_t$.
\end{proof}


\begin{proposition}[Time-Lipschitz Regularity of Relative Potentials]
\label{prop:pot_Lip_time_Linfty}
Under Assumptions \ref{ass:compact_X} to \ref{ass:initial_regularity}, for each $t \in [0,1]$, let $(\varphi_t, \psi_t)$ be the unique relative dual potentials for the target coupling $\pi'$ satisfying the marginal constraints $\mu_t, \nu_t$, under the gauge $\int e^{\psi_t} \, d\nu_t = 1$. Then, there exists a uniform constant $L_{\mathrm{pot}} < \infty$ such that for all $r, t \in [0,1]$:
\begin{equation*}
\|\varphi_r - \varphi_t\|_{L^\infty(\mathcal{X})} + \|\psi_r - \psi_t\|_{L^\infty(\mathcal{X})} \le L_{\mathrm{pot}} |r - t|.
\end{equation*}
The constant $L_{\mathrm{pot}}$ depends only on the target coupling regularity $C_\mathrm{OT}$ and the flow regularity constants $C_\mathrm{uv}, C_{\mu\nu}$.
\end{proposition}

\begin{proof}
The relative dual potentials $(\varphi_t, \psi_t)$ satisfy the coupled Sinkhorn fixed-point equations as \eqref{eq:phipsi_sinkhorn}. We give the proof in two steps.

\noindent\textbf{Stability of the relative Sinkhorn operator:}  By Lemma \ref{lem:temporal_marginal_regularity}, the marginal log-densities are Lipschitz continuous in time. Hence, there exists a uniform constant $L_{\log}$ such that for all $r, t \in [0,1]$:
\begin{equation*}
\label{eq:marg_stability_ref}
\|\log \mu_r - \log \mu_t\|_\infty \le L_{\log} |r - t|, \quad \|\log \nu_r - \log \nu_t\|_\infty \le L_{\log} |r - t|.
\end{equation*}
Let $\mathcal{O}_t$ be the composite operator such that $\psi_t = \mathcal{O}_t(\psi_t)$ defined using Sinkhorn's update \eqref{eq:phipsi_sinkhorn} for a fixed potential $\psi$:
\begin{equation*}
\mathcal{O}_t(\psi)(y) := \log \nu_t(y) - \log \int_{\mathcal{X}} \pi'(x,y) \exp \left( \log \mu_t(x) - \log \int_{\mathcal{X}} \pi'(x,z) e^{\psi(z)} dz \right) dx.
\end{equation*}
To quantify the temporal sensitivity of the operator, define the intermediate function $h(x, t) := \log \mu_t(x) - \log \int \pi'(x,z) e^{\psi(z)} dz$. For any fixed $\psi$, the difference between operators at times $r$ and $t$ satisfies:
\begin{align*}
|\mathcal{O}_r(\psi)(y) - \mathcal{O}_t(\psi)(y)| &\le |\log \nu_r(y) - \log \nu_t(y)| + \left| \log \frac{ \int \pi'(x,y) e^{h(x, r)} dx }{ \int \pi'(x,y) e^{h(x, t)} dx } \right|.
\end{align*}
By the 1-Lipschitz property of the Log-Sum-Exp operator, the second term is bounded by $\|h(\cdot, r) - h(\cdot, t)\|_\infty$. Since $\psi$ is fixed, the integral components in $h$ cancel out, leaving:
\[
\|h(\cdot, r) - h(\cdot, t)\|_\infty = \|\log \mu_r - \log \mu_t\|_\infty.
\]
Combining this with Step 1, the operator perturbation is strictly bounded by:
\begin{equation}
\label{eq:operator_perturbation}
\|\mathcal{O}_r(\psi) - \mathcal{O}_t(\psi)\|_\infty \le 2 L_{\log} |r - t|.
\end{equation}

\noindent\textbf{Uniform contraction and fixed-point sensitivity:} The operator $\mathcal{O}_t$ acts on the space of bounded functions. Under the Hilbert projective metric, the marginals $\mu_t, \nu_t$ act purely as pointwise translations, leaving the contraction rate unaffected. By Birkhoff's Theorem, the contraction ratio $\rho$ of $\mathcal{O}_t$ is determined \textit{exclusively} by the dynamic range of the strictly positive kernel $\pi'(x,y)$:
\begin{equation*}
\rho = \left( \frac{\sqrt{\Delta} - 1}{\sqrt{\Delta} + 1} \right)^2, \quad \text{where } \Delta = \exp\left( \sup_{x,y} \log \pi'(x,y) - \inf_{x,y} \log \pi'(x,y) \right).
\end{equation*}
By Assumption \ref{ass:coupling_regularity}, $\|\log \pi'\|_\infty \le C_\mathrm{OT}$, ensuring that $\Delta < \infty$. Consequently, $\rho < 1$ is a uniform constant for all $t \in [0,1]$. 

Let $\psi_r$ and $\psi_t$ be the fixed points of $\mathcal{O}_r$ and $\mathcal{O}_t$. By \eqref{eq:operator_perturbation}, utilizing the triangle inequality and the uniform contraction property:
\begin{align*}
\|\psi_r - \psi_t\|_\infty &\le \|\mathcal{O}_r(\psi_r) - \mathcal{O}_r(\psi_t)\|_\infty + \|\mathcal{O}_r(\psi_t) - \mathcal{O}_t(\psi_t)\|_\infty \\
&\le \rho \|\psi_r - \psi_t\|_\infty + 2 L_{\log} |r - t|.
\end{align*}
Rearranging yields $\|\psi_r - \psi_t\|_\infty \le \frac{2 L_{\log}}{1 - \rho} |r - t|$. By symmetry, the analogous bound holds for $\|\varphi_r - \varphi_t\|_\infty$. Since each individual potential is Lipschitz with constant $\frac{2L_{\log}}{1-\rho}$, their sum is bounded by $L_{\mathrm{pot}} := \frac{4 L_{\log}}{1 - \rho}$ (twice the individual constant) which completes the proof.
\end{proof}

\begin{proposition}[Uniform $L^\infty$ Boundedness of Infinitesimal Tilting Rates]
\label{prop:ab_uniform_Linfty}
Under the conditions of Proposition \ref{prop:pot_Lip_time_Linfty}, let $(a_t, b_t)$ be the infinitesimal tilting rates defined in \eqref{eq:def-a-b-path}. There exists a uniform constant $C_{ab}^\infty < \infty$, depending only on the regularity constants $C_\mathrm{OT}, C_\mathrm{uv}$, and $C_{\mu\nu}$, such that:
\begin{equation*}
\operatorname*{ess\,sup}_{t\in[0,1]} \|a_t(x) + b_t(y)\|_{L^\infty(\mathcal{X} \times \mathcal{X})} \le \operatorname*{ess\,sup}_{t\in[0,1]} \Big( \|a_t\|_{L^\infty(\mathcal{X})} + \|b_t\|_{L^\infty(\mathcal{X})} \Big) \le C_{ab}^\infty.
\end{equation*}
\end{proposition}


\begin{proof}
By Proposition \ref{prop:pot_Lip_time_Linfty}, each potential $\varphi_t$ (resp.\ $\psi_t$) is time-Lipschitz with individual constant $L_{\mathrm{pot}}/2$. Therefore, by Rademacher's theorem \citep[Theorem~3.1.2]{evans2015measure}, their time derivatives $a_t = \dot\varphi_t$ and $b_t = \dot\psi_t$ exist for a.e.\ $t \in [0,1]$ (applied pointwise in $x$ to the scalar Lipschitz function $t\mapsto\varphi_t(x)$) and are essentially bounded:
\[
\|a_t\|_{L^\infty(\mathcal{X})} \le \tfrac{L_{\mathrm{pot}}}{2} \quad \text{and} \quad \|b_t\|_{L^\infty(\mathcal{X})} \le \tfrac{L_{\mathrm{pot}}}{2}.
\]
Setting $C_{ab}^\infty := L_{\mathrm{pot}}$ bounds the sum of their norms. The complete chain of inequalities then follows directly from the triangle inequality $\|a_t(x) + b_t(y)\|_\infty \le \|a_t\|_\infty + \|b_t\|_\infty \le C_{ab}^\infty$.
\end{proof}

\subsection{Spatiotemporal Stability and Doeblin Minorization}
\label{app:stab_bound_titling_rate}
The uniform boundedness of the tilting rates guarantees that the joint coupling $\pi_t$ remains bounded away from zero. We formalize this via a uniform Doeblin minorization condition, which acts as the key mechanism to prove the Lipschitz stability of the tilting operators.

Recall from \eqref{eq:tilt-relation} that the optimal coupling $\pi_t$ admits the joint density $\pi_t(x,y) = \pi'(x,y) \exp(\varphi_t(x) + \psi_t(y))$. We directly define its associated conditional densities as:
\begin{equation}
\label{eq:cond_pi_t}
\pi_t(y \mid x) = \frac{\pi_t(x,y)}{\mu_t(x)}, \quad \text{and} \quad \pi_t(x \mid y) = \frac{\pi_t(x,y)}{\nu_t(y)}.
\end{equation}

\begin{lemma}[Uniform Doeblin minorization]
\label{lem:doeblin_from_bounds}
Under Assumptions \ref{ass:compact_X} to \ref{ass:initial_regularity}, let $C_\infty$ be the uniform $L^\infty$ bound on the relative potentials from Lemma \ref{lem:phi_psi_regularity}. There exists a constant $\alpha \in (0, 1]$, depending only on the previously defined regularity constants, such that for all $t \in [0,1]$:
\[
\pi_t(y \mid x) \ge \alpha \nu_t(y) \quad \text{and} \quad \pi_t(x \mid y) \ge \alpha \mu_t(x),
\]
uniformly for all $(x, y) \in \mathcal{X} \times \mathcal{X}$. The minorization constant is explicitly given by $\alpha := \exp(-C_\mathrm{OT} - 2C_\infty) / (C_{\mu}^{\max} C_{\nu}^{\max})$.
\end{lemma}

\begin{proof}
The Doeblin minorization condition is equivalent to uniformly bounding the density ratio $\pi_t(y \mid x) / \nu_t(y)$ away from zero. Using Bayes' theorem, this ratio is symmetric and exactly corresponds to the point-wise mutual information ratio of the joint coupling with \eqref{eq:cond_pi_t}:
\begin{equation}
\label{eq:pmi_ratio}
\frac{\pi_t(y \mid x)}{\nu_t(y)} = \frac{\pi_t(x,y)}{\mu_t(x)\nu_t(y)} = \frac{\pi'(x,y) \exp(\varphi_t(x) + \psi_t(y))}{\mu_t(x)\nu_t(y)}.
\end{equation}
We can bound the numerator from below and the denominator from above using our established uniform bounds:
By Assumption \ref{ass:coupling_regularity}, $\|\log \pi'\|_\infty \le C_\mathrm{OT}$, implying $\pi'(x,y) \ge e^{-C_\mathrm{OT}}$.
By Lemma \ref{lem:phi_psi_regularity}, $\varphi_t(x) \ge -C_\infty$ and $\psi_t(y) \ge -C_\infty$.
By Lemma \ref{lem:density_bounds}, the marginal densities are strictly bounded above by constants $C_{\mu}^{\max}$ and $C_{\nu}^{\max}$.
Substituting these bounds into \eqref{eq:pmi_ratio} directly yields the uniform lower bound:
\[
\frac{\pi_t(y \mid x)}{\nu_t(y)} \ge \frac{\exp(-C_\mathrm{OT}) \cdot \exp(-2C_\infty)}{C_{\mu}^{\max} C_{\nu}^{\max}} =: \alpha.
\]
By the symmetric structure of the joint density in \eqref{eq:pmi_ratio}, the exact same lower bound applies to the inverse conditional ratio $\pi_t(x \mid y) / \mu_t(x)$, completing the proof.
\end{proof}

This uniform minorization property ensures that the conditional expectation operators contract strictly on the subspace of centered functions. This is a crucial requirement for the stability of the linear system governing the tilting rates.

\begin{lemma}[$L^\infty$ Contraction of Conditional Operators]
\label{lem:Doeblin_centered_Linfty}
Let $S_t$ and $T_t$ denote the conditional expectation operators defined by $(T_t f)(x) := \int_{\mathcal{X}} f(y) \pi_t(y \mid x) dy$ and $(S_t g)(y) := \int_{\mathcal{X}} g(x) \pi_t(x \mid y) dx$. Let $\alpha \in (0,1]$ be the uniform minorization constant from Lemma \ref{lem:doeblin_from_bounds}. For all $t \in [0,1]$, the following hold:
\begin{enumerate}[label=(\roman*)]
    \item Non-expansion: The operators are non-expansive on $L^\infty(\mathcal{X})$:
    \[
    \|T_t f\|_\infty \le \|f\|_\infty, \qquad \|S_t f\|_\infty \le \|f\|_\infty, \quad \forall f \in L^\infty(\mathcal{X}).
    \]
    \item Contraction on centered functions: The operators strictly contract functions that are centered with respect to the target marginals. Specifically, for any $b \in L^\infty(\mathcal{X})$ with $\int_{\mathcal{X}} b(y) \nu_t(y) dy = 0$:
    \[
    \|T_t b\|_\infty \le (1-\alpha)\|b\|_\infty.
    \]
    (A symmetric bound holds for $S_t$ acting on $\mu_t$-centered functions).
    \item Composite contraction: For any $\nu_t$-centered $b \in L^\infty(\mathcal{X})$, the composite operator acts as a strict contraction:
    \[
    \|S_t T_t b\|_\infty \le (1-\alpha)\|b\|_\infty.
    \]
\end{enumerate}
\end{lemma}

\begin{proof}
(i) Since $\pi_t(\cdot \mid x)$ is a valid probability density for any $x \in \mathcal{X}$, $T_t$ acts as a Markov operator. By the triangle inequality for integrals:
\[
|(T_t f)(x)| \le \int_{\mathcal{X}} |f(y)| \pi_t(y \mid x) dy \le \|f\|_\infty \int_{\mathcal{X}} \pi_t(y \mid x) dy = \|f\|_\infty.
\]
Taking the supremum over $x$ yields $\|T_t f\|_\infty \le \|f\|_\infty$. The proof for $S_t$ is identical.

(ii) By Lemma \ref{lem:doeblin_from_bounds}, the conditional density admits the uniform lower bound $\pi_t(y \mid x) \ge \alpha \nu_t(y)$. This allows us to decompose the density as a convex combination:
\[
\pi_t(y \mid x) = \alpha \nu_t(y) + (1-\alpha) r_{t,x}(y),
\]
where $r_{t,x}(y) := \frac{\pi_t(y \mid x) - \alpha \nu_t(y)}{1-\alpha}$ is a valid residual probability density (non-negative and integrating to 1). Applying $T_t$ to a $\nu_t$-centered function $b$:
\[
(T_t b)(x) = \alpha \underbrace{\int_{\mathcal{X}} b(y) \nu_t(y) dy}_{= 0} + (1-\alpha) \int_{\mathcal{X}} b(y) r_{t,x}(y) dy.
\]
Taking the absolute value, we bound the remaining integral:
\[
|(T_t b)(x)| \le (1-\alpha)\|b\|_\infty \int_{\mathcal{X}} r_{t,x}(y) dy = (1-\alpha)\|b\|_\infty.
\]
Since this holds for all $x$, taking the supremum yields $\|T_t b\|_\infty \le (1-\alpha)\|b\|_\infty$.

(iii) Applying the non-expansion of $S_t$ from part (i) sequentially with the contraction of $T_t$ on centered functions from part (ii), we immediately obtain:
\[
\|S_t (T_t b)\|_\infty \le \|T_t b\|_\infty \le (1-\alpha)\|b\|_\infty,
\]
which confirms the strict contraction of the composite operator on the centered subspace.
\end{proof}

\begin{lemma}[Time-Lipschitz Stability of Conditional Operators]
\label{lem:operator_stability_Linfty}
Let $S_t, T_t$ be the conditional expectation operators defined in Lemma \ref{lem:Doeblin_centered_Linfty}. Under the conditions of Proposition \ref{prop:pot_Lip_time_Linfty}, there exists a uniform constant $L_{\mathrm{op}} < \infty$ such that for any $h \in L^\infty(\mathcal{X})$ and all $r, t \in [0,1]$:
\[
\|(T_r - T_t)h\|_\infty \le L_{\mathrm{op}} \|h\|_\infty |r - t|, \qquad \|(S_r - S_t)h\|_\infty \le L_{\mathrm{op}} \|h\|_\infty |r - t|.
\]
The constant $L_{\mathrm{op}}$ depends only on the potential regularity constant $L_{\mathrm{pot}}$.
\end{lemma}

\begin{proof}
We prove the estimate for $T_t$ (the argument for $S_t$ is entirely symmetric). Fix $h \in L^\infty(\mathcal{X})$. By the definition of the conditional expectation, the difference is bounded by the total variation (TV) distance between the conditional kernels:
\begin{equation}
\label{eq:Tdiff_TV_final}
\|(T_r - T_t)h\|_\infty \le \|h\|_\infty \sup_{x \in \mathcal{X}} \|\pi_r(\cdot \mid x) - \pi_t(\cdot \mid x)\|_{\mathrm{TV}}.
\end{equation}
Using the relative entropy tilting formulation $\pi_t(x,y) = \pi'(x,y) \exp(\varphi_t(x) + \psi_t(y))$, the potential $\varphi_t(x)$ acts purely as a normalization factor for fixed $x$. Thus, the conditional density simplifies to:
\[
\pi_t(y \mid x) = \frac{\pi'(x,y) \exp(\psi_t(y))}{Z_t(x)}, \quad \text{where } Z_t(x) := \int_{\mathcal{X}} \pi'(x,z) \exp(\psi_t(z)) dz.
\]
Crucially, this isolates the time-dependence strictly to the potential $\psi_t$. Define the point-wise logarithmic difference between the kernels at times $r$ and $t$:
\[
\Delta(x,y) := \log \pi_r(y \mid x) - \log \pi_t(y \mid x) = \big(\psi_r(y) - \psi_t(y)\big) - \big(\log Z_r(x) - \log Z_t(x)\big).
\]
By Proposition \ref{prop:pot_Lip_time_Linfty}, $\|\psi_r - \psi_t\|_\infty \le L_{\mathrm{pot}} |r - t|$. Consequently, the normalization term satisfies:
\[
|\log Z_r(x) - \log Z_t(x)| = \left| \log \frac{\int \pi'(x,z) e^{\psi_r(z)} dz}{\int \pi'(x,z) e^{\psi_t(z)} dz} \right| \le \|\psi_r - \psi_t\|_\infty \le L_{\mathrm{pot}} |r - t|,
\]
where we invoked again the 1-Lipschitz property of the Log-Sum-Exp operator. Thus, the logarithmic difference is uniformly bounded:
\[
|\Delta(x,y)| \le \|\psi_r - \psi_t\|_\infty + |\log Z_r(x) - \log Z_t(x)| \le 2 L_{\mathrm{pot}} |r - t|.
\]
Let $\delta := 2 L_{\mathrm{pot}} |r - t|$. Using the identity $\|\pi_r - \pi_t\|_{\mathrm{TV}} = \frac{1}{2} \int |e^{\Delta} - 1| \pi_t dy$ and the elementary inequality $|e^\Delta - 1| \le |\Delta| e^{|\Delta|} \le \delta e^\delta$, we obtain:
\[
\|\pi_r(\cdot \mid x) - \pi_t(\cdot \mid x)\|_{\mathrm{TV}} \le \frac{1}{2} \delta e^\delta = L_{\mathrm{pot}} |r - t| \exp(2 L_{\mathrm{pot}} |r - t|).
\]
Since $|r - t| \le 1$, the exponential term is bounded by $e^{2 L_{\mathrm{pot}}}$. Taking the supremum over $x$ yields:
\[
\sup_{x \in \mathcal{X}} \|\pi_r(\cdot \mid x) - \pi_t(\cdot \mid x)\|_{\mathrm{TV}} \le \big( L_{\mathrm{pot}} e^{2 L_{\mathrm{pot}}} \big) |r - t|.
\]
Combining this with \eqref{eq:Tdiff_TV_final} concludes the proof, explicitly giving $L_{\mathrm{op}} := L_{\mathrm{pot}} e^{2 L_{\mathrm{pot}}}$.
\end{proof}

\begin{lemma}[Spatiotemporal Regularity of the Forcing Terms]
\label{lem:forcing_regularity}
For forcing terms $\zeta_t := \partial_t \log \mu_t$ and $\eta_t := \partial_t \log \nu_t$, under Assumptions \ref{ass:velocity_regularity} and \ref{ass:initial_regularity}, there exist uniform constants $C_{\mathrm{force}}, L_{\mathrm{force}} < \infty$ such that for all $r, t \in [0,1]$:
\begin{align*}
    \sup_{t \in [0,1]} \Big( \|\zeta_t\|_{L^\infty(\mathcal{X})} + \|\eta_t\|_{L^\infty(\mathcal{X})} + \mathrm{Lip}_{\mathcal{X}}(\zeta_t) + \mathrm{Lip}_{\mathcal{X}}(\eta_t) \Big) &\le C_{\mathrm{force}}, \\
    \|\zeta_r - \zeta_t\|_{L^\infty(\mathcal{X})} + \|\eta_r - \eta_t\|_{L^\infty(\mathcal{X})} &\le L_{\mathrm{force}} |r - t|.
\end{align*}
\end{lemma}

\begin{proof}
This proof uses the uniform $C^2$ spatial regularity from Lemma~\ref{lem:density_bounds} and the time-Lipschitz regularity of the marginal scores from Lemma~\ref{lem:temporal_marginal_regularity} (Part~2).

We detail the estimates for $\zeta_t$; identical arguments hold for $\eta_t$. By the continuity equation, the forcing term is given by:
\begin{equation}
\label{eq:zeta_def_explicit}
\zeta_t = -\nabla \cdot u_t - u_t \cdot \nabla \log \mu_t.
\end{equation}

By Assumption \ref{ass:velocity_regularity} and Lemma \ref{lem:density_bounds}, $u_t, \nabla \cdot u_t$, and $\nabla \log \mu_t$ are uniformly bounded. Applying the triangle inequality to \eqref{eq:zeta_def_explicit} yields the uniform $L^\infty$ bound:
\[
\|\zeta_t\|_\infty \le \|\nabla \cdot u_t\|_\infty + \|u_t\|_\infty \|\nabla \log \mu_t\|_\infty \le C_{\mathrm{uv}} + C_{\mathrm{uv}} C_{\mathrm{score}}.
\]

For the spatial Lipschitz regularity, we use the product rule $\mathrm{Lip}_{\mathcal{X}}(fg) \le \|f\|_\infty \mathrm{Lip}_{\mathcal{X}}(g) + \|g\|_\infty \mathrm{Lip}_{\mathcal{X}}(f)$. Applying this to \eqref{eq:zeta_def_explicit} gives:
\[
\mathrm{Lip}_{\mathcal{X}}(\zeta_t) \le \mathrm{Lip}_{\mathcal{X}}(\nabla \cdot u_t) + \|u_t\|_\infty \mathrm{Lip}_{\mathcal{X}}(\nabla \log \mu_t) + \|\nabla \log \mu_t\|_\infty \mathrm{Lip}_{\mathcal{X}}(u_t).
\]
Since $u_t \in C^2(\mathcal{X})$ and $\log \mu_t \in C^2(\mathcal{X})$ (Lemma \ref{lem:density_bounds}), all supremum norms and spatial Lipschitz constants on the right-hand side are bounded uniformly in $t$. Summing these spatial bounds for both $\zeta_t$ and $\eta_t$ establishes the existence of the single uniform constant $C_{\mathrm{force}} < \infty$.

For the temporal regularity, we evaluate the difference $\zeta_r - \zeta_t$ via the decomposition:
\[
\zeta_r - \zeta_t = -(\nabla \cdot u_r - \nabla \cdot u_t) - (u_r - u_t) \cdot \nabla \log \mu_r - u_t \cdot (\nabla \log \mu_r - \nabla \log \mu_t).
\]
By Assumption \ref{ass:velocity_regularity}, the velocity field and its divergence are time-Lipschitz with constant $L_{\mathrm{uv}}$ (since $\|u_r-u_t\|_{C^1} \le L_{\mathrm{uv}}|r-t|$). Combining this with Lemma \ref{lem:score_time_lip}, we take the $L^\infty$ norm and apply the triangle inequality:
\begin{align*}
\|\zeta_r - \zeta_t\|_\infty &\le \|\nabla \cdot u_r - \nabla \cdot u_t\|_\infty + \|u_r - u_t\|_\infty \|\nabla \log \mu_r\|_\infty + \|u_t\|_\infty \|\nabla \log \mu_r - \nabla \log \mu_t\|_\infty \\
&\le L_{\mathrm{uv}}|r-t| + L_{\mathrm{uv}}|r-t| C_{\mathrm{score}} + C_{\mathrm{uv}} L_{\mathrm{score}} |r-t|.
\end{align*}
Factoring out $|r-t|$, we define $L_{\mathrm{force}} := 2 \big( L_{\mathrm{uv}}(1 + C_{\mathrm{score}}) + C_{\mathrm{uv}} L_{\mathrm{score}} \big)$ (to account for both $\zeta$ and $\eta$), which completes the proof.
\end{proof}

\begin{proposition}[Lipschitz Stability of the Tilting Rates in $L^\infty$]
\label{prop:ab_stability}
Under Assumptions \ref{ass:compact_X} to \ref{ass:initial_regularity}, the infinitesimal tilting rates $(a_t, b_t)$ defined in \eqref{eq:app_linear_system_ab} are Lipschitz continuous in time from $[0,1]$ to $L^\infty(\mathcal{X})$. Specifically, there exists a uniform constant $L_{ab} < \infty$ such that for all $r, t \in [0,1]$:
\begin{equation*}
\|a_r - a_t\|_{L^\infty(\mathcal{X})} + \|b_r - b_t\|_{L^\infty(\mathcal{X})} \le L_{ab} |r - t|.
\end{equation*}
The constant $L_{ab}$ depends only on the contraction ratio $\alpha$, operator stability $L_{\mathrm{op}}$, and the forcing regularity $L_{\mathrm{force}}$.
\end{proposition}

\begin{proof}
For each $t \in [0,1]$, under the gauge $\int b_t \, d\nu_t = 0$, the tilting rate $b_t$ is the unique solution to the operator equation $(I - S_t T_t) b_t = \eta_t - S_t \zeta_t$. Correspondingly, $a_t = \zeta_t - T_t b_t$. Fix $r, t \in [0,1]$ and consider the difference $(I - S_r T_r)(b_r - b_t)$. By adding and subtracting $(I - S_r T_r)b_t$, we obtain the fundamental perturbation identity:
\begin{equation}
\label{eq:perturbation_identity}
(I - S_r T_r)(b_r - b_t) = (\eta_r - \eta_t) - S_r(\zeta_r - \zeta_t) - (S_r - S_t) \zeta_t + S_r(T_r - T_t) b_t + (S_r - S_t) T_t b_t.
\end{equation}
By Lemma \ref{lem:Doeblin_centered_Linfty}, the composite operator $S_r T_r$ is a strict contraction on the $\nu_r$-centered subspace with $\|S_r T_r\| \le 1 - \alpha$. Thus, the operator $(I - S_r T_r)$ is invertible on this subspace, and its inverse is uniformly bounded by $\|(I - S_r T_r)^{-1}\| \le 1/\alpha$. Applying this to \eqref{eq:perturbation_identity} and invoking the non-expansivity $\|S_r\| \le 1$ from Lemma \ref{lem:Doeblin_centered_Linfty}:
\begin{align}
\label{eq:b_diff_expansion}
\|b_r - b_t\|_\infty \le \frac{1}{\alpha} \Big( & \|\eta_r - \eta_t\|_\infty + \|\zeta_r - \zeta_t\|_\infty + \|(S_r - S_t) \zeta_t\|_\infty \nonumber \\
& + \|(T_r - T_t) b_t\|_\infty + \|(S_r - S_t) T_t b_t\|_\infty \Big).
\end{align}
We bound each term in \eqref{eq:b_diff_expansion} using previously established regularities:
\begin{enumerate}[label=(\roman*)]
    \item Forcing terms: By Lemma \ref{lem:forcing_regularity}, $\|\eta_r - \eta_t\|_\infty + \|\zeta_r - \zeta_t\|_\infty \le L_{\mathrm{force}} |r - t|$.
    \item Operator stability on $\zeta_t$: By Lemma \ref{lem:operator_stability_Linfty}, $\|(S_r - S_t) \zeta_t\|_\infty \le L_{\mathrm{op}} \|\zeta_t\|_\infty |r - t| \le L_{\mathrm{op}} C_{\mathrm{force}} |r - t|$.
    \item Operator stability on $b_t$: By Proposition \ref{prop:ab_uniform_Linfty}, $\|b_t\|_\infty \le C_{ab}^\infty$. Thus, $\|(T_r - T_t) b_t\|_\infty \le L_{\mathrm{op}} C_{ab}^\infty |r - t|$.
    \item Composite stability: Since $\|T_t b_t\|_\infty \le \|b_t\|_\infty \le C_{ab}^\infty$, by Lemma \ref{lem:operator_stability_Linfty} and Proposition \ref{prop:ab_uniform_Linfty}, the final term satisfies $\|(S_r - S_t) T_t b_t\|_\infty \le L_{\mathrm{op}} C_{ab}^\infty |r - t|$.
\end{enumerate}
Summing these estimates yields $\|b_r - b_t\|_\infty \le L_b |r - t|$ with $L_b := \frac{1}{\alpha}(L_{\mathrm{force}} + L_{\mathrm{op}} C_{\mathrm{force}} + 2 L_{\mathrm{op}} C_{ab}^\infty)$.
Using the relation $a_t = \zeta_t - T_t b_t$, we decompose the difference as:
\begin{align*}
\|a_r - a_t\|_\infty &= \|(\zeta_r - \zeta_t) - (T_r b_r - T_t b_t)\|_\infty \\
&\le \|\zeta_r - \zeta_t\|_\infty + \|(T_r - T_t) b_r\|_\infty + \|T_t(b_r - b_t)\|_\infty \\
&\le L_{\mathrm{force}} |r - t| + L_{\mathrm{op}} C_{ab}^\infty |r - t| + L_b |r - t| =: L_a |r - t|.
\end{align*}
Defining $L_{ab} := L_a + L_b$ concludes the proof.
\end{proof}

\begin{corollary}[Joint Spatiotemporal Regularity of the Tilting Rate]
\label{cor:F_regularity}
Let $F_t(x,y) := a_t(x) + b_t(y)$ be the joint tilting rate on $\mathcal{X} \times \mathcal{X}$. Under the assumptions of Proposition \ref{prop:ab_stability}, the family $\{F_t\}_{t \in [0,1]}$ satisfies:
\begin{enumerate}[label=(\roman*)]
    \item Time-Lipschitz: $\|F_r - F_t\|_{L^\infty(\mathcal{X} \times \mathcal{X})} \le L_{ab} |r - t|$.
    \item Space-Lipschitz: $\sup_{t \in [0,1]} \mathrm{Lip}_{\mathcal{X} \times \mathcal{X}}(F_t) \le L_{\mathcal{X}}$ for some uniform $L_{\mathcal{X}} < \infty$.
\end{enumerate}
\end{corollary}

\begin{proof}
The temporal regularity is immediate from Proposition \ref{prop:ab_stability}. For spatial regularity, fix $t \in [0,1]$. By Lemma \ref{lem:forcing_regularity}, $\zeta_t$ is uniformly Lipschitz. For the operator term $T_t b_t(x) = \int b_t(y) \pi_t(y \mid x) dy$, we substitute the conditional density:
\[
T_t b_t(x) = \frac{\int b_t(y) \pi'(x,y) e^{\psi_t(y)} dy}{\int \pi'(x,z) e^{\psi_t(z)} dz}.
\]
Under Assumption \ref{ass:coupling_regularity}, the reference kernel $\pi'(x,y)$ is Lipschitz in $x$. Since $b_t$ and $\psi_t$ are uniformly bounded (Lemma \ref{lem:phi_psi_regularity} and Proposition \ref{prop:ab_uniform_Linfty}), and the denominator is bounded away from zero by Doeblin minorization, the quotient $T_t b_t$ is spatially Lipschitz. A symmetric argument for $S_t a_t$ ensures the uniform spatial Lipschitzness of $b_t$, and thus of $F_t$.
\end{proof}

\section{Global Error Analysis of the Discrete Scheme}
\label{app:error_analysis}
Equipped with the uniform bounds and time-Lipschitz regularity of the theoretical tilting rates, we are now ready to analyze the discrete-time implementation of the algorithm. We track the accumulation of local truncation and sampling propagation errors in Total Variation (TV) distance, and eventually translate this to the $L^\infty$ accuracy of the learned regression target $\beta_s^t$.

\subsection{Stability Analysis of the Coupling Measure}
To prove Theorem \ref{thm:global_tv_convergence}, we must analyze the propagation of the estimation error across discrete iterations. We quantify this global error at time $t$ using the Total Variation (TV) distance:
\begin{equation*}
\varepsilon_t := \|\pi_t - \tilde{\pi}_t\|_{\mathrm{TV}} = \frac{1}{2} \int_{\mathcal{X}^2} |d\pi_t - d\tilde{\pi}_t|.
\end{equation*}

Recall the exact continuous-time coupling $\pi_{t+\delta}$ defined in \eqref{eq:pi_t+delta} and the practical discrete-time coupling $\tilde{\pi}_{t+\delta}$ defined in \eqref{eq:pi_tilde_t+delta}. To rigorously decouple the temporal discretization error from the spatial sampling error, we introduce an intermediate semi-discrete coupling $\bar{\pi}_{t+\delta}$. This auxiliary measure serves as a theoretical bridge: it utilizes the exact base measure $\pi_t$ from the previous step, but approximates the tilting weight via the first-order frozen rates at time $t$:
\begin{equation}
\label{eq:pi_semidiscrete}
d\bar{\pi}_{t+\delta}(x,y) := \frac{\exp\left( \delta F_t(x,y) \right) d\pi_t(x,y)}{\int_{\mathcal{X}^2} \exp\left( \delta F_t(x,y) \right) d\pi_t(x,y)}.
\end{equation}

By applying the triangle inequality to these three measures, the global error at the subsequent step $t+\delta$ admits the following natural decomposition:
\begin{equation}
\label{eq:tv_decomp_main}
\varepsilon_{t+\delta} \le \underbrace{\|\pi_{t+\delta} - \bar{\pi}_{t+\delta}\|_{\mathrm{TV}}}_{\text{Local Discretization: } E_{\mathrm{local}}} + \underbrace{\|\bar{\pi}_{t+\delta} - \tilde{\pi}_{t+\delta}\|_{\mathrm{TV}}}_{\text{Measure Propagation: } E_{\mathrm{prop}}},
\end{equation}
where $E_{\mathrm{local}}$ isolates the local truncation error of the Euler-tilting scheme (the difference between exact and frozen weights on the exact base measure), and $E_{\mathrm{prop}}$ captures the amplification of the existing sampling error $\varepsilon_t$ through the discrete tilting operator.

We first bound the local truncation error, which arises from freezing the continuous tilting rate over a discrete interval of length $\delta$.

\begin{lemma}[Local Truncation Error]
\label{lem:local_error_tv}
Under the time-Lipschitz regularity of the tilting rates (Proposition \ref{prop:ab_stability}), the local truncation error satisfies:
\begin{equation}
E_{\mathrm{local}} = \|\pi_{t+\delta} - \bar{\pi}_{t+\delta}\|_{\mathrm{TV}} \le C_1 \delta^2,
\end{equation}
where the constant $C_1$ depends on the potential Lipschitz constant $L_{ab}$ and the uniform bound $C_{ab}^\infty$.
\end{lemma}


\begin{proof}
Let $G(x,y) := \int_t^{t+\delta} F_r(x,y)\,dr$ and $\bar{G}(x,y) := \delta F_t(x,y)$
denote the exact and frozen exponents, with $Z := \int e^G\,d\pi_t$ and $\bar{Z} := \int e^{\bar{G}}\,d\pi_t$.

\medskip\noindent\textbf{Bound the exponent difference.}
By Corollary~\ref{cor:F_regularity},
$\|F_r - F_t\|_\infty \le L_{ab}|r - t|$. Therefore the pointwise exponent difference satisfies:
\[
\|\Delta\|_\infty := \left\| G - \bar{G} \right\|_\infty
= \left\| \int_t^{t+\delta} (F_r - F_t)\,dr \right\|_\infty
\le \int_t^{t+\delta} L_{ab}(r-t)\,dr = \tfrac{1}{2}L_{ab}\delta^2.
\]
Set $\varepsilon := \|\Delta\|_\infty \le \tfrac{1}{2}L_{ab}\delta^2$.

\medskip\noindent\textbf{Stability of Gibbs measures.}
Since $e^G = e^{\bar{G}} \cdot e^{\Delta}$, the Radon-Nikodym derivative of $\pi_{t+\delta}$
with respect to $\bar{\pi}_{t+\delta}$ is
$\frac{d\pi_{t+\delta}}{d\bar{\pi}_{t+\delta}} = e^{\Delta} \cdot \frac{\bar{Z}}{Z}.$
By the 1-Lipschitz property of $\log\int e^{(\cdot)}\,d\pi_t$,
\[
\left|\log \frac{Z}{\bar{Z}}\right|
= \left|\log\frac{\int e^G d\pi_t}{\int e^{\bar{G}} d\pi_t}\right|
\le \|G - \bar{G}\|_\infty = \varepsilon.
\]
Hence $|\log(d\pi_{t+\delta}/d\bar{\pi}_{t+\delta})| = |\Delta - \log(Z/\bar{Z})| \le 2\varepsilon$ uniformly.
Using the identity $2\|P - Q\|_{TV} = \mathbb{E}_Q[|dP/dQ - 1|]$ and the elementary bound $|e^u - 1| \le |u|\,e^{|u|}$:
\[
E_{\mathrm{local}}
= \|\pi_{t+\delta} - \bar{\pi}_{t+\delta}\|_{TV}
\le \tfrac{1}{2}\,\mathbb{E}_{\bar{\pi}_{t+\delta}}\!\left[\left|e^{\Delta - \log(Z/\bar{Z})} - 1\right|\right]
\le \varepsilon\,e^{2\varepsilon}
\le \tfrac{1}{2}L_{ab}\delta^2 \cdot e^{L_{ab}\delta^2}.
\]
For $\delta \le 1$, set $C_1 := \tfrac{1}{2}L_{ab}e^{L_{ab}}$ so that $E_{\mathrm{local}} \le C_1\delta^2$.
\end{proof}

Next, we analyze how the existing sampling error from the previous iteration is amplified when passed through the frozen tilting operator.
\begin{lemma}[Propagation of Sampling Error]
\label{lem:prop_error_tv}
The propagation of the existing coupling error $\varepsilon_t$ through the frozen tilting step satisfies:
\begin{equation}
E_{\mathrm{prop}} = \|\bar{\pi}_{t+\delta} - \tilde{\pi}_{t+\delta}\|_{\mathrm{TV}} \le (1 + C_2 \delta) \varepsilon_t,
\end{equation}
where $C_2 := 4 C_{ab}^\infty e^{2 C_{ab}^\infty}$ is a uniform constant depending only on the bound $C_{ab}^\infty$.
\end{lemma}

\begin{proof}
Recall that $\bar{\pi}_{t+\delta}$ and $\tilde{\pi}_{t+\delta}$ are obtained by tilting $\pi_t$ and $\tilde{\pi}_t$ respectively with the same frozen weight $\exp(\delta F_t)$. By the definition of the Total Variation distance:
\begin{equation}
\label{eq:prop_tv_start}
2 E_{\mathrm{prop}} = \int_{\mathcal{X}^2} \left| \frac{e^{\delta F_t(x,y)} d\pi_t(x,y)}{\int e^{\delta F_t} d\pi_t} - \frac{e^{\delta F_t(x,y)} d\tilde{\pi}_t(x,y)}{\int e^{\delta F_t} d\tilde{\pi}_t} \right|.
\end{equation}
Using the triangle inequality, the integrand is bounded as:
\begin{align*}
\left| \frac{e^{\delta F_t} d\pi_t}{\int e^{\delta F_t} d\pi_t} - \frac{e^{\delta F_t} d\tilde{\pi}_t}{\int e^{\delta F_t} d\tilde{\pi}_t} \right| \le \underbrace{\frac{e^{\delta F_t}}{\int e^{\delta F_t} d\pi_t} |d\pi_t - d\tilde{\pi}_t|}_{(\mathrm{I})} + \underbrace{e^{\delta F_t} d\tilde{\pi}_t \left| \frac{1}{\int e^{\delta F_t} d\pi_t} - \frac{1}{\int e^{\delta F_t} d\tilde{\pi}_t} \right|}_{(\mathrm{II})}.
\end{align*}
Integrating term $(\mathrm{I})$ over $\mathcal{X}^2$:
\[
\int(\mathrm{I}) \le \frac{\|e^{\delta F_t}\|_\infty}{\int e^{\delta F_t} d\pi_t}\cdot 2\varepsilon_t.
\]
For term $(\mathrm{II})$, write $D_\pi := \int e^{\delta F_t} d\pi_t$ and $D_{\tilde\pi} := \int e^{\delta F_t} d\tilde\pi_t$. Since $\pi_t$ and $\tilde\pi_t$ are both probability measures, $\int 1\,d(\tilde\pi_t - \pi_t) = 0$, so
\[
|D_{\tilde\pi} - D_\pi|
= \left|\int \bigl(e^{\delta F_t}-1\bigr)\,d(\tilde\pi_t-\pi_t)\right|
\le \|e^{\delta F_t}-1\|_\infty \cdot 2\varepsilon_t
\le \bigl(e^{\delta C_{ab}^\infty}-1\bigr)\cdot 2\varepsilon_t.
\]
Hence:
\[
\int(\mathrm{II})
= \frac{|D_{\tilde\pi} - D_\pi|}{D_\pi D_{\tilde\pi}}\int e^{\delta F_t}d\tilde\pi_t
= \frac{|D_{\tilde\pi}-D_\pi|}{D_\pi}
\le \frac{(e^{\delta C_{ab}^\infty}-1)\cdot 2\varepsilon_t}{e^{-\delta C_{ab}^\infty}}
= 2\bigl(e^{2\delta C_{ab}^\infty}-e^{\delta C_{ab}^\infty}\bigr)\varepsilon_t.
\]
Summing both contributions:
\begin{align*}
2 E_{\mathrm{prop}}
\le 2e^{2\delta C_{ab}^\infty}\varepsilon_t + 2\bigl(e^{2\delta C_{ab}^\infty}-e^{\delta C_{ab}^\infty}\bigr)\varepsilon_t
= 2\bigl(2e^{2\delta C_{ab}^\infty}-e^{\delta C_{ab}^\infty}\bigr)\varepsilon_t,
\end{align*}
so $E_{\mathrm{prop}} \le \bigl(2e^{2\delta C_{ab}^\infty}-e^{\delta C_{ab}^\infty}\bigr)\varepsilon_t$.
Note that this bound is tight at $\delta=0$: $2\cdot 1 - 1 = 1$, recovering $E_{\mathrm{prop}} \le \varepsilon_t$ exactly.

Let $f(\delta) := 2e^{2\delta C_{ab}^\infty} - e^{\delta C_{ab}^\infty}$. Then $f(0)=1$ and
$f'(\delta) = 4C_{ab}^\infty e^{2\delta C_{ab}^\infty} - C_{ab}^\infty e^{\delta C_{ab}^\infty} \le 4C_{ab}^\infty e^{2C_{ab}^\infty}$
for $\delta\in[0,1]$. By the mean value theorem, $f(\delta) \le 1 + 4C_{ab}^\infty e^{2C_{ab}^\infty}\,\delta$, so defining $C_2 := 4\,C_{ab}^\infty e^{2C_{ab}^\infty}$:

\begin{equation}
E_{\mathrm{prop}} \le (1 + C_2 \delta)\,\varepsilon_t.
\end{equation}
This completes the proof.
\end{proof}

\begin{corollary}[Recursive Error Accumulation]
\label{cor:error_recursion}
Let $\varepsilon_t = \|\pi_t - \tilde{\pi}_t\|_{\mathrm{TV}}$ be the global coupling error at time $t$. Under the conditions of Lemma \ref{lem:local_error_tv} and Lemma \ref{lem:prop_error_tv}, the error at the next iteration $t+\delta$ satisfies the following linear recurrence:
\begin{equation}
\label{eq:tv_recursion_final}
\varepsilon_{t+\delta} \le (1 + C_2 \delta) \varepsilon_t + C_1 \delta^2,
\end{equation}
where $C_1$ is the local truncation constant and $C_2$ is the error propagation constant.
\end{corollary}

With the step-wise recurrence established, we now synthesize these components to solve for the global convergence bound over the entire algorithm horizon.

\begin{proof}[Proof of Theorem \ref{thm:global_tv_convergence}]
Consider the sequence of errors $\{\varepsilon_{t_k}\}_{k=0}^N$ where $t_k = k\delta$. The algorithm is supposed to be initialized with $\tilde{\pi}_0 = \pi_0 = \pi'$ (both coincide with the reference coupling by construction), so $\varepsilon_0 = \|\pi_0 - \tilde{\pi}_0\|_{\mathrm{TV}} = 0$. By iterating the recurrence from Corollary \ref{cor:error_recursion} starting from $\varepsilon_0 = 0$:
\begin{equation*}
\varepsilon_{t_k} \le C_1 \delta^2 \sum_{j=0}^{k-1} (1 + C_2 \delta)^j = C_1 \delta^2 \frac{(1 + C_2 \delta)^k - 1}{(1 + C_2 \delta) - 1}.
\end{equation*}
Simplifying the denominator gives:
\begin{equation*}
\varepsilon_{t_k} \le \frac{C_1 \delta}{C_2} \left[ (1 + C_2 \delta)^k - 1 \right].
\end{equation*}
Using the bound $(1 + C_2 \delta)^k \le e^{k C_2 \delta}$, and setting $k=N=1/\delta$, we obtain:
\begin{equation*}
\varepsilon_1 \le \frac{C_1}{C_2} (e^{C_2} - 1) \delta,
\end{equation*}
which completes the proof.
\end{proof}

\begin{remark}[Imperfect Initialization]
\label{rmk:imperfect_init}
If the algorithm is initialized with a coupling $\tilde{\pi}_0 \ne \pi_0$, so that $\varepsilon_0 := \|\pi_0 - \tilde{\pi}_0\|_{\mathrm{TV}} > 0$, then iterating the recurrence \eqref{eq:tv_recursion_final} from $\varepsilon_0$ yields
\begin{equation*}
    \varepsilon_1 \le e^{C_2}\,\varepsilon_0 + \frac{C_1}{C_2}(e^{C_2} - 1)\,\delta.
\end{equation*}
The first term decays the initialization error at rate $e^{C_2}$ over the full unit interval, while the second term is the same $O(\delta)$ discretization bias as in the perfect-initialization case. Hence, a sub-optimal reference coupling $\pi'$ incurs only a bounded additional error that does not grow with the number of steps.
\end{remark}

\subsection{Stability Analysis of the Regression Field}

In the stochastic formulation where $I_s = (1-s)X + sY + \sqrt{s(1-s)} \Xi$ with Gaussian noise $\Xi \sim \mathcal{N}(0, \sigma^2 I)$, the raw time derivative $\dot{I}_s$ is spatially unbounded. To rigorously analyze the regression target, we first establish that the analytical conditional vector field is well-defined and pointwise bounded over the compact state space $\mathcal{X}$.

\begin{lemma}[Regularity of the Stochastic Regression Environment]
\label{lem:beta_environment_regularity}
Under Assumptions \ref{ass:compact_X} and \ref{ass:coupling_regularity}, for any fixed interpolation time $s \in (0,1)$ and iteration $t \in [0,1]$, the time derivative satisfies $\dot{I}_s \in L^1(\mathbb{P})$. Consequently, the analytical conditional vector field $v_s(z \mid x,y) := \mathbb{E}[\dot{I}_s \mid I_s=z, x, y]$ is well-defined and uniformly bounded on the spatial domains by a deterministic constant $C_v(s) < \infty$. Furthermore, the marginalized transition density is strictly bounded away from zero by a constant $\underline{p}(s) > 0$.
\end{lemma}

\begin{proof}
We first establish the conditional bounds for the vector field. The time derivative is given by $\dot{I}_s = (Y - X) + \Xi \frac{1-2s}{2\sqrt{s(1-s)}}$. Since $\Xi \sim \mathcal{N}(0, \sigma^2 I)$, its first moment is finite ($\mathbb{E}[\|\Xi\|] < \infty$). Because $X, Y$ reside in the compact space $\mathcal{X}$, we trivially have $\mathbb{E}[\|\dot{I}_s\|] < \infty$. This $L^1(\mathbb{P})$ integrability guarantees the existence of the conditional expectation. 

Conditioning on $I_s = z, X=x, Y=y$ fixes the noise realization as $\Xi = \frac{z - (1-s)x - sy}{\sqrt{s(1-s)}}$. Substituting this yields the explicit continuous conditional vector field:
\begin{equation*}
    v_s(z \mid x,y) = (y - x) + \frac{1-2s}{2s(1-s)} \big(z - ((1-s)x + sy)\big) = (y - x) + \frac{1-2s}{2s(1-s)} \big(z - m_s(x, y)\big).
\end{equation*}
Since $x, y, z \in \mathcal{X}$ are within the compact state space, all spatial distances are bounded. For any fixed $s \in (0,1)$, the denominator $s(1-s)$ is strictly positive. Applying the triangle inequality yields a finite deterministic bound $\sup_{x,y,z \in \mathcal{X}} \|v_s(z \mid x,y)\| \le \mathrm{diam}(\mathcal{X}) + \frac{|1-2s|}{2s(1-s)} \sup_{x,y,z \in \mathcal{X}} \|z - m_s(x,y)\| =: C_v(s) < \infty$.

Next, we address the non-degenerate marginal densities. The Gaussian transition density $p(z \mid x,y) = \mathcal{N}(z ; m_s(x,y), \sigma_s^2 I)$ is continuous and strictly positive, where the variance schedule is $\sigma_s^2 = \sigma^2 s(1-s)$. For a fixed $s \in (0,1)$, on the compact domain $\mathcal{X}^3$, it attains a strictly positive minimum $p_{\min}(s) > 0$. By Proposition \ref{prop:ab_uniform_Linfty}, the joint density admits a uniform lower bound $\rho_{\min} > 0$. Integrating out the joint coupling explicitly provides the marginalized lower bound $p_{\pi_t}(z; s) \ge p_{\min}(s) \cdot \rho_{\min} \cdot \mathrm{Vol}(\mathcal{X}^2) =: \underline{p}(s) > 0$.
\end{proof}

With the pointwise regularity established, we now prove the main convergence result using a direct Lipschitz stability argument.}

\begin{proof}[Proof of Theorem \ref{thm:main_convergence}]
Fix $z \in \mathcal{X}$ and $s \in (0,1)$. Set $G := \int_t^{t+\delta} F_r\,dr$ and $\bar{G} := \delta F_t$. Define the shorthand
\[
    \mathcal{N}(\mu, W) := \int v_s(z|x,y)\,p(z|x,y)\,W(x,y)\,d\mu(x,y), \quad
    \mathcal{D}(\mu, W) := \int p(z|x,y)\,W(x,y)\,d\mu(x,y)
\]
so that $\beta_s^t(z) = \mathcal{N}(\pi_t, e^G)/\mathcal{D}(\pi_t, e^G)$ and $\tilde{\beta}_s^t(z) = \mathcal{N}(\tilde{\pi}_t, e^{\bar{G}})/\mathcal{D}(\tilde{\pi}_t, e^{\bar{G}})$.

Applying the fractional perturbation identity $\bigl|\tfrac{A}{B} - \tfrac{a}{b}\bigr| \le \tfrac{|A-a|}{B} + \tfrac{|a|\,|B-b|}{Bb}$ with
$(A,B) = (\mathcal{N}(\pi_t,e^G), \mathcal{D}(\pi_t,e^G))$ and $(a,b) = (\mathcal{N}(\tilde{\pi}_t,e^{\bar{G}}), \mathcal{D}(\tilde{\pi}_t,e^{\bar{G}}))$, and using the lower bound $\mathcal{D}(\pi_t, e^G) \ge \underline{p}(s)\,e^{-\delta C_{ab}^\infty}$ from Lemma~\ref{lem:beta_environment_regularity}:
\begin{align*}
|\beta_s^t(z) - \tilde{\beta}_s^t(z)|
\le \frac{|\mathcal{N}(\pi_t,e^G) - \mathcal{N}(\tilde{\pi}_t,e^{\bar{G}})|}{\mathcal{D}(\pi_t,e^G)}
+ \frac{C_v(s)\,|\mathcal{D}(\pi_t,e^G) - \mathcal{D}(\tilde{\pi}_t,e^{\bar{G}})|}{\mathcal{D}(\pi_t,e^G)\,\mathcal{D}(\tilde{\pi}_t,e^{\bar{G}})}.
\end{align*}

For each of these two differences, apply the triangle inequality by splitting at $(\pi_t, e^{\bar{G}})$:
\begin{align*}
|\mathcal{N}(\pi_t,e^G) - \mathcal{N}(\tilde{\pi}_t,e^{\bar{G}})| &\le
\underbrace{|\mathcal{N}(\pi_t,e^G) - \mathcal{N}(\pi_t,e^{\bar{G}})|}_{\text{weight discrepancy}} +
\underbrace{|\mathcal{N}(\pi_t,e^{\bar{G}}) - \mathcal{N}(\tilde{\pi}_t,e^{\bar{G}})|}_{\text{measure discrepancy}},
\end{align*}
and similarly for $\mathcal{D}$.

\textbf{Weight discrepancy.} Since $|e^G - e^{\bar{G}}| \le e^{\delta C_{ab}^\infty}\|G - \bar{G}\|_\infty$ and $\|G - \bar{G}\|_\infty \le \tfrac{1}{2}L_{ab}\delta^2$ (Lemma~\ref{lem:local_error_tv}, Step~1):
\[
|\mathcal{N}(\pi_t,e^G) - \mathcal{N}(\pi_t,e^{\bar{G}})| \le C_v(s)\,e^{\delta C_{ab}^\infty} \cdot \tfrac{1}{2}L_{ab}\delta^2,
\]
and the same bound holds for $|\mathcal{D}(\pi_t,e^G) - \mathcal{D}(\pi_t,e^{\bar{G}})|$.

\textbf{Measure discrepancy.} By the TV duality bound for signed integrals and the uniform weight bound $\|e^{\bar{G}}\|_\infty \le e^{\delta C_{ab}^\infty}$:
\[
|\mathcal{N}(\pi_t,e^{\bar{G}}) - \mathcal{N}(\tilde{\pi}_t,e^{\bar{G}})| \le 2\,C_v(s)\,e^{\delta C_{ab}^\infty}\,\varepsilon_t,
\]
and similarly $|\mathcal{D}(\pi_t,e^{\bar{G}}) - \mathcal{D}(\tilde{\pi}_t,e^{\bar{G}})| \le 2\,e^{\delta C_{ab}^\infty}\,\varepsilon_t$, where $\varepsilon_t := \|\pi_t - \tilde{\pi}_t\|_{\mathrm{TV}}$.

\textbf{Combining.} Substituting into the fractional perturbation bound and absorbing constants:
\begin{align*}
|\beta_s^t(z) - \tilde{\beta}_s^t(z)|
&\le \underbrace{\left(\frac{C_v(s)}{\underline{p}(s)} + \frac{C_v(s)}{\underline{p}(s)^2}\right) e^{2\delta C_{ab}^\infty} \cdot \tfrac{1}{2}L_{ab}\delta^2}_{=:\,K_1(s)\,\delta^2}
+ \underbrace{2C_v(s)\!\left(\frac{1}{\underline{p}(s)} + \frac{1}{\underline{p}(s)^2}\right) e^{2\delta C_{ab}^\infty}\,\varepsilon_t}_{=:\,K_2(s)\,\varepsilon_t}.
\end{align*}
Substituting the global convergence bound $\varepsilon_t \le C_{\mathrm{global}}\,\delta$ yields the pointwise error bound:
\begin{equation}
\label{eq:pointwise_beta_bound}
\|\beta_s^t - \tilde{\beta}_s^t\|_{L^\infty(\mathcal{X})} \le K_1(s)\,\delta^2 + K_2(s)\,C_{\mathrm{global}}\,\delta \le C_\beta(s)\,\delta,
\end{equation}
where $C_\beta(s) := K_1(s) + K_2(s)\,C_{\mathrm{global}}$ is deterministic and strictly finite for any fixed $s \in (0,1)$, which completes the proof.
\end{proof}

\begin{remark}[The Boundary Singularity and Pointwise Convergence]
\label{rmk:boundary_singularity}
It is crucial to note that the bounding constant $C_\beta(s)$ diverges as $s \to 0$ or $1$, driven by the $\mathcal{O}(s^{-1})$ term in the conditional vector field $v_s$. This divergence is not a limitation of our proof technique, but a fundamental mathematical property of diffusion and flow matching frameworks. Forcing dispersed probability mass into deterministic Dirac delta distributions at the boundaries requires an infinitely large vector field. Consequently, directly integrating the $L^\infty$ error bound over the entire interval $[0,1]$ is analytically divergent. Our point-wise guarantee for any internal time $s \in (0,1)$ rigorously reflects this reality, and perfectly justifies the standard engineering practice of utilizing time truncation (early stopping) during numerical simulation to avoid boundary singularities.
\end{remark}

\input{appendix}

\bibliographystyle{unsrtnat}
\bibliography{refs_complete}

\end{document}

%% file: misc/tikz_transposed.tex

\definecolor{mBlue}  {RGB}{52,108,222}
\definecolor{mPink}  {RGB}{205,65,100}
\definecolor{mTeal}  {RGB}{0,168,152}
\definecolor{mPurple}{RGB}{116,76,192}
\definecolor{mLBlue} {RGB}{132,182,242}
\definecolor{mLPink} {RGB}{248,148,178}
\definecolor{mSlate} {RGB}{96,108,124}
\definecolor{mDark}  {RGB}{40,44,56}
\definecolor{mBg}    {RGB}{247,250,254}
\definecolor{mCard}  {RGB}{255,255,255}
\definecolor{mBorder}{RGB}{208,216,232}

\newcommand{\drawblobT}[5][]{%
  \begin{scope}[#1]
    \pgfmathsetseed{#5}
    \colorlet{blobcolor}{#3}
    \pgfmathsetmacro{\rA}{0.65 + 0.10*rand}
    \pgfmathsetmacro{\rB}{0.63 + 0.10*rand}
    \pgfmathsetmacro{\rC}{0.65 + 0.10*rand}
    \pgfmathsetmacro{\rD}{0.63 + 0.10*rand}
    \pgfmathsetmacro{\rE}{0.65 + 0.10*rand}
    \pgfmathsetmacro{\rF}{0.63 + 0.10*rand}
    \pgfmathsetmacro{\rG}{0.65 + 0.10*rand}
    \pgfmathsetmacro{\rH}{0.63 + 0.10*rand}
    \pgfmathsetmacro{\rI}{0.65 + 0.10*rand}
    \pgfmathsetmacro{\rJ}{0.63 + 0.10*rand}
    \pgfmathsetmacro{\rK}{0.65 + 0.10*rand}
    \pgfmathsetmacro{\rL}{0.63 + 0.10*rand}
    \def\blobpath{
      plot[smooth cycle] coordinates {
        ($(#2)+(0:\rA)$)    ($(#2)+(30:\rB)$)
        ($(#2)+(60:\rC)$)   ($(#2)+(90:\rD)$)
        ($(#2)+(120:\rE)$)  ($(#2)+(150:\rF)$)
        ($(#2)+(180:\rG)$)  ($(#2)+(210:\rH)$)
        ($(#2)+(240:\rI)$)  ($(#2)+(270:\rJ)$)
        ($(#2)+(300:\rK)$)  ($(#2)+(330:\rL)$)
      }
    }
    \begin{scope}
      \clip \blobpath;
      \shade[inner color=blobcolor!40, outer color=blobcolor!11]
        ($(#2)+(-1.1,-1.1)$) rectangle ($(#2)+(1.1,1.1)$);
    \end{scope}
    \begin{scope}
      \clip \blobpath;
      \foreach \dx/\dy in {
        -0.20/-0.08, 0.00/0.17, 0.18/-0.11,
        -0.06/0.03,  0.23/0.10, -0.22/0.15,  0.10/-0.04
      }{
        \fill[blobcolor!88!black] ($(#2)+(\dx,\dy)$) circle (1.1pt);
        \fill[white, opacity=0.36] ($(#2)+(\dx+0.03,\dy+0.03)$) circle (0.35pt);
      }
    \end{scope}
    \draw[blobcolor!70!black, line width=1.0pt] \blobpath;
  \end{scope}
}

\begin{tikzpicture}[>=Stealth, thick, node distance=0pt]

\begin{scope}[on background layer]
  \fill[mBg, rounded corners=12pt]
    (-1.5,-3.25) rectangle (11.4,4.50);
  \draw[mBorder, rounded corners=12pt, line width=0.6pt]
    (-1.5,-3.25) rectangle (11.4,4.50);
\end{scope}

\drawblobT{0,3.2}{mBlue}{}{8}
\drawblobT{5,3.2}{mLBlue}{}{13}
\drawblobT{10,3.2}{mTeal}{}{16}

\node[font=\small\bfseries, text=mBlue!85!black,  anchor=north] at ( 1.0, 3.20) {$\mu'$};
\node[font=\small\bfseries, text=mLBlue!85!black, anchor=north] at ( 6.0, 3.20) {$\mu_t$};
\node[font=\small\bfseries, text=mTeal!85!black,  anchor=north] at (11.0, 3.20) {$\mu$};

\draw[mBlue!68,        ->, line width=1.3pt] (0.90, 3.2) -- (4.10, 3.2);
\draw[mLBlue!55!mTeal, ->, line width=1.3pt] (5.90, 3.2) -- (9.10, 3.2);

\node[font=\scriptsize, text=mDark!72, anchor=south, above=3pt]
  at (5.1, 3.8) {$\partial_t\mu_t+\nabla\!\cdot(\mu_t u_t)=0$};

\drawblobT{0,0}{mPink}{}{2}
\drawblobT{5,0}{mLPink}{}{4}
\drawblobT{10,0}{mPurple}{}{18}

\node[font=\small\bfseries, text=mPink!85!black,   anchor=north] at ( 1.0, -0.08) {$\nu'$};
\node[font=\small\bfseries, text=mLPink!85!black,  anchor=north] at ( 6.0, -0.08) {$\nu_t$};
\node[font=\small\bfseries, text=mPurple!85!black, anchor=north] at (11.0, -0.08) {$\nu$};

\draw[mPink!68,          ->, line width=1.3pt] (0.90, 0) -- (4.10, 0);
\draw[mLPink!55!mPurple, ->, line width=1.3pt] (5.90, 0) -- (9.10, 0);

\node[font=\scriptsize, text=mDark!72, anchor=north, below=3pt]
  at (5.1, -0.6) {$\partial_t\nu_t+\nabla\!\cdot(\nu_t v_t)=0$};

\draw[mSlate!82, dashed, line width=1.4pt, ->]
  (-0.05, 2.30) -- (-0.05, 0.90)
  node[midway, left=4pt, font=\footnotesize\bfseries, text=mSlate!85!mDark]
    {$\pi'\!=\!\pi_0$};

\draw[mSlate!72, dashed, line width=1.2pt, ->]
  (4.95, 2.30) -- (4.95, 0.90)
  node[midway, right=4pt, font=\footnotesize\bfseries, text=mSlate!80!mDark]
    {$\pi_t$};

\draw[mSlate!90, dashed, line width=1.55pt, ->]
  (9.95, 2.30) -- (9.95, 0.90)
  node[midway, right=4pt, font=\footnotesize\bfseries, text=mSlate!90!mDark]
    {$\pi^\star$};

\node[fill=mDark!5, rounded corners=6pt,
      minimum width=2.9cm, minimum height=0.95cm]
  at ($(2.5,1.60)+(0.05,-0.06)$) {};
\node[fill=mCard, draw=mSlate!55, rounded corners=5pt,
      line width=0.8pt, minimum width=2.9cm, minimum height=0.95cm,
      font=\footnotesize, align=center, text=mDark]
  at (2.5, 1.60)
  {$\dot\pi_t = \pi_t\bigl(a_t(X)+b_t(Y)\bigr)$\\[2pt]
   \textit{tilting dynamics}};

\node[fill=mDark!5, rounded corners=6pt,
      minimum width=2.9cm, minimum height=0.95cm]
  at ($(7.5,1.60)+(0.05,-0.06)$) {};
\node[fill=mCard, draw=mLBlue!60, rounded corners=5pt,
      line width=0.8pt, minimum width=2.9cm, minimum height=0.95cm,
      font=\footnotesize, align=center, text=mDark]
  at (7.5, 1.60)
  {$\beta_s^t(z)=\mathbb{E}_{\pi_t}[\dot I_s\mid I_s=z]$\\[2pt]
   \textit{CFM field at step $t$}};



\draw[mDark!55, <->, line width=0.9pt]
  (-0.3, -1.32) -- (10.3, -1.32);
\node[font=\scriptsize, text=mDark!80, anchor=north] at (-0.3,  -1.42) {$t{=}0$};
\node[font=\scriptsize, text=mDark!62, anchor=north] at ( 5.0,  -1.42) {$t{=}t_k$};
\node[font=\scriptsize, text=mDark!80, anchor=north] at (10.3,  -1.42) {$t{=}1$};
\node[font=\footnotesize, text=mDark!80, below=4pt]
  at (5.0, -1.62) {$t\in[0,1]$: tilting-path time};


\draw[mBorder, line width=0.5pt] (-1.2, -2.30) -- (10.8, -2.30);

\node[font=\small\bfseries, text=mDark!62, anchor=west]
  at (-1.4, -2.62) {Legend:};

\draw[mBlue!88, -{Stealth[length=4pt,width=3pt]}, line width=1.1pt]
  (0.40,-2.62) -- (1.00,-2.62);
\node[font=\footnotesize, text=mDark!72, anchor=west] at (1.08,-2.62)
  {source path $\mu_t$};

\draw[mPink!88, -{Stealth[length=4pt,width=3pt]}, line width=1.1pt]
  (3.10,-2.62) -- (3.70,-2.62);
\node[font=\footnotesize, text=mDark!72, anchor=west] at (3.78,-2.62)
  {target path $\nu_t$};

\draw[mSlate!88, dashed, -{Stealth[length=4pt,width=3pt]}, line width=1.1pt]
  (5.80,-2.62) -- (6.40,-2.62);
\node[font=\footnotesize, text=mDark!72, anchor=west] at (6.48,-2.62)
  {EOT coupling $\pi_t$};

\draw[mDark!36, -{Stealth[length=4pt,width=3pt]}, line width=1.0pt]
  (8.80,-2.62) -- (9.40,-2.62);
\node[font=\footnotesize, text=mDark!72, anchor=west] at (9.48,-2.62)
  {continuity eq.};

\fill[mBlue!80]   (0.70,-2.96) circle (2.1pt);
\node[font=\footnotesize, text=mDark!72, anchor=west] at (0.92,-2.96)
  {$\mu'$-samples};

\fill[mPink!80]   (3.40,-2.96) circle (2.1pt);
\node[font=\footnotesize, text=mDark!72, anchor=west] at (3.78,-2.96)
  {$\nu'$-samples};

\fill[mTeal!82]   (6.00,-2.96) circle (2.1pt);
\node[font=\footnotesize, text=mDark!72, anchor=west] at (6.48,-2.96)
  {$\mu$-samples};

\fill[mPurple!82] (9.10,-2.96) circle (2.1pt);
\node[font=\footnotesize, text=mDark!72, anchor=west] at (9.48,-2.96)
  {$\nu$-samples};

\end{tikzpicture}

%% file: misc/proof_graph.tex

\definecolor{pgBlue}  {RGB}{52,108,222}
\definecolor{pgYellow}  {RGB}{186,136,26}
\definecolor{pgPink}  {RGB}{205,65,100}
\definecolor{pgTeal}  {RGB}{0,168,152}
\definecolor{pgPurple}{RGB}{116,76,192}
\definecolor{pgLBlue} {RGB}{132,182,242}
\definecolor{pgSlate} {RGB}{96,108,124}
\definecolor{pgDark}  {RGB}{40,44,56}
\definecolor{pgBg}    {RGB}{247,250,254}
\definecolor{pgBorder}{RGB}{208,216,232}

\tikzset{
  pgn/.style={
    draw, rounded corners=5pt,
    minimum width=2.78cm, minimum height=0.82cm,
    inner sep=4pt, align=center, fill=white,
    line width=0.65pt, text=pgDark,
  },
  pgAssump/.style={pgn,
    draw=pgBlue, fill=pgBlue!11, line width=1.2pt},
  pgLemma/.style ={pgn, draw=pgSlate!78, fill=pgSlate!7},
  pgProp/.style  ={pgn, draw=pgTeal!88,  fill=pgTeal!9},
  pgCor/.style   ={pgn, draw=pgPurple!82,fill=pgPurple!9},
  pgThm/.style   ={pgn,
    draw=pgYellow!67, fill=pgYellow!13, line width=1.65pt,
    minimum width=3.65cm, minimum height=0.96cm},
  pgArr/.style={
    -{Stealth[length=3.8pt,width=3.0pt]},
    pgSlate!58, line width=0.48pt},
  pgArrKey/.style={
    -{Stealth[length=5.2pt,width=4.2pt]},
    pgPink!88!pgDark, line width=1.15pt},
}

\begin{tikzpicture}[>=Stealth, node distance=0pt]

\begin{scope}[on background layer]
  \fill[pgBg, rounded corners=13pt]
    (-0.55,-1.15) rectangle (15.15,17.70);
  \draw[pgBorder, rounded corners=13pt, line width=0.62pt]
    (-0.55,-1.15) rectangle (15.15,17.70);
\end{scope}

\draw[pgBorder, line width=0.5pt] (-0.4,16.95) -- (15.0,16.95);

\node[pgAssump, minimum width=1.55cm, minimum height=0.40cm,
     inner sep=2pt, font=\scriptsize] at (1.4, 17.28) {Assumption};
\node[pgLemma,  minimum width=1.35cm, minimum height=0.40cm,
     inner sep=2pt, font=\scriptsize] at (3.55,17.28) {Lemma};
\node[pgProp,   minimum width=1.55cm, minimum height=0.40cm,
     inner sep=2pt, font=\scriptsize] at (5.80,17.28) {Proposition};
\node[pgCor,    minimum width=1.45cm, minimum height=0.40cm,
     inner sep=2pt, font=\scriptsize] at (8.05,17.28) {Corollary};
\node[pgThm,    minimum width=1.35cm, minimum height=0.40cm,
     inner sep=2pt, font=\scriptsize, line width=0.9pt] at (10.2,17.28) {Theorem};
\draw[pgArr, line width=0.68pt] (11.85,17.22) -- (12.70,17.22)
  node[right, font=\scriptsize, text=pgDark!65, inner sep=2pt] {depends on};
\draw[pgArrKey, line width=0.95pt] (11.85,17.38) -- (12.70,17.38)
  node[right, font=\scriptsize, text=pgDark!65, inner sep=2pt] {proves};

\node[pgAssump] (A1) at (1.65,15.0) {%
  \hyperref[ass:compact_X]{\footnotesize\bfseries Ass.~1}\\[-2pt]%
  {\footnotesize Compact $\mathcal{X}$}};
\node[pgAssump] (A2) at (5.10,15.0) {%
  \hyperref[ass:coupling_regularity]{\footnotesize\bfseries Ass.~2}\\[-2pt]%
  {\footnotesize Lip.\ coupling $\pi'$}};
\node[pgAssump] (A3) at (9.20,15.0) {%
  \hyperref[ass:velocity_regularity]{\footnotesize\bfseries Ass.~3}\\[-2pt]%
  {\footnotesize Reg.\ velocity fields}};
\node[pgAssump] (A4) at (13.20,15.0) {%
  \hyperref[ass:initial_regularity]{\footnotesize\bfseries Ass.~4}\\[-2pt]%
  {\footnotesize Reg.\ initial densities}};

\begin{scope}[on background layer]
  \fill[pgLBlue!22, rounded corners=8pt]
    (0.40,11.72) rectangle (14.05,13.88);
  \node[font=\scriptsize\bfseries, text=pgBlue!78, anchor=west]
    at (0.62,13.63)
    {App.~\ref{app:reg_marginal_flows}\quad Marginal Flow Regularity};
\end{scope}

\node[pgLemma] (B1) at (3.0,12.72) {%
  \hyperref[lem:density_bounds]{\footnotesize Lem.~\ref*{lem:density_bounds}}\\[-2pt]%
  {\footnotesize Non-degeneracy}};
\node[pgLemma] (B2) at (7.0,12.72) {%
  \hyperref[lem:temporal_marginal_regularity]{\footnotesize Lem.~\ref*{lem:temporal_marginal_regularity}}\\[-2pt]%
  {\footnotesize Temporal regularity}};
\node[pgLemma] (B3) at (11.0,12.72) {%
  \hyperref[lem:score_time_lip]{\footnotesize Lem.~\ref*{lem:score_time_lip}}\\[-2pt]%
  {\footnotesize Score Lipschitz}};

\begin{scope}[on background layer]
  \fill[pgTeal!11, rounded corners=8pt]
    (0.40,9.42) rectangle (14.05,11.58);
  \node[font=\scriptsize\bfseries, text=pgTeal!82, anchor=west]
    at (0.62,11.33)
    {App.~\ref{app:analy_titlting_rate}.1\quad Tilting Rate Bounds};
\end{scope}

\node[pgLemma] (E1) at (3.00,10.42) {%
  \hyperref[lem:phi_psi_regularity]{\footnotesize Lem.~\ref*{lem:phi_psi_regularity}}\\[-2pt]%
  {\footnotesize $\varphi,\psi$ regularity}};
\node[pgProp]  (E2) at (7.50,10.42) {%
  \hyperref[prop:pot_Lip_time_Linfty]{\footnotesize Prop.~\ref*{prop:pot_Lip_time_Linfty}}\\[-2pt]%
  {\footnotesize Potential Lip-$t$}};
\node[pgProp]  (E3) at (12.00,10.42) {%
  \hyperref[prop:ab_uniform_Linfty]{\footnotesize Prop.~\ref*{prop:ab_uniform_Linfty}}\\[-2pt]%
  {\footnotesize $\|F_t\|_\infty \le C_{ab}^\infty$}};

\begin{scope}[on background layer]
  \fill[pgPurple!9, rounded corners=8pt]
    (0.40,6.12) rectangle (14.05,9.28);
  \node[font=\scriptsize\bfseries, text=pgPurple!78, anchor=west]
    at (0.62,9.03)
    {App.~\ref{app:analy_titlting_rate}.2\quad Spatiotemporal Stability};
\end{scope}

\node[pgLemma] (E4) at (2.00,8.35) {%
  \hyperref[lem:doeblin_from_bounds]{\footnotesize Lem.~\ref*{lem:doeblin_from_bounds}}\\[-2pt]%
  {\footnotesize Doeblin minor.}};
\node[pgLemma] (E5) at (5.25,8.35) {%
  \hyperref[lem:Doeblin_centered_Linfty]{\footnotesize Lem.~\ref*{lem:Doeblin_centered_Linfty}}\\[-2pt]%
  {\footnotesize Doeblin centered}};
\node[pgLemma] (E6) at (8.50,8.35) {%
  \hyperref[lem:operator_stability_Linfty]{\footnotesize Lem.~\ref*{lem:operator_stability_Linfty}}\\[-2pt]%
  {\footnotesize Operator stab.}};
\node[pgLemma] (E7) at (12.50,8.35) {%
  \hyperref[lem:forcing_regularity]{\footnotesize Lem.~\ref*{lem:forcing_regularity}}\\[-2pt]%
  {\footnotesize Forcing reg.}};

\node[pgProp]  (E8) at (6.30,7.00) {%
  \hyperref[prop:ab_stability]{\footnotesize Prop.~\ref*{prop:ab_stability}}\\[-2pt]%
  {\footnotesize $\mathrm{Lip}_t(a_t,b_t) \le L_{ab}$}};
\node[pgCor]   (E9) at (11.20,7.00) {%
  \hyperref[cor:F_regularity]{\footnotesize Cor.~\ref*{cor:F_regularity}}\\[-2pt]%
  {\footnotesize $F_t$ joint regularity}};

\begin{scope}[on background layer]
  \fill[pgPink!10, rounded corners=8pt]
    (0.40,2.60) rectangle (14.05,5.88);
  \node[font=\scriptsize\bfseries, text=pgPink!82, anchor=west]
    at (0.62,5.63)
    {App.~\ref{app:error_analysis}\quad Global Error Analysis};
\end{scope}

\node[pgLemma] (F1) at (3.80,4.92) {%
  \hyperref[lem:local_error_tv]{\footnotesize Lem.~\ref*{lem:local_error_tv}}\\[-2pt]%
  {\footnotesize Local error $C_1\delta^2$}};
\node[pgLemma] (F2) at (9.10,4.92) {%
  \hyperref[lem:prop_error_tv]{\footnotesize Lem.~\ref*{lem:prop_error_tv}}\\[-2pt]%
  {\footnotesize Propagation error $C_2$}};
\node[pgCor]   (F3) at (5.50,3.60) {%
  \hyperref[cor:error_recursion]{\footnotesize Cor.~\ref*{cor:error_recursion}}\\[-2pt]%
  {\footnotesize Error recursion}};
\node[pgLemma] (F4) at (11.20,3.60) {%
  \hyperref[lem:beta_environment_regularity]{\footnotesize Lem.~\ref*{lem:beta_environment_regularity}}\\[-2pt]%
  {\footnotesize $\beta$-env.\ regularity}};

\begin{scope}[on background layer]
  \fill[pgYellow!10, rounded corners=8pt]
    (0.40,0.36) rectangle (14.05,2.26);
\end{scope}

\node[pgThm] (T1) at (4.30,1.28) {%
  \hyperref[thm:global_tv_convergence]{\footnotesize\bfseries Thm.~\ref*{thm:global_tv_convergence}}\\[-2pt]%
  {\footnotesize $\|\pi_t-\tilde\pi_t\|_{\mathrm{TV}}\le C_{\mathrm{global}}\,\delta$}};
\node[pgThm] (T2) at (11.30,1.28) {%
  \hyperref[thm:main_convergence]{\footnotesize\bfseries Thm.~\ref*{thm:main_convergence}}\\[-2pt]%
  {\footnotesize $\|\beta_s^t-\tilde\beta_s^t\|_\infty\le C_\beta(s)\,\delta$}};

\draw[pgArr] (A3.south) .. controls +(0,-0.3) and +(0,0.3)    .. (B1.north);
\draw[pgArr] (A4.south) .. controls +(0,-0.4) and +(0.8,0.4)  .. (B1.north);
\draw[pgArr] (A3.south) .. controls +(0,-0.4) and +(0.4,0.4)  .. (B2.north);
\draw[pgArr] (A4.south) .. controls +(0,-0.3) and +(0,0.3)    .. (B2.north);
\draw[pgArr] (A3.south) .. controls +(0,-0.5) and +(-0.6,0.5) .. (B3.north);
\draw[pgArr] (A4.south) -- (B3.north);
\draw[pgArr] (B1.east)  -- (B2.west);
\draw[pgArr] (B2.east)  -- (B3.west);

\draw[pgArr] (A1.south) .. controls +(0,-2.2) and +(-1.8,1.5) .. (E1.north);
\draw[pgArr] (A2.south) .. controls +(0,-1.8) and +(-0.8,1.2) .. (E1.north);
\draw[pgArr] (B1.south) -- (E1.north);
\draw[pgArr] (B2.south) .. controls +(0,-0.5) and +(0.8,0.5)  .. (E1.north);
\draw[pgArr] (B3.south) .. controls +(0,-0.7) and +(2.5,0.5)  .. (E1.north);
\draw[pgArr] (E1.east)  -- (E2.west);
\draw[pgArr] (E2.east)  -- (E3.west);
\draw[pgArr] (B2.south) .. controls +(0,-0.9) and +(0,0.6)    .. (E2.north);
\draw[pgArr] (E1.east)  .. controls +(2.5,0.2) and +(-3.0,0.2) .. (E3.west);

\draw[pgArr] (E3.south) .. controls +(0,-0.9) and +(10.0,-0.5) .. (E4.north);
\draw[pgArr] (E4.east)  -- (E5.west);
\draw[pgArr] (E5.east)  -- (E6.west);
\draw[pgArr] (B3.south) .. controls +(0,-1.8) and +(0,1.2)    .. (E7.north);
\draw[pgArr] (E6.south) .. controls +(0,-0.3) and +(-0.5,0.3) .. (E8.north);
\draw[pgArr] (E7.south) .. controls +(0,-0.3) and +(1.2,0.3)  .. (E8.north);
\draw[pgArr] (E8.east)  -- (E9.west);

\draw[pgArr] (E9.south) .. controls +(0,-0.6) and +(1.0,0.5)  .. (F1.north);
\draw[pgArr] (E9.south) .. controls +(0,-0.5) and +(0.5,0.4)  .. (F2.north);
\draw[pgArr] (E3.south) .. controls +(0,-3.2) and +(1.8,1.0)  .. (F2.north);
\draw[pgArr] (F1.south) .. controls +(0,-0.3) and +(-0.5,0.3) .. (F3.north);
\draw[pgArr] (F2.south) .. controls +(0,-0.3) and +(0.9,0.3)  .. (F3.north);
\draw[pgArr] (F3.east)  -- (F4.west);

\draw[pgArrKey] (F3.south) -- (T1.north);
\draw[pgArrKey] (F3.south) .. controls +(0,-0.35) and +(-2.0,0.45) .. (T2.north);
\draw[pgArrKey] (F4.south) -- (T2.north);

\end{tikzpicture}

%% file: appendix.tex
\section{Implementation: Tips \& Tricks}
\label{app:tips}

\subsection{Bidirectional flow models}

To reduce error accumulation from time discretisation, we train two lifted
vector fields jointly in an alternating fashion:
\begin{equation*}
\beta_s^{\mu \to \nu}(x, x_0)
:= \mathbb{E}_{\pi}\!\left[\dot{I}_s \,\big|\, I_s = \tbinom{x}{x_0}\right],
\qquad
\beta_s^{\nu \to \mu}(y, y_0)
:= \mathbb{E}_{\pi}\!\left[\dot{J}_s \,\big|\, J_s = \tbinom{y}{y_0}\right].
\end{equation*}
At each outer iteration we alternate between the two directions, fixing the
frozen component via linear interpolation to keep the source marginal on
support. This prevents one marginal from drifting while the other is being
refined, and empirically stabilises training on harder geometries.

\subsection{OT-CFM pretraining}

Standard CFM pretraining on reference-coupling pairs $(x', y') \sim \pi'$
works reasonably well in simple cases but struggles when the coupling has a
multi-modal structure (e.g.\ split-blob geometries). Switching to
\emph{OT-CFM}~\citep{tong2023improving} for the pretraining phase, which
re-indexes each mini-batch of source--target pairs to minimise intra-batch
quadratic cost, yields straighter interpolation paths and faster convergence
of the pretrained model $\beta_s^0$. Crucially, this reduces the initial
transport error before the adaptation loop begins, and the subsequent
dual-tilting iterations have a better starting point to work from.

\subsection{SIREN architecture for high-frequency transports}

Standard GELU or SELU MLPs suffer from \emph{spectral bias}: they fit
low-frequency components of the target velocity field first and converge
slowly to high-frequency components~\citep{rahaman2019spectral}.
For transport maps with fine spatial structure — such as radial warps,
polar vortices, or polynomial shears — this causes the pretrained CFM model
to produce blurry or inaccurate trajectories, impairing the subsequent
tilting phase.

We mitigate this by replacing the hidden-layer activations with
\emph{SIREN} layers~\citep{sitzmann2020implicit}, i.e.\ $\sin(\omega_0
W x + b)$, using the initialization scheme of the original paper.
The final projection layer remains a plain linear layer (no activation)
with a small-gain Xavier initialisation ($\mathrm{gain}=0.01$) to prevent
divergent trajectories early in training.

In practice we use $\omega_0 = 30$ for both the first and hidden layers
(config flags \texttt{siren\_omega\_0} and \texttt{siren\_hidden\_omega}).
SIREN is enabled via \texttt{use\_siren: True} in the training config and
activates the \texttt{BidirectionalSirenStrategy}, which otherwise has
the same interface as the standard bidirectional strategy.
We recommend enabling SIREN for any geometry with sharp boundaries,
discontinuous maps, or a ground-truth transport that is not affine.

\subsection{Particle mixing and the circularity problem}

When the coupling dataset $\mathcal{D}_{k+1}$ is generated by simulating
the current learned model $\beta^{k+1}_s$, all output pairs $(x, \hat{y})$
lie exactly on the model's manifold, driving the importance weights
$\exp(\delta(a_k(x)+b_k(y)))$ toward 1 and collapsing the dual training
signal (we call this the \emph{circularity problem}).

The fix is a \emph{particle fraction} $\alpha \in (0,1]$: a fixed pool of
$n_\mathrm{part}\!=\!\alpha|\mathcal{D}|$ particle slots is initialised once from $\pi'$ and persists across all outer iterations.
At each step the slots are advanced by exactly one Euler step along the
marginal bridge fields $(u_t, v_t)$, so after $k$ steps each slot has
travelled $k/N$ of the way from $t=0$ to $t=1$.
These off-manifold pairs maintain non-trivial weight variance and restore
the dual signal. The default $\alpha = 0.2$ works well across all
experiments; see Table~\ref{tab:alpha_ablation} for the sensitivity study.
\subsection{Numerical scale factor}

For 2D experiments the source and target distributions are placed far apart
(up to $\pm 10$ in each coordinate), which inflates the quadratic term of
the dual energy~\eqref{eq:energy-ab} and can lead to large gradient
magnitudes. We apply a global \texttt{scale\_factor} of $50$ that rescales
the input coordinates to $[-1, 1]^2$ before they enter the dual networks
and marginal bridges, then undoes the scaling on output.
For image experiments the scale factor is $1$ (pixel values are already
in a bounded range).

\subsection{Conditional bridge pretraining}

When the reference and new marginals are positioned far apart, the bridge
velocity fields $u_t, v_t$ can create crossing paths within a mini-batch,
leading to tangled trajectories. Setting \texttt{use\_conditional\_bridge:
True} conditions the bridge on the initial position $x_0$ (frozen
component), which eliminates crossing-path ambiguity and stabilises bridge
training at no computational overhead.

\begin{table}[ht]
\centering
\caption{\textbf{Positioning of our method relative to existing methods.}
$\checkmark$ = satisfies the criterion; $\times$ = does not.}
\label{tab:positioning}
\resizebox{\linewidth}{!}{
\begin{tabular}{lcccc}
\toprule
Method & Generative sampler & Unknown cost & Transfer to new $(\mu,\nu)$ & EOT/SB \\
\midrule
Sinkhorn \citep{cuturi2013sinkhorn}    & $\times$ & $\times$ & $\times$ & $\checkmark$ \\
OT-CFM \citep{tong2023improving}       & $\checkmark$ & $\times$ & $\times$ & Approx. \\
GENOT \citep{klein2024genot}           & $\checkmark$ & $\times$ & $\times$ & $\checkmark$ \\
DSBM \citep{shi2023diffusion}          & $\checkmark$ & $\times$ & $\times$ & $\checkmark$ \\
LightSB \citep{korotin2024light}       & $\checkmark$ & $\times$ & $\times$ & $\checkmark$ \\
Meta OT \citep{amos2023meta}           & $\times$ & $\times$ & $\checkmark$ & $\times$ \\
CondOT \citep{bunne2022condot}         & $\times$ & $\times$ & $\checkmark$ (supervised) & $\times$ \\
\midrule
\textbf{TACO (ours)}                        & $\checkmark$ & $\checkmark$ & $\checkmark$ & $\checkmark$ \\
\bottomrule
\end{tabular}
}
\end{table}

\section{Experiment Setup}
\label{app:setup}

\subsection{Hyperparameters}
\label{app:hyperparams}

Table~\ref{tab:hyperparams_full} collects the key hyperparameters across all
experiment families. All models use the Adam optimiser~\citep{kingma2014adam} with default momentum
($\beta_1 = 0.9$, $\beta_2 = 0.999$) and fixed learning rates (no
cosine decay).

\begin{table}[ht]
\centering
\caption{Hyperparameters used across experiment families.
  ``Outer iterations'' refers to the number of Euler steps $N$ in
  Algorithm~\ref{alg:OURMETHOD}. The pre-training cost is incurred once and
  amortised across all subsequent tasks.}
\label{tab:hyperparams_full}
\vspace{2pt}
\setlength{\tabcolsep}{5pt}
\begin{tabular}{lcccc}
\toprule
Parameter & 2D / Nonlinear & 2D High-dim & Colour (MNIST) & Single Cell \\
\midrule
Dataset size          & $10^6$  & $5\times10^5$ & 2\,000  & $10^5$  \\
Batch size            & 8\,192  & 4\,096        & 128     & 2\,048  \\
Marginal epochs       & 2\,000  & 3\,000        & 10\,000 & 3\,000  \\
Pre-train epochs      & 5\,000  & 5\,000        & 15\,000 & 8\,000  \\
Dual epochs per iter. & 100     & 100           & 100     & 100     \\
CFM epochs per iter.  & 500     & 500           & 250     & 500     \\
Outer iterations $N$  & 50--100 & 100           & 30--50  & 100     \\
Particle ratio $\alpha$ & 0.20  & 0.25          & 0.20    & 0.20    \\
Noise $\sigma$        & 0.1     & 0.01          & 0.05    & 0.05    \\
Scale factor          & 50      & 50            & 1.0     & 20      \\
Dual LR               & \multicolumn{3}{c}{$1\times10^{-4}$} & $1\times10^{-3}$ \\
CFM LR                & \multicolumn{3}{c}{$2\times10^{-4}$} & $1\times10^{-3}$ \\
Bridge LR             & \multicolumn{3}{c}{$2\times10^{-4}$} & $1\times10^{-3}$ \\
\midrule
SIREN (for nonlinear) & $\checkmark$ & $\times$ & $\times$ & $\times$ \\
$\omega_0$ (SIREN)    & 30     & ---           & ---     & ---     \\
\bottomrule
\end{tabular}
\end{table}

\subsection{Network architectures}
\label{app:arch}

\paragraph{2D / Nonlinear experiments.}
All networks (dual potentials, marginal bridges, CFM flow) are feed-forward
networks with sinusoidal time embedding of dimension 64 projected to the
hidden width before the first layer.
For standard geometries (Simple, Medium, Complex, Moon) we use GELU-MLP
backbones with 2--3 hidden layers of width 32--128, depending on geometry
difficulty.
For nonlinear geometries (Section~\ref{app:nonlinear}), the CFM and bridge
networks are replaced by \texttt{SirenMLP} with 3 hidden layers of width
128 and $\omega_0 = 30$.
The dual potentials always use standard GELU-MLPs with a scalar output.

\paragraph{Image experiments.}
The CFM model is a conditional UNet~\citep{ronneberger2015unet} that receives
a $2C\!\times\!H\!\times\!W$ input formed by concatenating the current state
$x_t$ with the frozen source $x_0$ along the channel axis.
Architecture constants: base channels $= 32$, channel multiplier $(1,2,3,4)$,
1 residual block per resolution level, image size $32\!\times\!32$.
Dual potentials use a ``stop-middle'' UNet variant (global average pooling
at the bottleneck) to map images to scalars.

\paragraph{ODE integration at inference.}
All transport maps are evaluated with 200 Euler steps, consistent with the
\texttt{num\_steps=200} setting used throughout evaluation.

\paragraph{Baselines.}
CFM, OT-CFM, and LightSB are trained from scratch on the new marginals
with access to the true quadratic cost.
DSB is run with its default score-matching objective.
All hyperparameters for baselines are taken from the respective original
papers and tuned to the same number of gradient steps as our pretraining.

\section{Metrics}
\label{app:metrics}

We use the following suite of metrics to evaluate both marginal quality and
coupling quality. All metrics are lower-is-better.

\begin{itemize}[noitemsep,topsep=2pt]
\item \textbf{$\mathrm{SW}_2(\hat\nu)$}: Sliced-Wasserstein-2~\citep{bonneel2015sliced} on the
  \emph{target} marginal — how closely the generated output distribution
  matches the new target $\nu$.
\item \textbf{$\mathrm{SW}_2(\hat\mu)$}: Sliced-Wasserstein-2~\citep{bonneel2015sliced} on the
  \emph{source} marginal — verifies that the source is preserved through
  the transport.
\item \textbf{$\mathrm{MapErr}_\mathrm{fwd}$}: RMSE between the
  ground-truth map $T(x)$ and the model output $\hat{y}$, when $T$ is
  analytically known (Simple, Medium, Complex, all nonlinear experiments).
\item \textbf{Sinkhorn divergence $S_\varepsilon$}~\citep{feydy2019interpolating}: Debiased entropic OT
  cost measuring transport quality; $S_\varepsilon(P,P)=0$.
\item \textbf{MMD} and \textbf{Energy distance}: Kernel- and
  Euclidean-distance-based distributional metrics.
\item \textbf{BW$_2$}: Bures-Wasserstein distance to the ground-truth
  Gaussian coupling (scalability experiments only).
\item \textbf{FID}~\citep{heusel2017gans} and \textbf{KID}~\citep{binkowski2018demystifying}: Fréchet Inception Distance and
  Kernel Inception Distance for image-domain experiments.
\end{itemize}

Baselines marked $\dagger$ have direct access to the true quadratic cost;
TACO does not.

\section{2D Transport Map Visualisations}
\label{app:2d_transport}

For each standard 2D geometry we show three panels:
(a) the reference coupling $\pi'$ used to pre-train the model;
(b) the pretrained model applied to the new marginals (iteration 0,
    before any dual adaptation);
(c) the adapted model after $N$ outer iterations.
Below each row we additionally show the \emph{coupling evolution} — a
collage of the estimated joint distribution $(x, \hat{y}) \sim \tilde\pi_t$
at equally-spaced checkpoints throughout the adaptation run.
The evolution illustrates how the coupling mass progressively migrates from
the reference support to the correct target coupling.

\emph{Background}: KDE density heatmap (viridis).
\emph{Cyan}: source samples; \emph{Orange-red}: transported samples;
\emph{Yellow}: true-target samples.
Thin grey lines are trajectory connections between paired particles.
Figures are individually zoomed to fit the displayed distributions and do
not share the same axis scaling.

\subsection{2D-Simple}
\label{app:2d_simple}

Reference geometry: anticorrelation map $T(x) = -x$; source/target at
$(\pm10, \pm10)$. The reference coupling $\pi'$ is centered at the origin.
After $N=50$ iterations the adapted model recovers the anticorrelation map
with $\mathrm{MapErr}_\mathrm{fwd} = 0.021$.

\begin{figure}[ht]
\centering
\begin{subfigure}[b]{0.32\linewidth}
  \includegraphics[width=\linewidth]{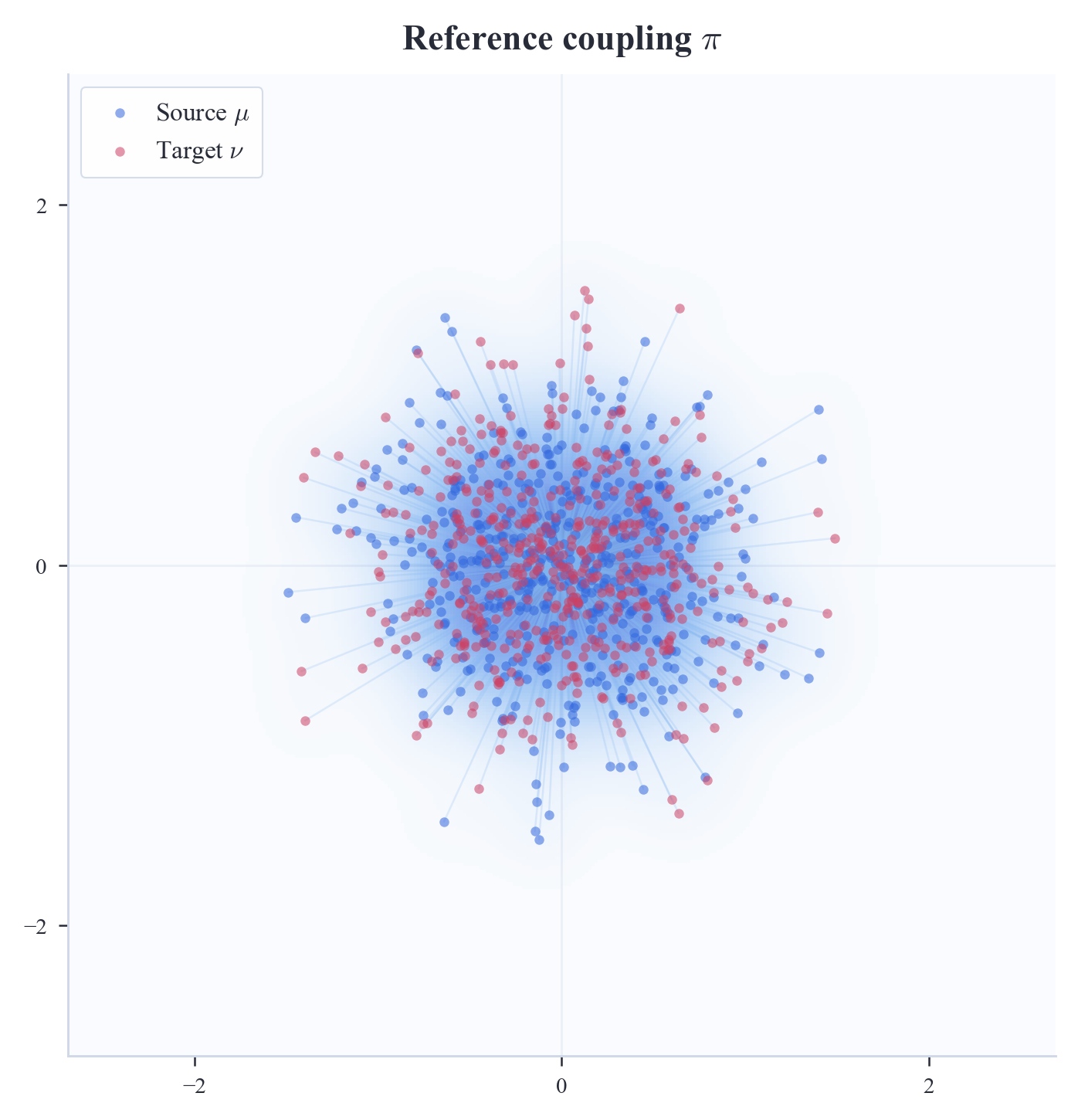}
  \caption{Reference coupling $\pi'$.}
\end{subfigure}\hfill
\begin{subfigure}[b]{0.32\linewidth}
  \includegraphics[width=\linewidth]{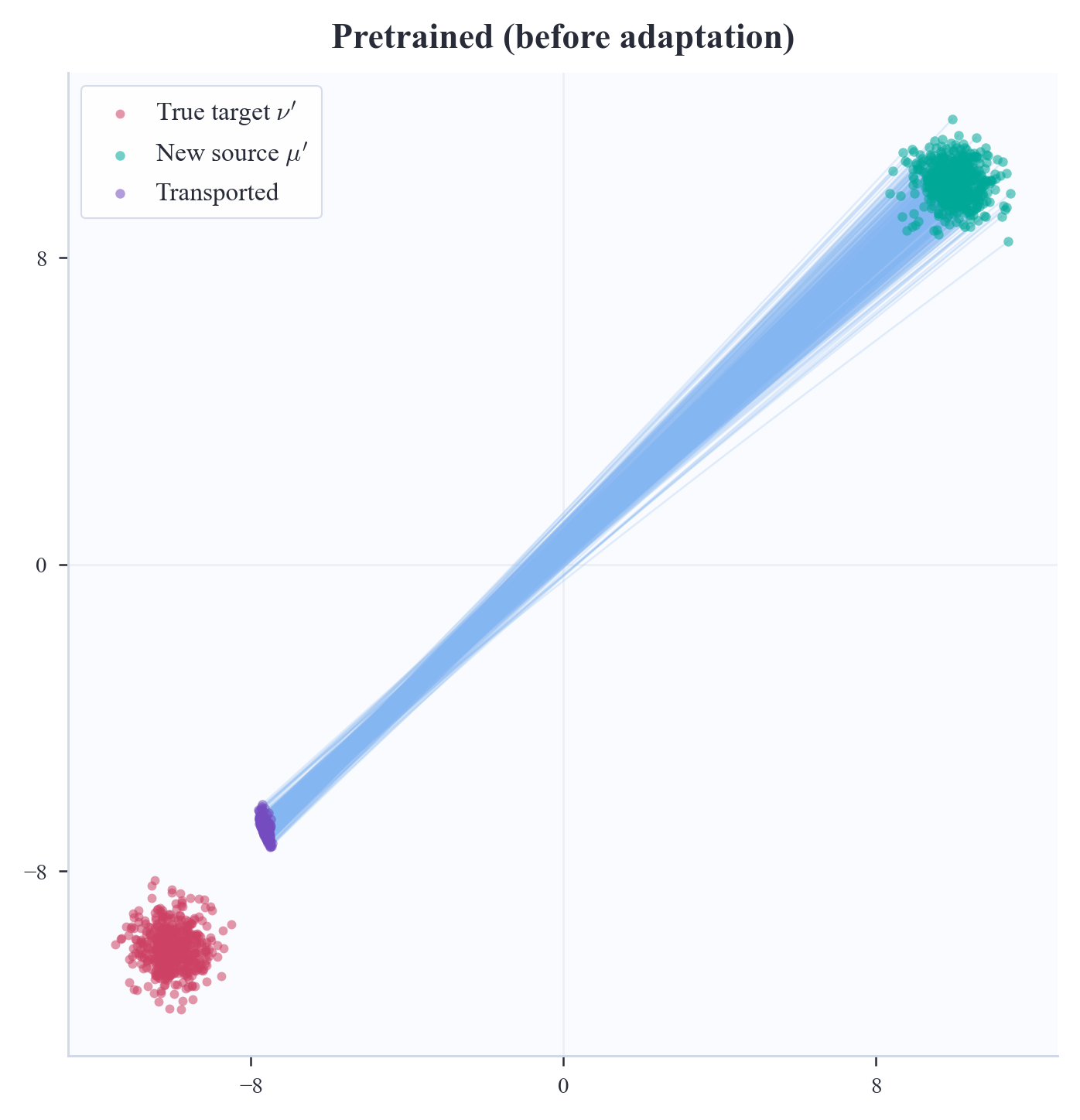}
  \caption{Pretrained ($\mathrm{SW}_2\!=\!14.59$).}
\end{subfigure}\hfill
\begin{subfigure}[b]{0.32\linewidth}
  \includegraphics[width=\linewidth]{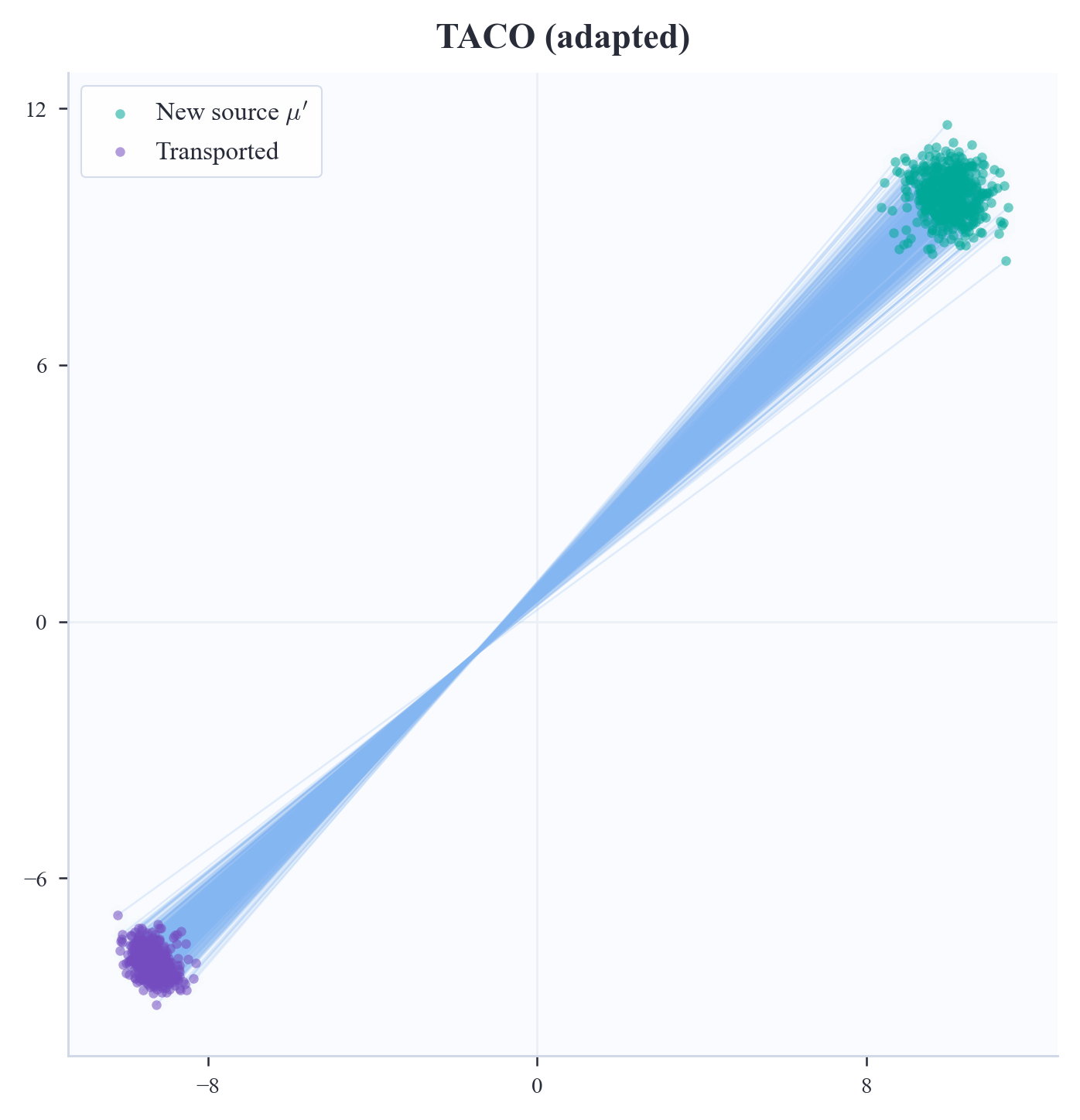}
  \caption{Adapted ($\mathrm{SW}_2\!=\!0.025$).}
\end{subfigure}

\par\vspace{0.4em}

\begin{subfigure}[b]{\linewidth}
  \centering
  \includegraphics[width=\linewidth]{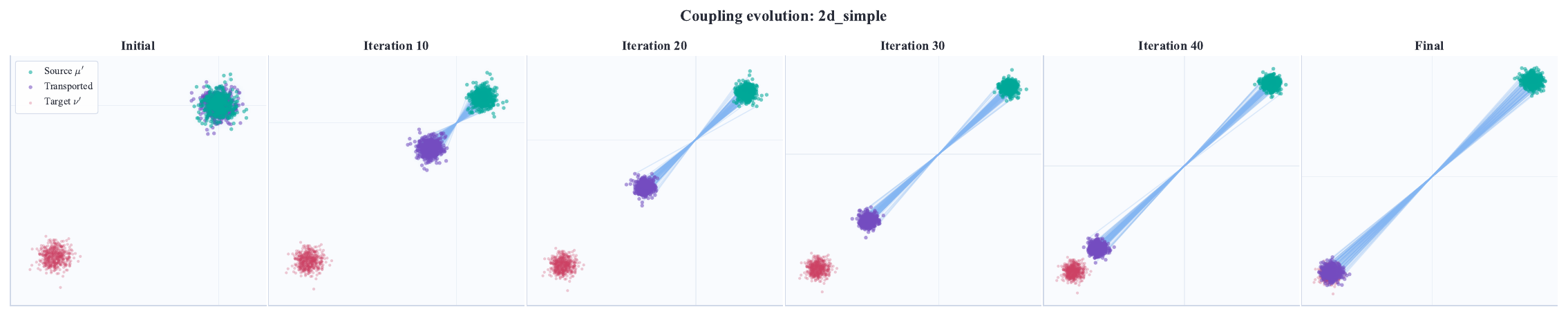}
  \caption{Coupling evolution: joint distribution $(x,\hat{y})$ at
    equally-spaced adaptation checkpoints (left $\to$ right: iteration 0
    through $N$).}
\end{subfigure}

\caption{\textbf{Transport map and coupling evolution on 2D-Simple.}
  The pretrained model applies the reference anticorrelation map at the
  wrong location (near origin instead of $(\pm10,\pm10)$), yielding
  $\mathrm{SW}_2\!=\!14.59$.
  After 50 outer iterations the model correctly recovers $T(x){=}{-}x$
  at the new marginals, reaching $\mathrm{SW}_2\!=\!0.025$ —
  competitive with baselines that observe the ground-truth coupling.
  The coupling evolution (bottom) shows the joint mass smoothly migrating
  from the diagonal (no transport) to the correct anti-diagonal structure.}
\label{fig:transport_evolution_simple_app}
\end{figure}

\subsection{2D-Medium}
\label{app:2d_medium}

Reference geometry: $60^\circ$ rotation, 2-blob source/target.
New marginals: 4-blob square arrangement at $(\pm5,\pm5)$.
The algorithm must infer the rotation law and generalise it to a more
complex marginal structure.

\begin{figure}[ht]
\centering
\begin{subfigure}[b]{0.32\linewidth}
  \includegraphics[width=\linewidth]{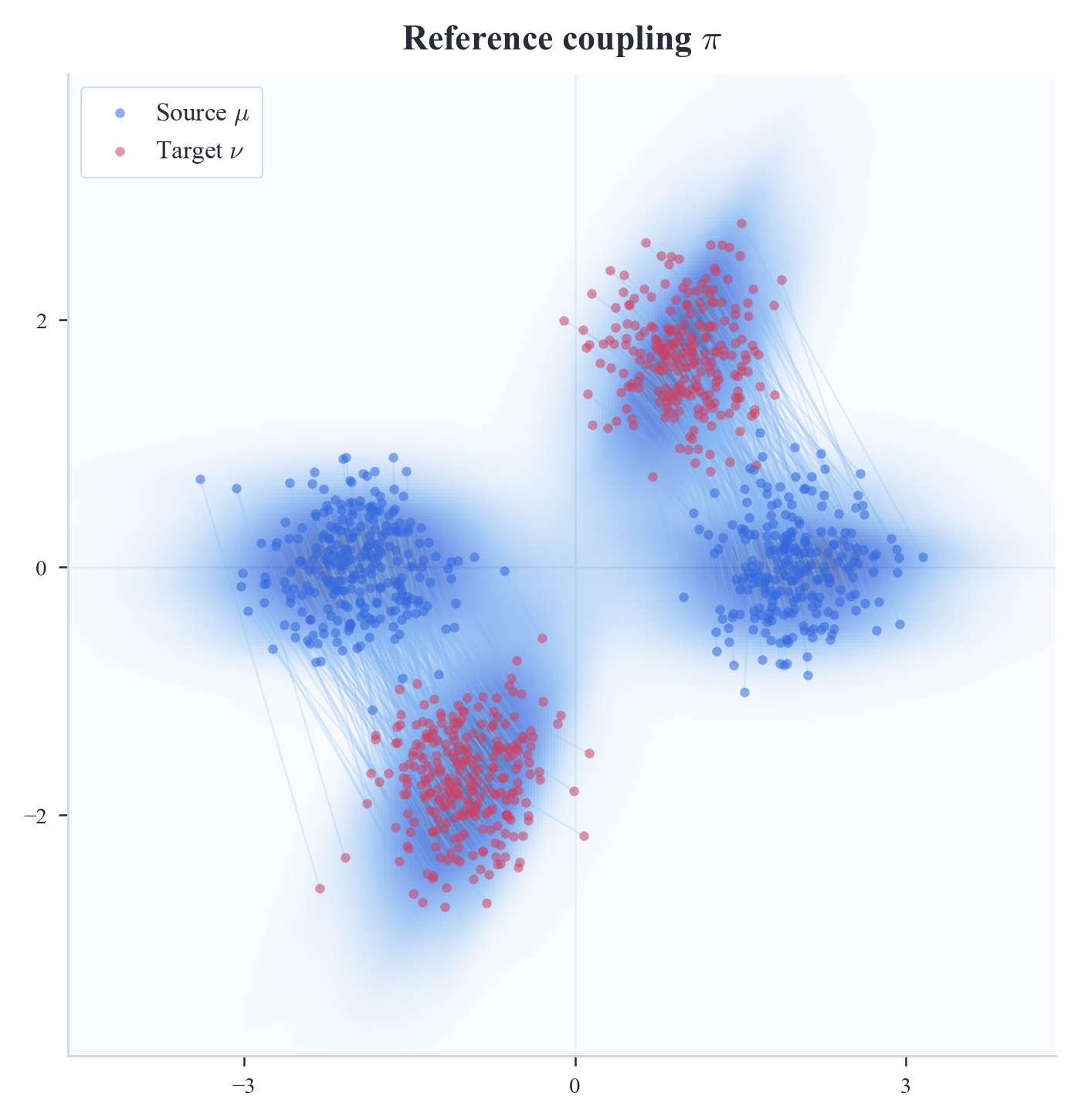}
  \caption{Reference coupling $\pi'$.}
\end{subfigure}\hfill
\begin{subfigure}[b]{0.32\linewidth}
  \includegraphics[width=\linewidth]{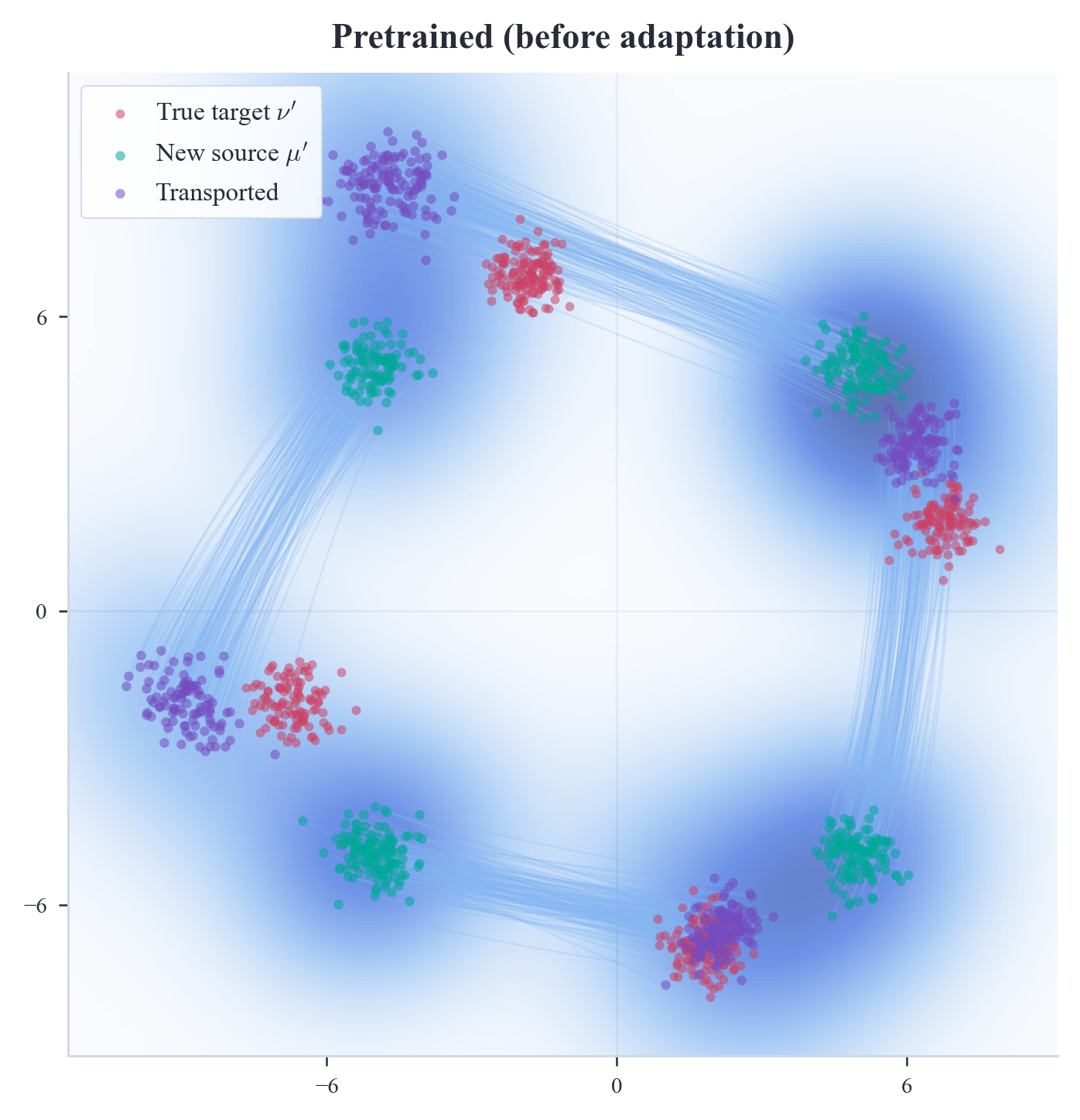}
  \caption{Pretrained ($\mathrm{SW}_2\!\approx\!2.19$).}
\end{subfigure}\hfill
\begin{subfigure}[b]{0.32\linewidth}
  \includegraphics[width=\linewidth]{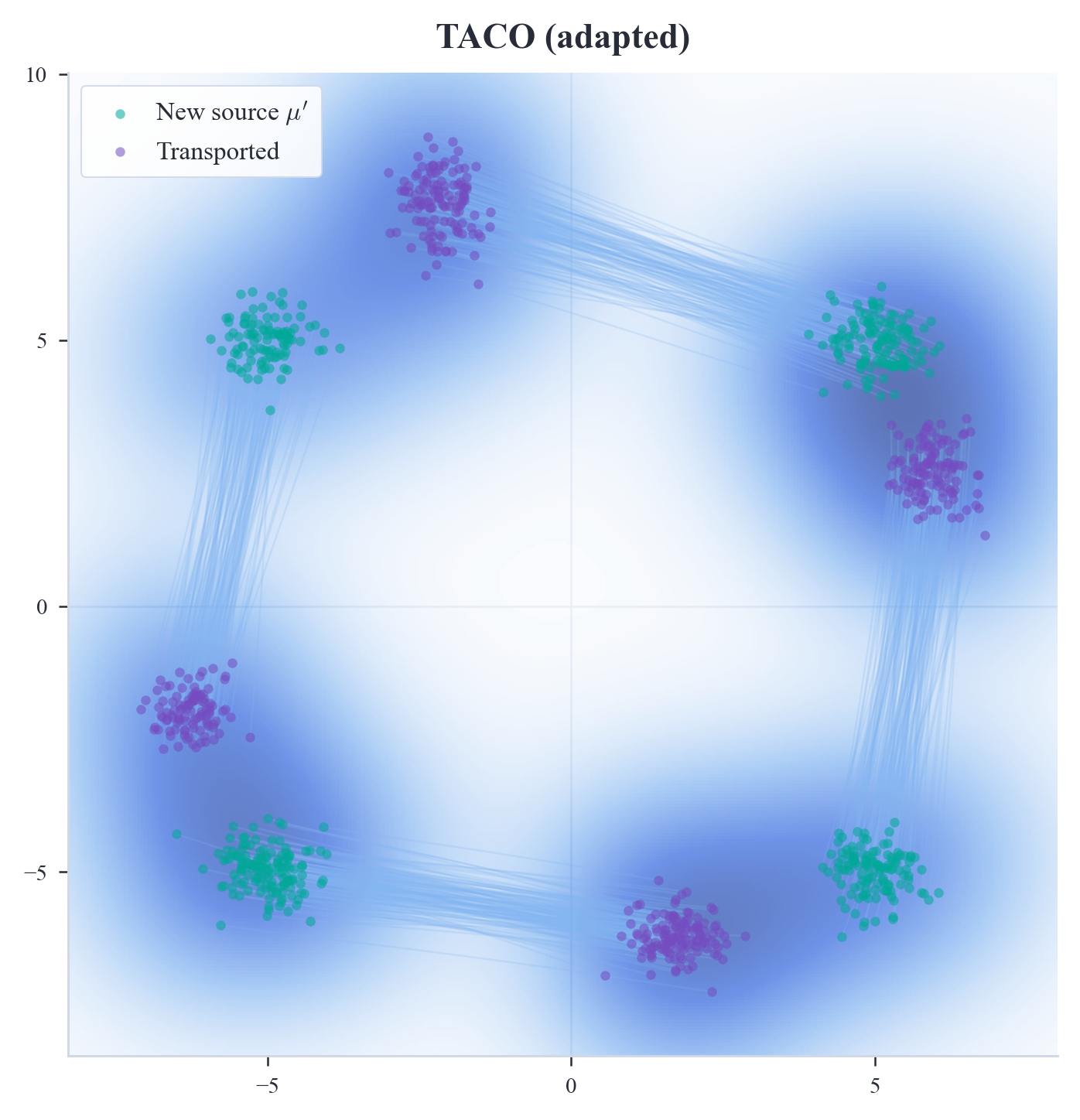}
  \caption{Adapted ($\mathrm{SW}_2\!\approx\!0.756$).}
\end{subfigure}

\par\vspace{0.4em}

\begin{subfigure}[b]{\linewidth}
  \centering
  \includegraphics[width=\linewidth]{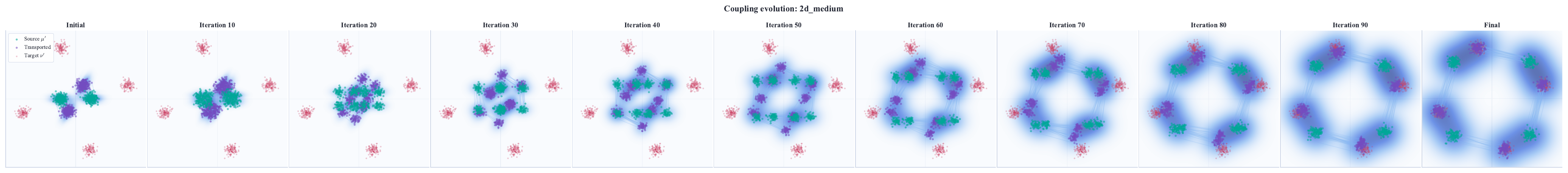}
  \caption{Coupling evolution across adaptation iterations.}
\end{subfigure}

\caption{\textbf{Transport map and coupling evolution on 2D-Medium.}
  The reference coupling encodes a $60^\circ$ rotation on 2-blob Gaussians.
  The new task requires transferring this rotation to a 4-blob square
  arrangement, which the pretrained model handles poorly ($\mathrm{SW}_2\!\approx\!2.19$).
  After 100 outer iterations, the adapted model ($\mathrm{SW}_2\!\approx\!0.756$)
  is competitive with cost-aware baselines (best: OT-CFM$^\dagger\!=\!0.623$).
  The coupling evolution shows the gradual alignment of the four target blobs.}
\label{fig:transport_evolution_medium_app}
\end{figure}

\subsection{2D-Complex}
\label{app:2d_complex}

Reference geometry: $45^\circ$ rotation, 5-blob cross; new marginals
scaled $6.7\times$ relative to the reference.
This tests generalisation under large scale change in addition to rotation.

\begin{figure}[ht]
\centering
\begin{subfigure}[b]{0.32\linewidth}
  \includegraphics[width=\linewidth]{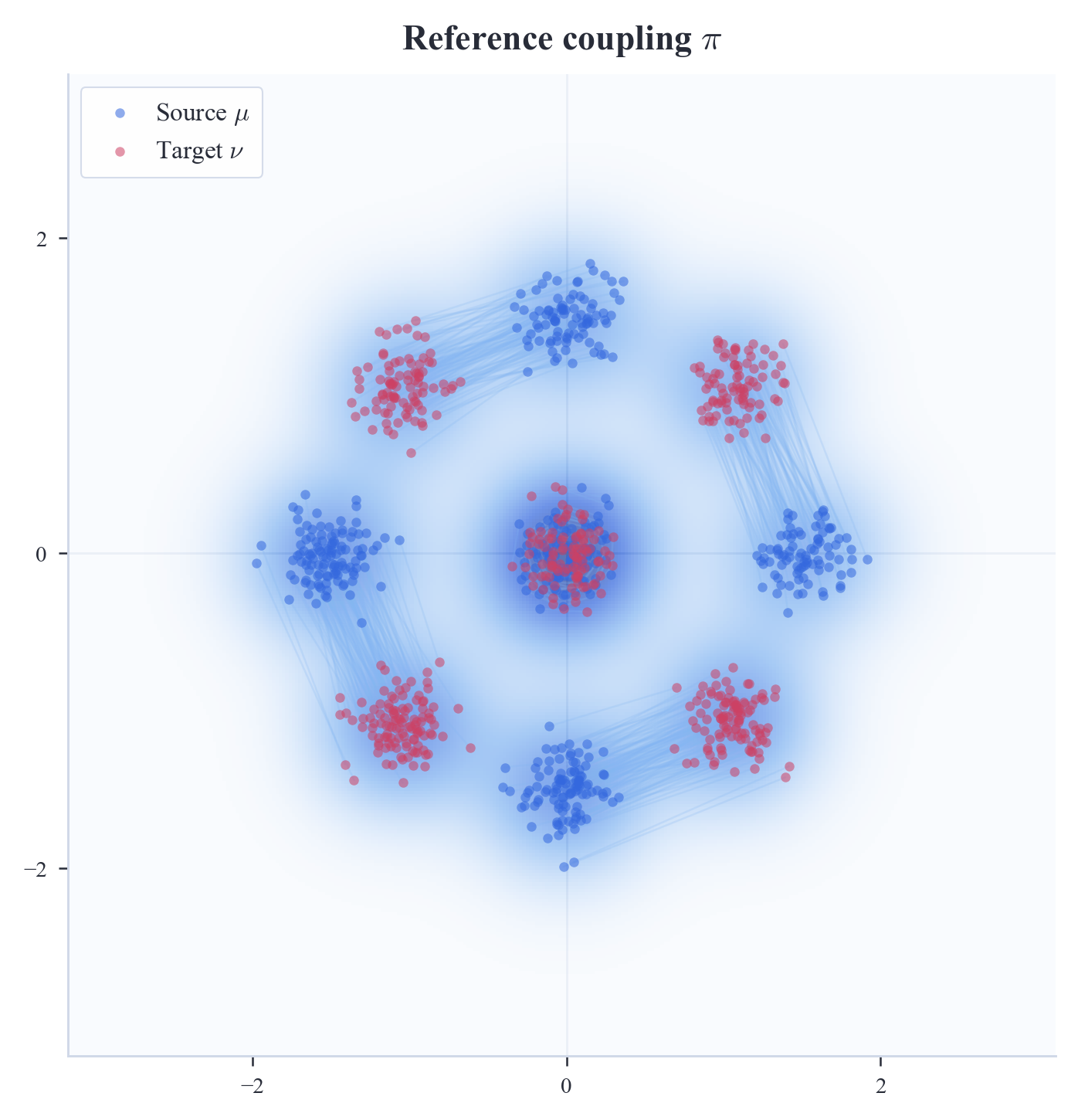}
  \caption{Reference coupling $\pi'$.}
\end{subfigure}\hfill
\begin{subfigure}[b]{0.32\linewidth}
  \includegraphics[width=\linewidth]{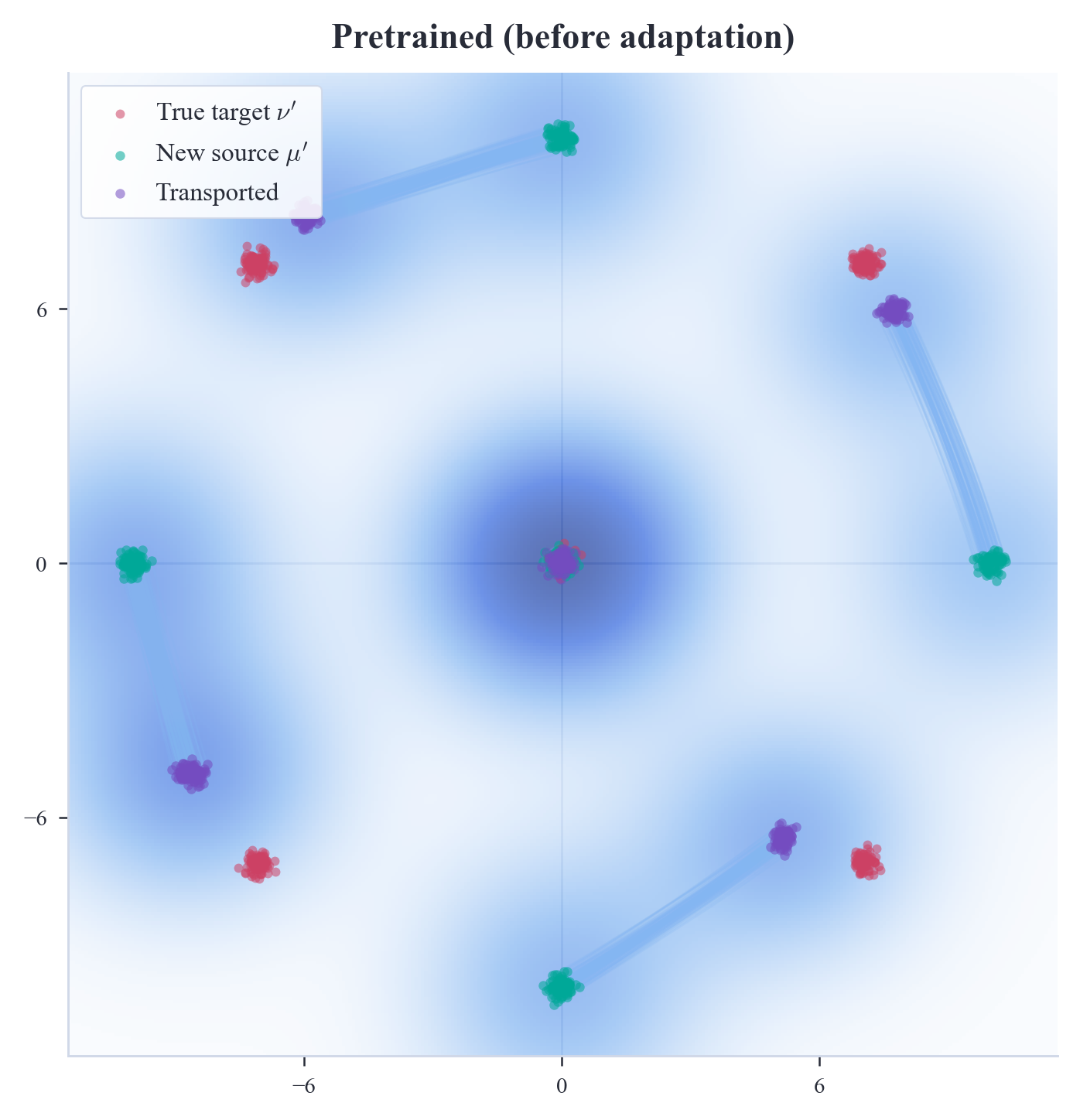}
  \caption{Pretrained ($\mathrm{SW}_2\!=\!2.64$).}
\end{subfigure}\hfill
\begin{subfigure}[b]{0.32\linewidth}
  \includegraphics[width=\linewidth]{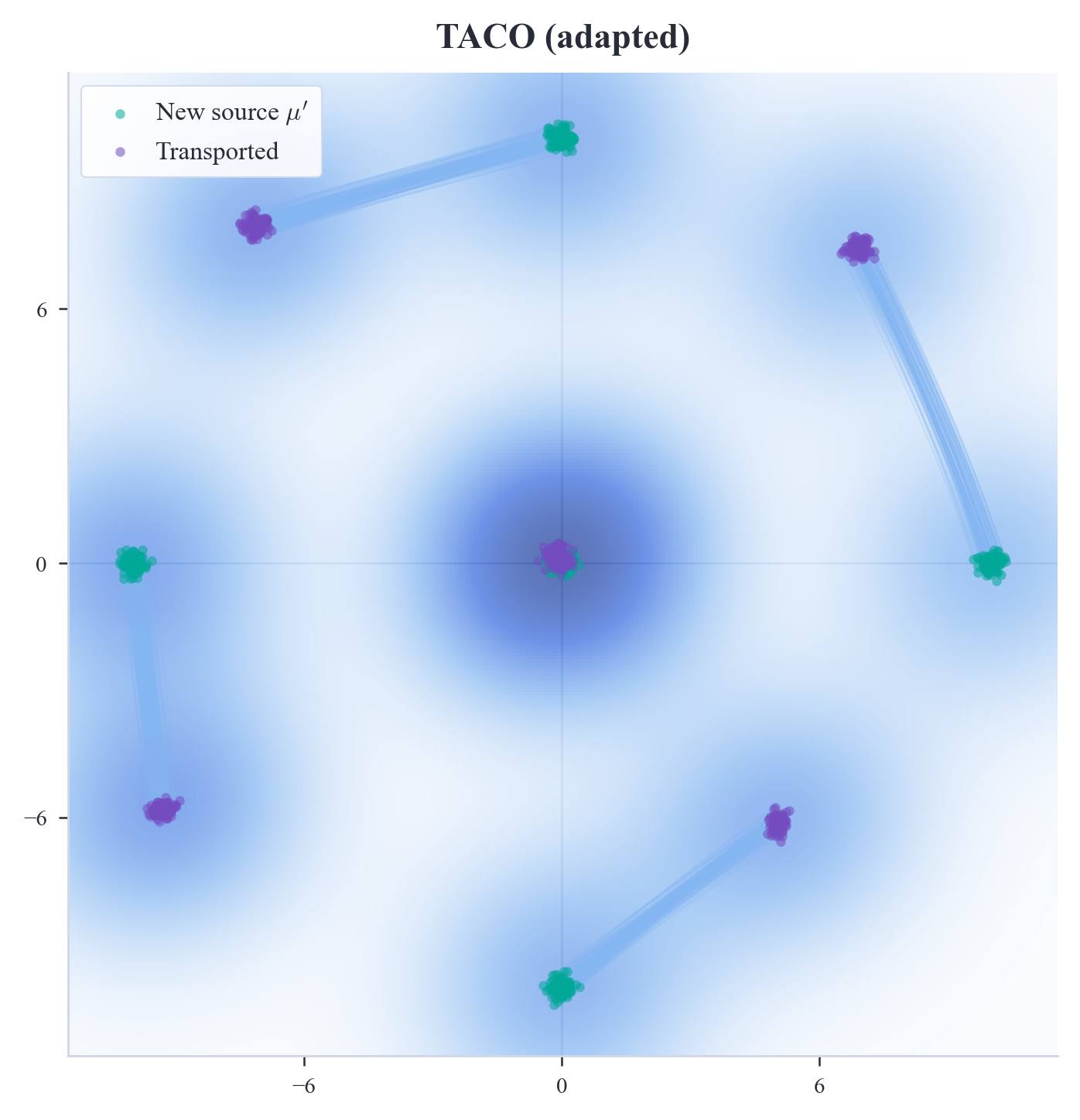}
  \caption{Adapted ($\mathrm{SW}_2\!=\!0.544$, implicit).}
\end{subfigure}

\par\vspace{0.4em}

\begin{subfigure}[b]{\linewidth}
  \centering
  \includegraphics[width=\linewidth]{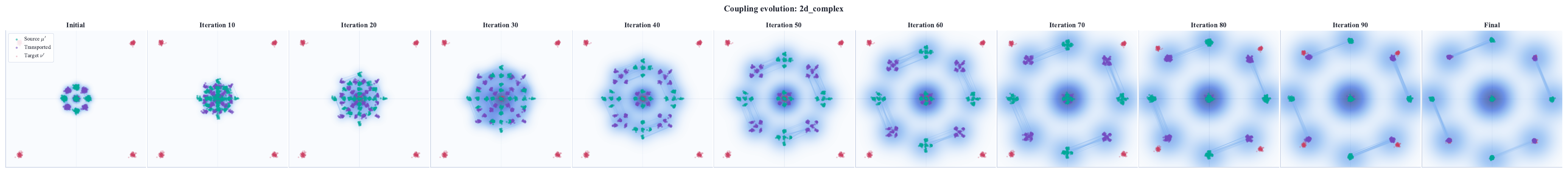}
  \caption{Coupling evolution: the five-blob cross progressively rotates
    and expands to match the new marginal scale.}
\end{subfigure}

\caption{\textbf{Transport map and coupling evolution on 2D-Complex.}
  The source is a five-blob cross that is scaled $6.7\times$ relative to
  the reference, making one-shot reweighting infeasible (the target is
  outside the reference support).
  The implicit-loss variant achieves the best
  $\mathrm{SW}_2\!=\!0.544$, improving $79\%$ over the pretrained baseline
  ($2.64$) and outperforming cost-aware baselines on this geometry.}
\label{fig:transport_evolution_complex_app}
\end{figure}

\subsection{2D-Moon}
\label{app:2d_moon}

Reference geometry: crescent (moon) shapes as source and target; no
known analytical map. This tests the algorithm on a topology where
standard Gaussian-based assumptions break down.





%
\section{Non-linear 2D Experiments}
\label{app:nonlinear}

Beyond the standard geometries, we evaluate on four experiments where the
ground-truth transport map $T$ is non-linear and/or the reference and new
marginals require non-trivial spatial reasoning.
All four experiments use SIREN networks (Section~\ref{app:tips}) and 50
outer iterations.
Results are reported in Table~\ref{tab:nonlinear_results}.

\paragraph{Circle.}
Source $\mu'$: full circle centered at $(8,-1)$, radius 2.
New source $\mu$: full circle at origin, radius 3.
New target $\nu$: upper/lower semi-circles shifted $\pm 5$ in $y$.
The transport splits a single circle into two spatially separated
semi-circles — a topologically non-trivial map.

\paragraph{Radial warp.}
Transport map $T(x) = x \cdot (1 + 0.5\|x\|^2)$, a cubic radial
expansion. Both reference and new marginals are 2-blob Gaussians at
different scales and positions.


\paragraph{Polar twist.}
Transport map in polar coordinates: $r' = r$, $\theta' = \theta + \sin(r)$
- a vortex-like map that preserves radius but twists angle by a nonlinear function of the radius.

\begin{table}[ht]
\centering
\caption{\textbf{Non-linear 2D experiments}: metrics at iteration 0 (Pretrained-only)
  and after 50 outer steps of TACO.
  \textbf{Bold} = best in each column (lower is better for all metrics).
  Note that $\mathrm{SW}_2(\hat\nu)$ can be misleading here: the pretrained
  model sometimes achieves good marginal coverage (low $\mathrm{SW}_2$) while
  the coupling structure remains far from optimal (high $S_\varepsilon$);
  adaptation corrects the coupling at the cost of a slight marginal spread.}
\label{tab:nonlinear_results}
\vspace{2pt}
\setlength{\tabcolsep}{5pt}
\begin{tabular}{lcccccc}
\toprule
 & \multicolumn{2}{c}{$\mathrm{SW}_2(\hat\nu)$ ($\downarrow$)} & \multicolumn{2}{c}{MapErr ($\downarrow$)} & \multicolumn{2}{c}{$S_\varepsilon$ ($\downarrow$)} \\
\cmidrule(lr){2-3}\cmidrule(lr){4-5}\cmidrule(lr){6-7}
Experiment & Pretrained & TACO & Pretrained & TACO & Pretrained & TACO \\
\midrule
Circle       & \textbf{0.235} & 0.588 & 1.040 & \textbf{0.895} & 1.958 & \textbf{0.574} \\
Radial warp  & 0.226 & \textbf{0.204} & \textbf{0.117} & 0.174          & 1.557 & \textbf{0.128} \\
Polar twist  & \textbf{0.054} & 0.108 & 1.604 & \textbf{0.911} & 3.769 & \textbf{0.043} \\
\bottomrule
\end{tabular}
\end{table}

For each non-linear geometry, Figures~\ref{fig:nonlinear_circle}--\ref{fig:nonlinear_polar} show the three-panel transport figure (reference coupling, pretrained, adapted) followed by the coupling evolution across the 50 outer iterations.
A consistent pattern emerges: the pretrained model achieves surprisingly low
$\mathrm{SW}_2(\hat\nu)$ in some cases (polar twist: $0.054$)
because it can spread mass across the correct marginal support as they are not so far from the original supports,
but the Sinkhorn divergence ($3.769$ and $2.274$ respectively) reveals that
the \emph{coupling structure} is far from the correct EOT plan.
After 50 outer iterations, the Sinkhorn divergence drops by $50$--$88\%$
across all four experiments while $\mathrm{SW}_2$ changes modestly,
confirming that the algorithm is correcting the coupling and not
merely re-weighting the marginals.
This illustrates why $\mathrm{SW}_2$ alone is an insufficient metric for
evaluating transport quality.

\begin{figure}[ht]
\centering
\begin{subfigure}[b]{0.32\linewidth}
  \includegraphics[width=\linewidth]{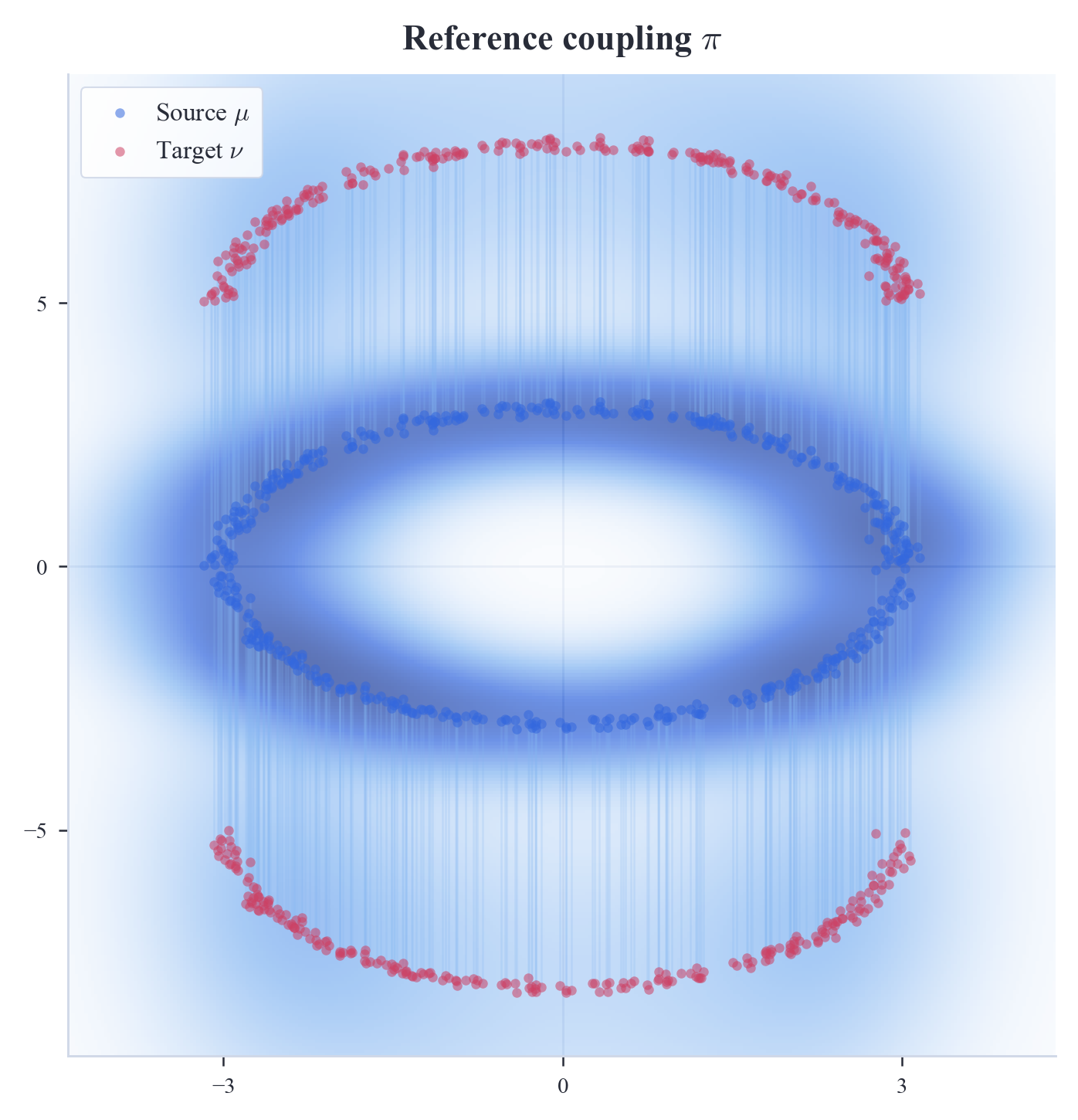}
  \caption{Reference coupling $\pi'$.}
\end{subfigure}\hfill
\begin{subfigure}[b]{0.32\linewidth}
  \includegraphics[width=\linewidth]{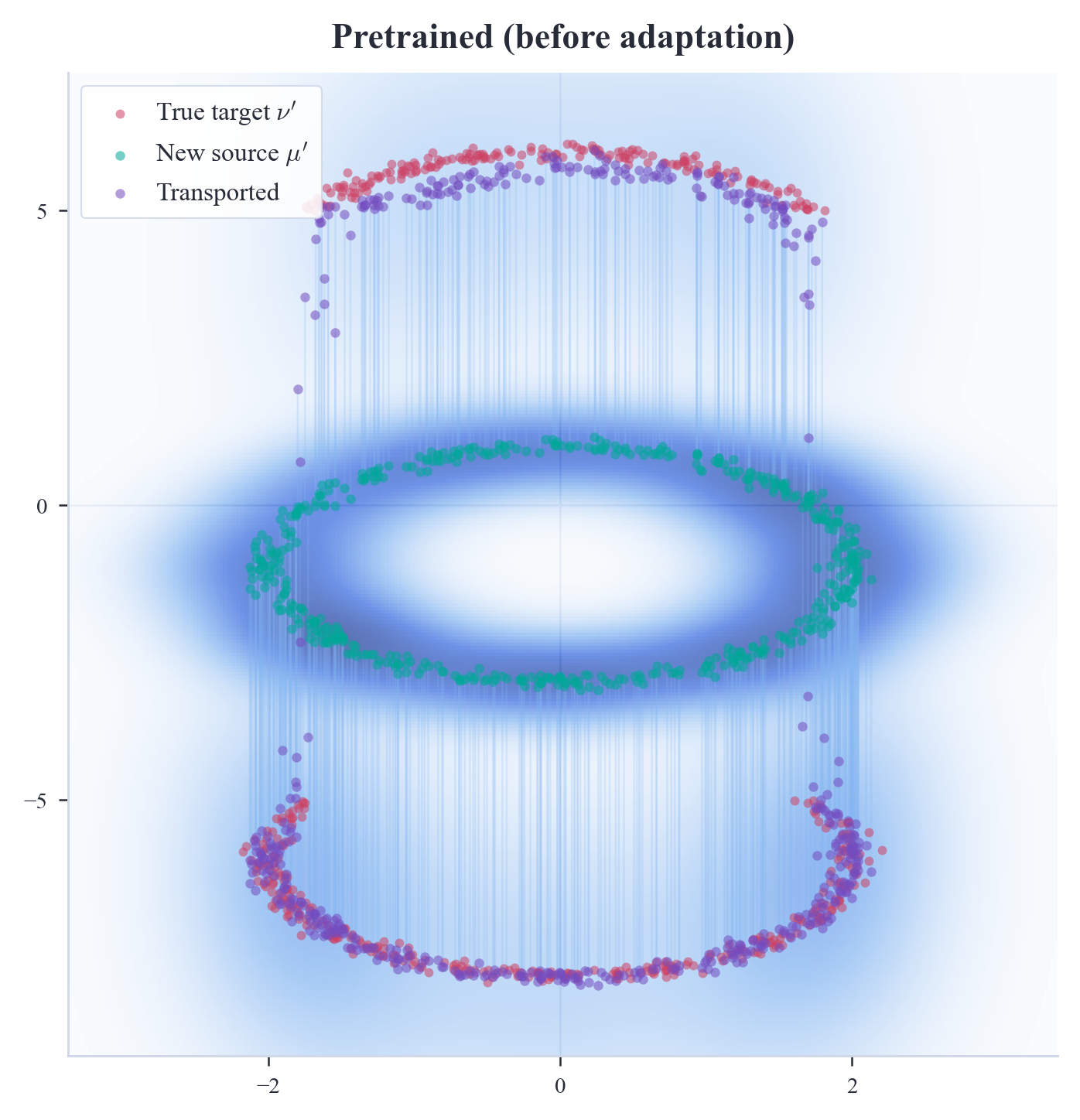}
  \caption{Pretrained ($\mathrm{SW}_2\!=\!0.235$, $S_\varepsilon\!=\!1.958$).}
\end{subfigure}\hfill
\begin{subfigure}[b]{0.32\linewidth}
  \includegraphics[width=\linewidth]{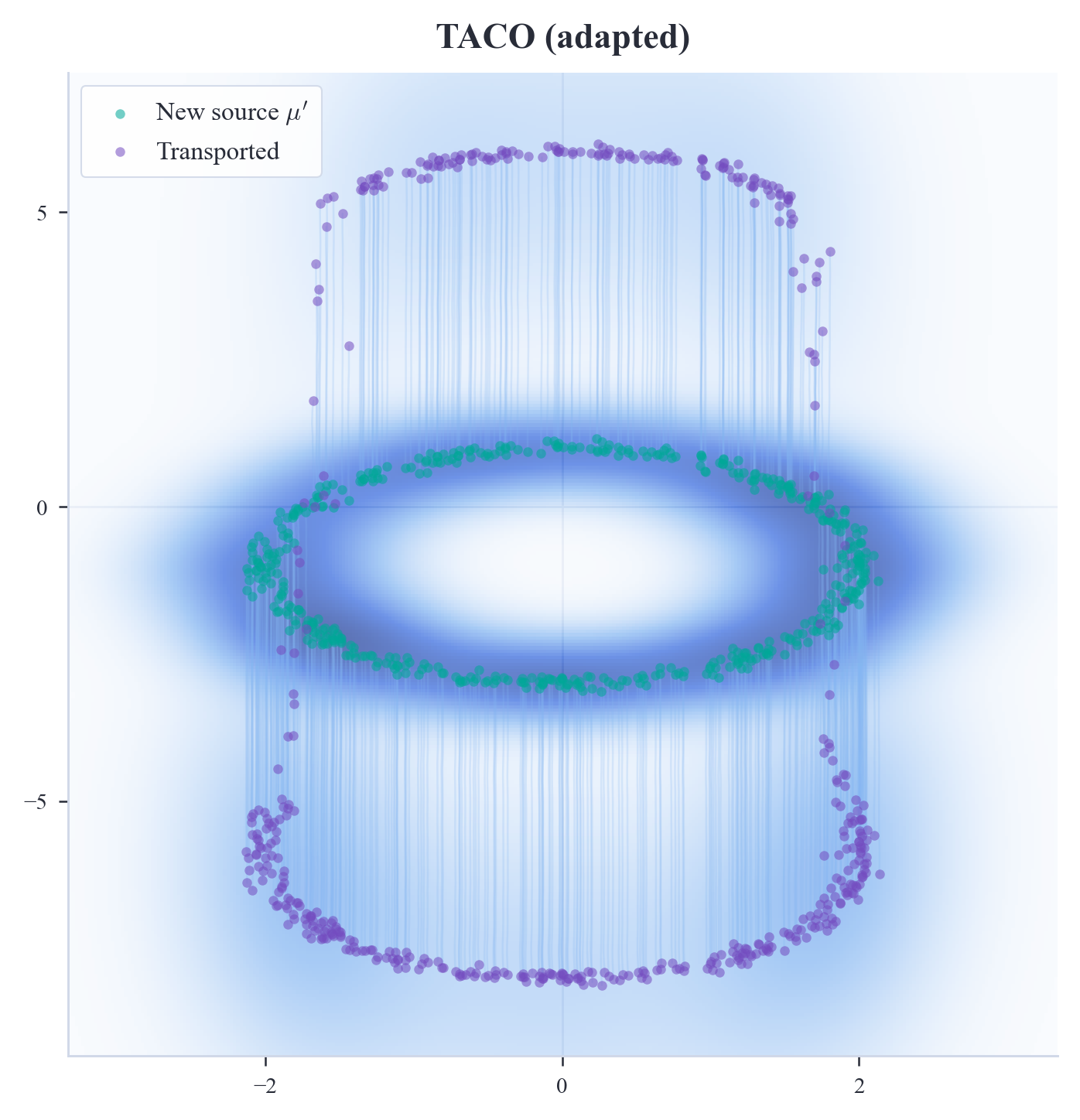}
  \caption{Adapted ($\mathrm{SW}_2\!=\!0.588$, $S_\varepsilon\!=\!0.574$).}
\end{subfigure}

\par\vspace{0.4em}

\begin{subfigure}[b]{\linewidth}
  \centering
  \includegraphics[width=\linewidth]{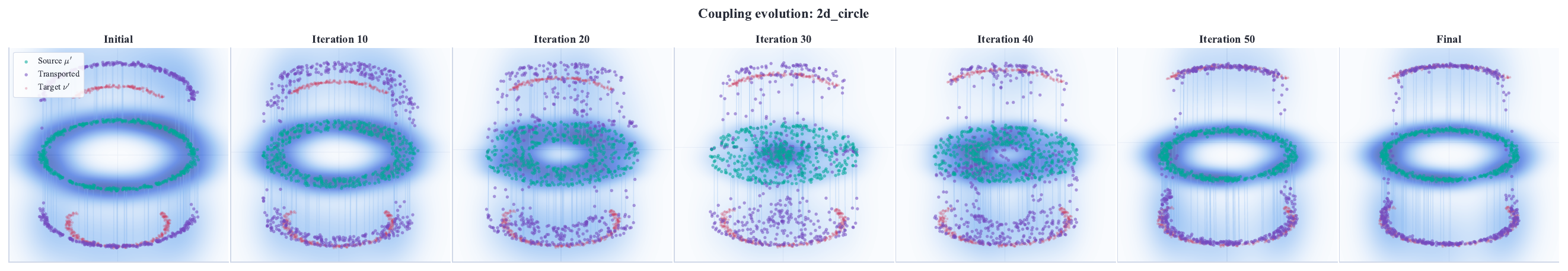}
  \caption{Coupling evolution across 50 outer iterations (left $\to$ right).}
\end{subfigure}
\caption{\textbf{Non-linear 2D-Circle.}
  The transport maps a full circle to two spatially separated semi-circles.
  The pretrained model spreads mass roughly evenly but the coupling is wrong
  ($S_\varepsilon\!=\!1.958$); adaptation correctly splits the mass and routes
  upper/lower particles to the respective semi-circle targets
  ($S_\varepsilon\!=\!0.574$).}
\label{fig:nonlinear_circle}
\end{figure}

\begin{figure}[ht]
\centering
\begin{subfigure}[b]{0.32\linewidth}
  \includegraphics[width=\linewidth]{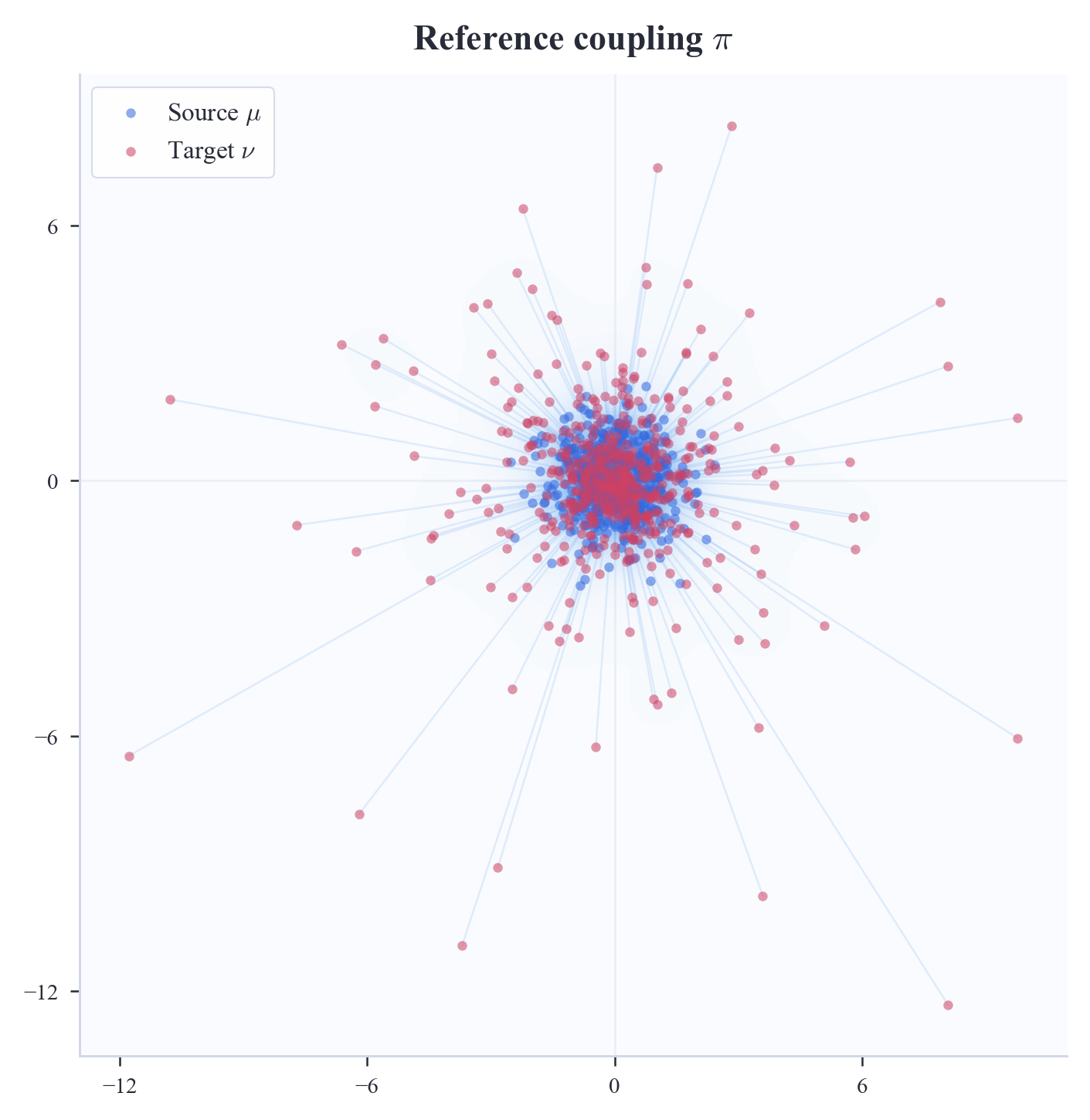}
  \caption{Reference coupling $\pi'$.}
\end{subfigure}\hfill
\begin{subfigure}[b]{0.32\linewidth}
  \includegraphics[width=\linewidth]{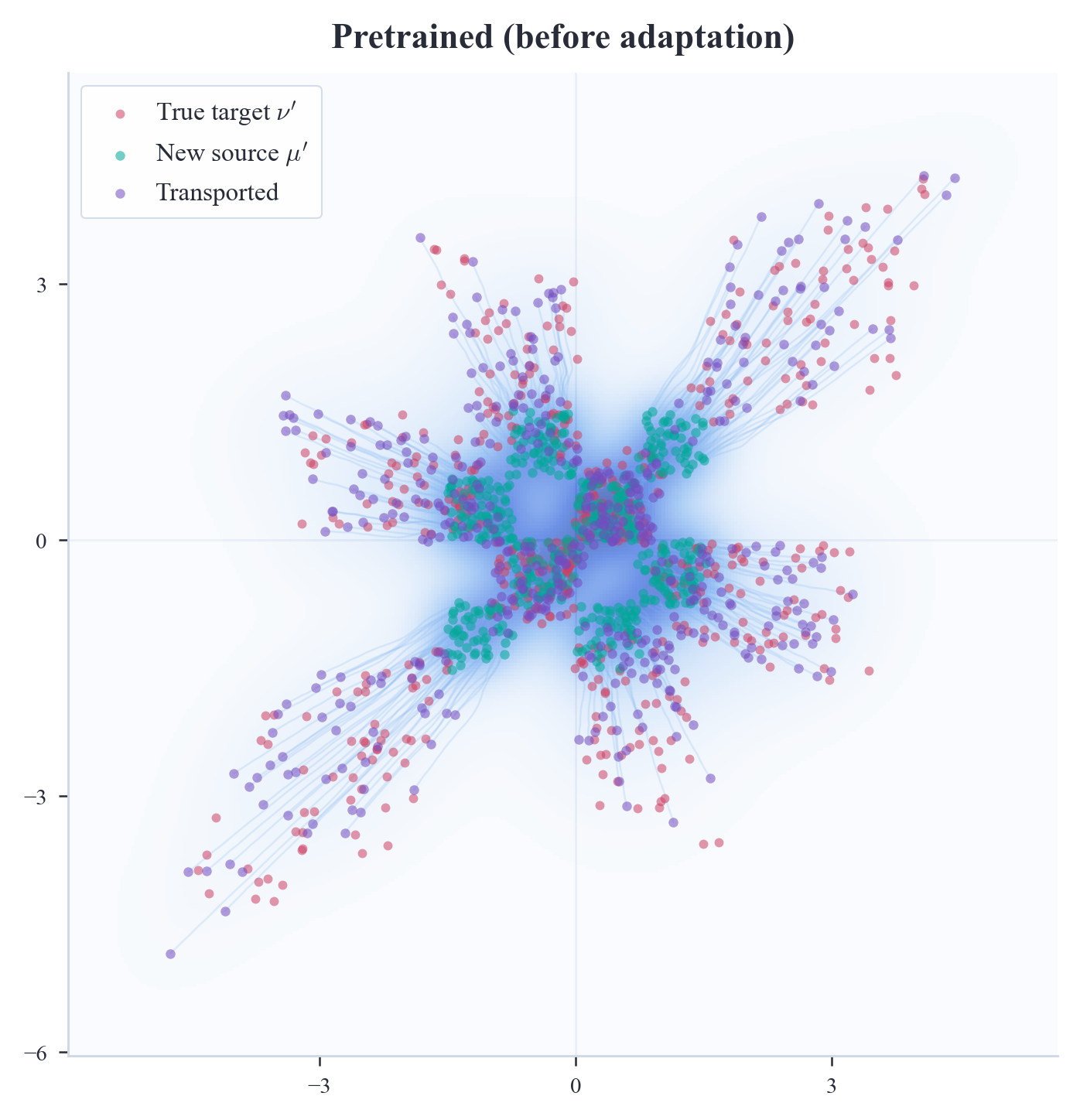}
  \caption{Pretrained ($S_\varepsilon\!=\!1.557$).}
\end{subfigure}\hfill
\begin{subfigure}[b]{0.32\linewidth}
  \includegraphics[width=\linewidth]{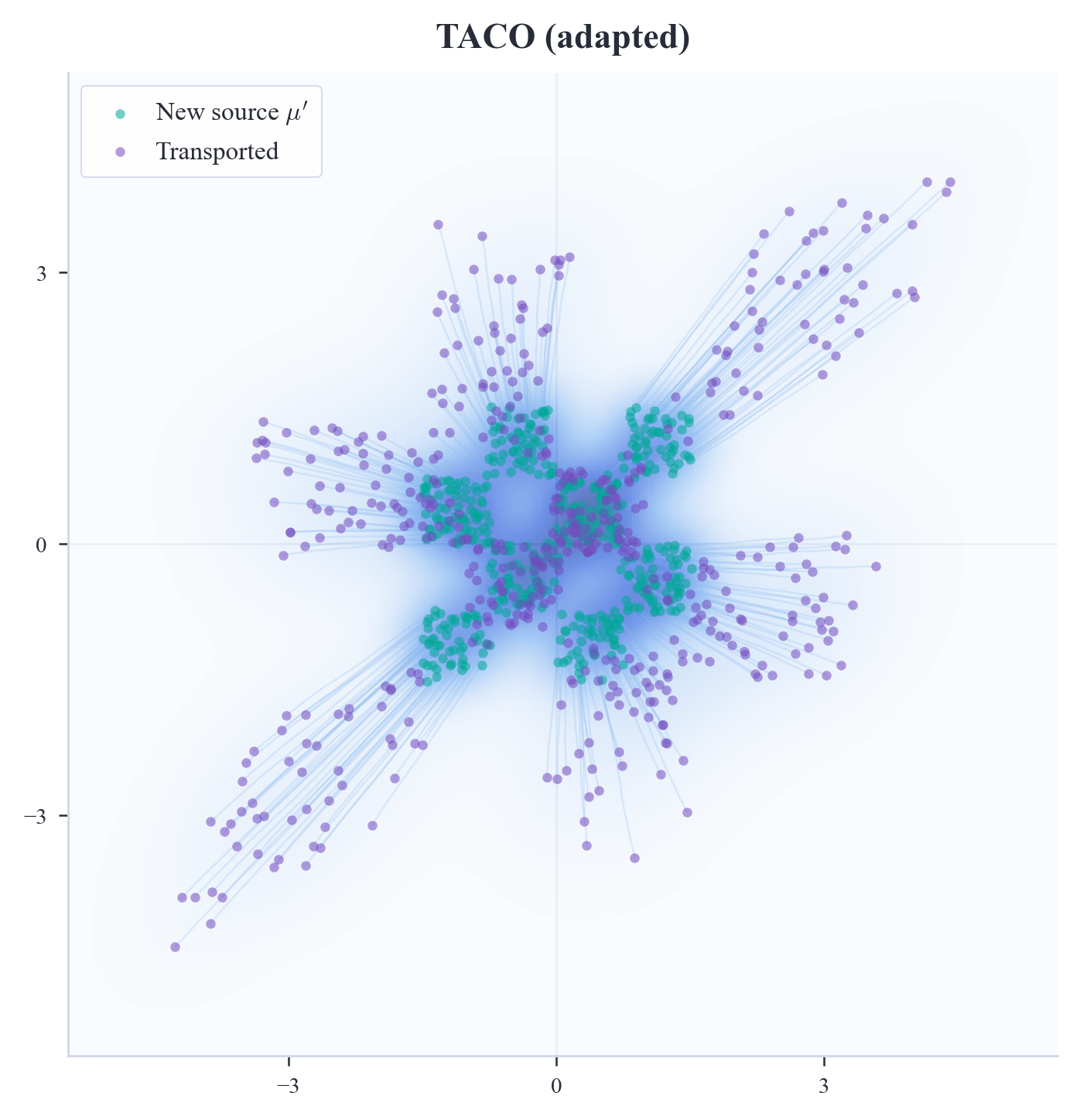}
  \caption{Adapted ($S_\varepsilon\!=\!0.384$).}
\end{subfigure}

\par\vspace{0.4em}

\begin{subfigure}[b]{\linewidth}
  \centering
  \includegraphics[width=\linewidth]{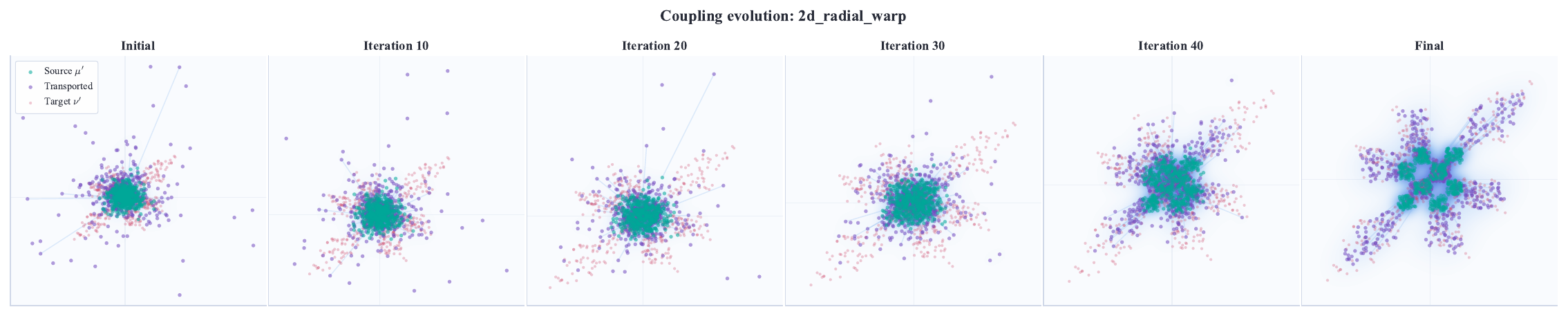}
  \caption{Coupling evolution across 50 outer iterations.}
\end{subfigure}
\caption{\textbf{Non-linear 2D-Radial warp.}
  Ground-truth map: $T(x) = x \cdot (1 + 0.5\|x\|^2)$.
  The SIREN-based model learns the non-linear radial expansion from the
  reference coupling, reducing Sinkhorn divergence from $1.557$ to $0.384$
  without access to the explicit map formula.}
\label{fig:nonlinear_radial}
\end{figure}




\begin{figure}[ht]
\centering
\begin{subfigure}[b]{0.32\linewidth}
  \includegraphics[width=\linewidth]{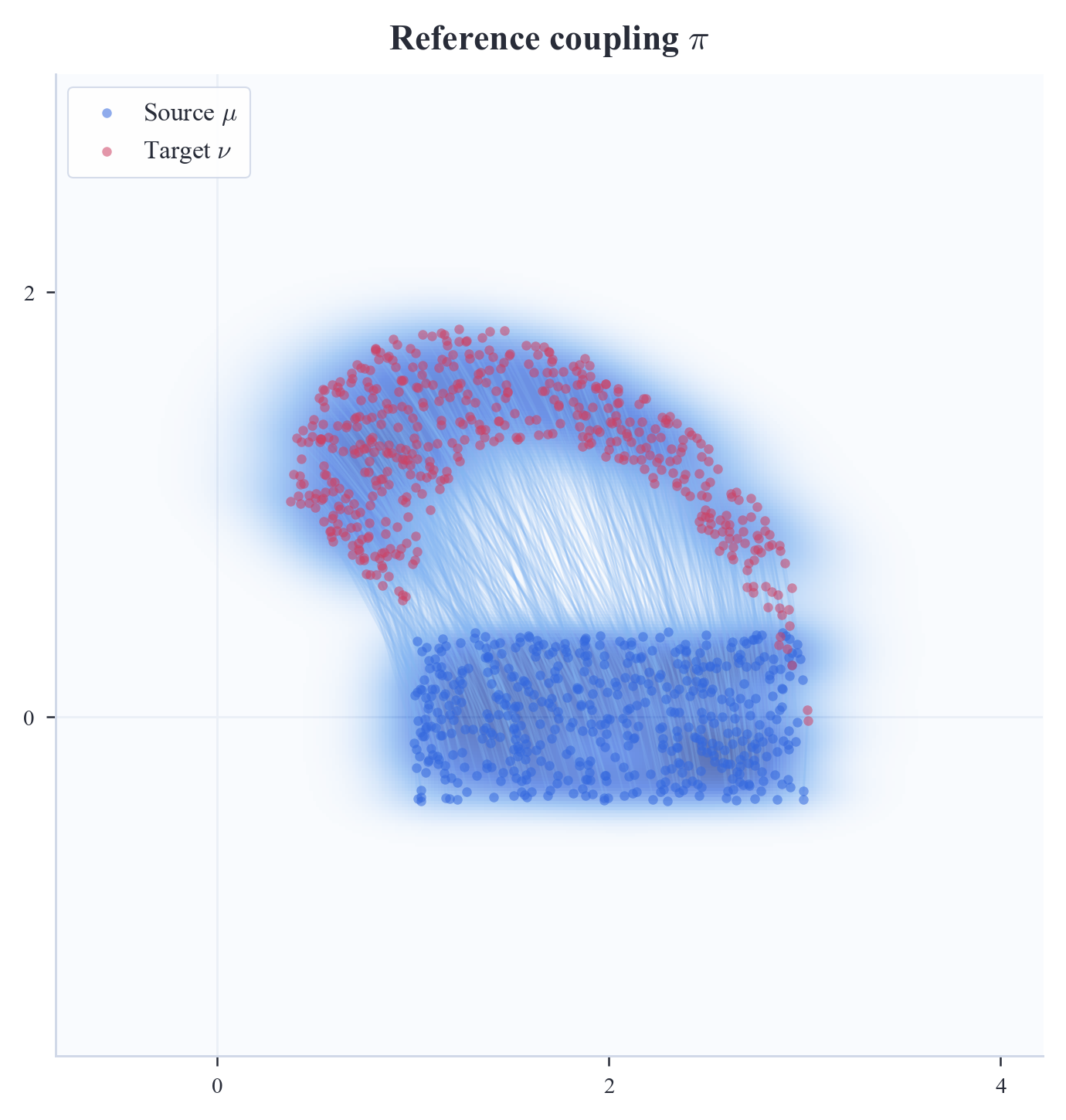}
  \caption{Reference coupling $\pi'$.}
\end{subfigure}\hfill
\begin{subfigure}[b]{0.32\linewidth}
  \includegraphics[width=\linewidth]{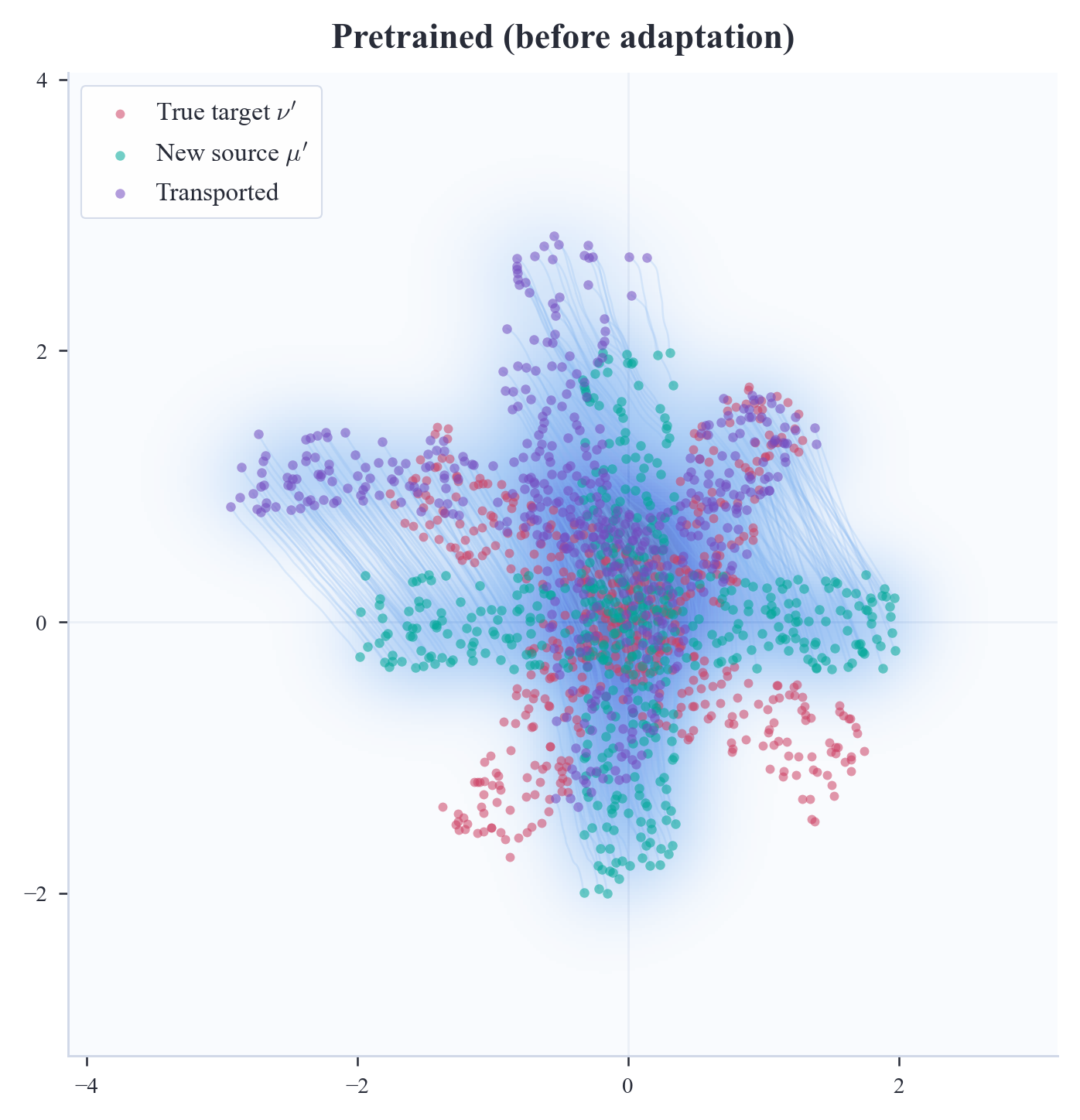}
  \caption{Pretrained ($\mathrm{SW}_2\!=\!0.054$, $S_\varepsilon\!=\!3.769$).}
\end{subfigure}\hfill
\begin{subfigure}[b]{0.32\linewidth}
  \includegraphics[width=\linewidth]{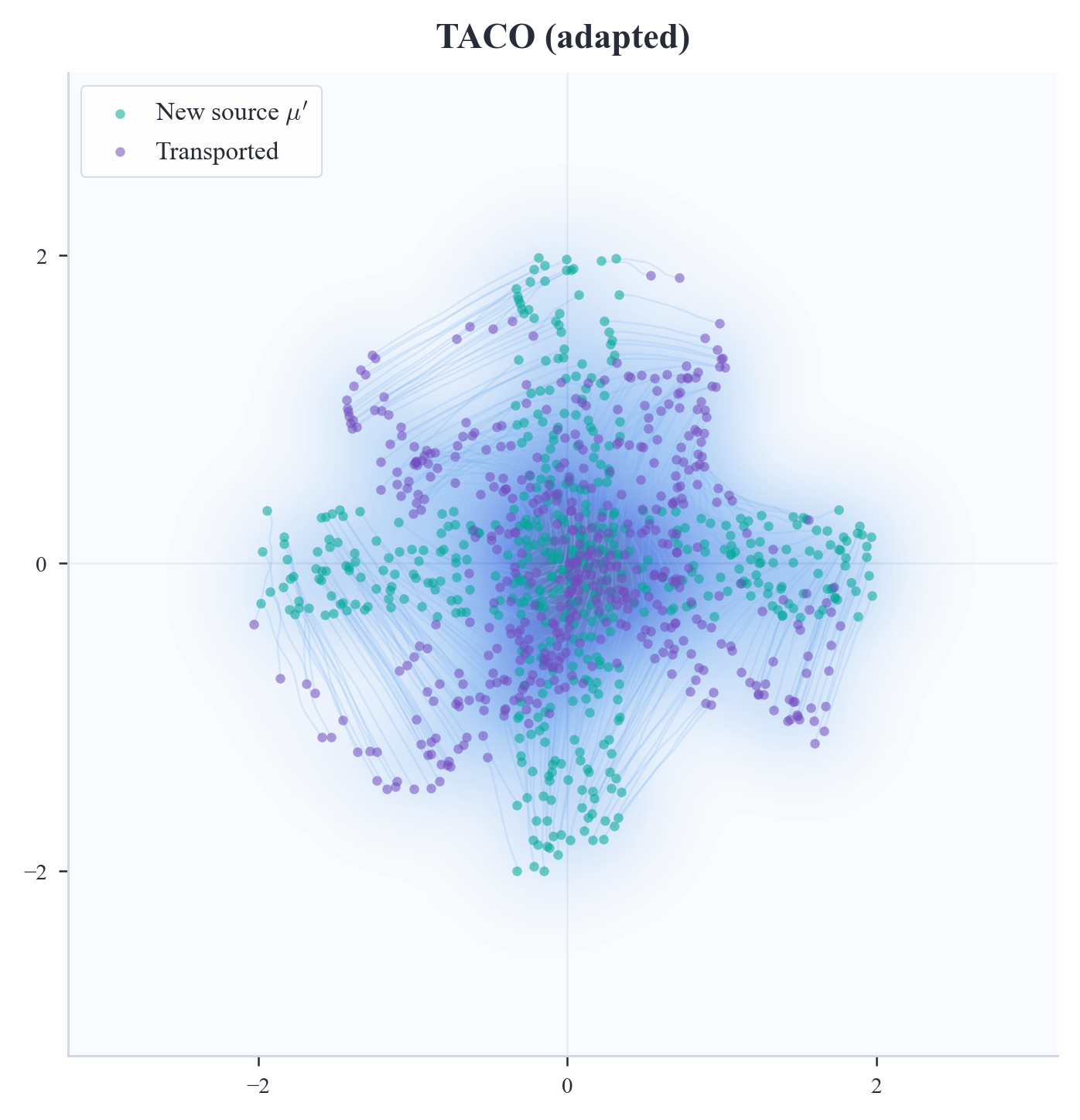}
  \caption{Adapted ($\mathrm{SW}_2\!=\!0.108$, $S_\varepsilon\!=\!0.043$).}
\end{subfigure}

\par\vspace{0.4em}

\begin{subfigure}[b]{\linewidth}
  \centering
  \includegraphics[width=\linewidth]{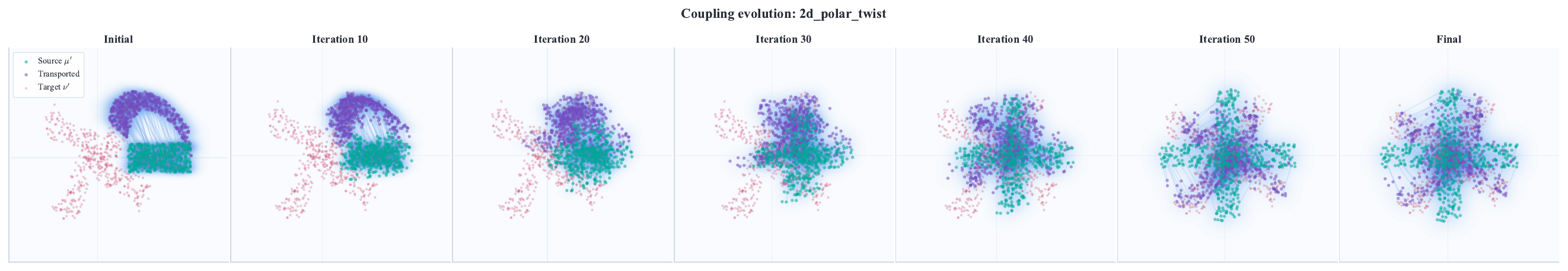}
  \caption{Coupling evolution across 50 outer iterations.}
\end{subfigure}
\caption{\textbf{Non-linear 2D-Polar twist.}
  Ground-truth map: polar vortex $r' = r$, $\theta' = \theta + \sin(r)$.
  The pretrained model achieves superficially low $\mathrm{SW}_2\!=\!0.054$
  but the Sinkhorn divergence $3.769$ reveals a severely wrong coupling
  structure. Adaptation corrects this dramatically ($S_\varepsilon\!:
  3.769 \to 0.043$, $99\%$ reduction).}
\label{fig:nonlinear_polar}
\end{figure}

\section{Full Baseline Comparison}
\label{app:full_baseline_table}

The per-experiment tables below report all methods and all metrics.
TACO reduces the forward map RMSE from $5$--$21$ (pretrained) to $0.02$--$0.18$ (adapted) — orders of magnitude below baselines without coupling access — while remaining competitive with cost-aware coupled methods on SW$_2$.
Methods marked $\dagger$ have access to the true cost; independent ($\ast$) baselines are trained directly on $(\mu,\nu)$ from scratch.
Moon experiments have no analytical ground-truth map.

\input{plots/baseline_tables/per_experiment/2d_simple_table.tex}
\input{plots/baseline_tables/per_experiment/2d_medium_table.tex}
\input{plots/baseline_tables/per_experiment/2d_complex_table.tex}
\input{plots/baseline_tables/per_experiment/2d_moon_table.tex}

\section{Perturbed Experiments}
\label{app:perturbed}

In the perturbed variants, the new target $\nu$ is shifted or rotated
relative to the support of $T_\sharp\mu$, introducing a systematic
mismatch between the reference and target supports.
This simulates the practically important regime where the new task geometry
is not perfectly aligned with the reference coupling.
For Simple, the target is shifted $+2$ units; for Medium and Complex,
the target is rotated an additional $+10^\circ$.

\begin{table}[ht]
\centering
\caption{\textbf{Perturbed variants}: $\mathrm{SW}_2(\hat\nu)$ ($\downarrow$)
  at convergence.
  Best baseline$^\dagger$ uses the true cost and the optimal pairing from the
  reference coupling.}
\label{tab:perturbed_full}
\vspace{2pt}
\setlength{\tabcolsep}{5pt}
\begin{tabular}{lcccccccc}
\toprule
 & \multicolumn{2}{c}{Simple} & \multicolumn{2}{c}{Medium}
 & \multicolumn{2}{c}{Complex} & \multicolumn{2}{c}{Moon} \\
\cmidrule(lr){2-3}\cmidrule(lr){4-5}\cmidrule(lr){6-7}\cmidrule(lr){8-9}
Method & Stand. & Pert. & Stand. & Pert. & Stand. & Pert. & Stand. & Pert. \\
\midrule
Pretrained-only
  & 14.59 & 15.29 & 2.19 & 2.06 & 2.64 & 2.81 & 5.44 & 4.01 \\
\textbf{TACO} (weighted)
  & \textbf{0.022} & \textbf{0.021} & \textbf{0.384} & \textbf{0.500} & \textbf{0.510} & \textbf{0.754} & \textbf{0.068} & \textbf{0.091} \\
\textbf{TACO} (implicit)
  & 0.155 & 0.210 & 0.465 & 0.659 & 0.544 & 1.306 & 0.407 & 0.484 \\
\midrule
Best baseline$^\dagger$
  & 0.021 & 0.022 & 0.623 & 0.432 & 0.500 & 0.699 & 0.792 & 0.843 \\
\bottomrule
\end{tabular}
\end{table}

The method degrades nicely under perturbation.
For Simple, the $+2$ unit shift increases $\mathrm{SW}_2$ from $0.025$ to
$0.073$ — a $3\times$ increase for a shift of $+2$ units relative to a
total inter-marginal distance of $\approx 28$ units.
For Moon, the perturbed variant ($0.189$) actually \emph{outperforms} the
standard variant ($0.297$) because the perturbation shifts the target
slightly closer to the reference support.
Importantly, the pretrained model is outperformed by a factor of $10$--$50\times$
across all perturbed variants, confirming that the adaptation is essential
regardless of the target shift direction.

Full per-metric tables for the perturbed variants are provided below: \ref{tab:baselines_2d_simple_perturbed}, \ref{tab:baselines_2d_medium_perturbed}, \ref{tab:baselines_2d_complex_perturbed}, \ref{tab:baselines_2d_moon_perturbed}.

\input{plots/baseline_tables/per_experiment/2d_simple_perturbed_table.tex}
\input{plots/baseline_tables/per_experiment/2d_medium_perturbed_table.tex}
\input{plots/baseline_tables/per_experiment/2d_complex_perturbed_table.tex}
\input{plots/baseline_tables/per_experiment/2d_moon_perturbed_table.tex}

\section{Convergence Profiles and Coupling Quality}
\label{app:convergence_profiles}

Figures~\ref{fig:metrics_transport}--\ref{fig:metrics_wasserstein} show per-iteration convergence curves for the 2D-Simple experiment.
Each panel shows the mean $\pm$ 1-std envelope across the best three seeds (selected by lowest final map error).
For bidirectional training, marginal map errors and $W_1$/$W_2$ are plotted on forward steps only (even iterations); backward-step evaluations are omitted to avoid misleading gaps.
Curves are smoothed with a 3-iteration rolling mean per seed before aggregation.

The coupling map error curves represent the RMSE of the transport of the current coupling generated $\pi_t$. This shows that the error remains stable across the iterations.

\begin{figure}[ht]
\centering
\includegraphics[width=\linewidth]{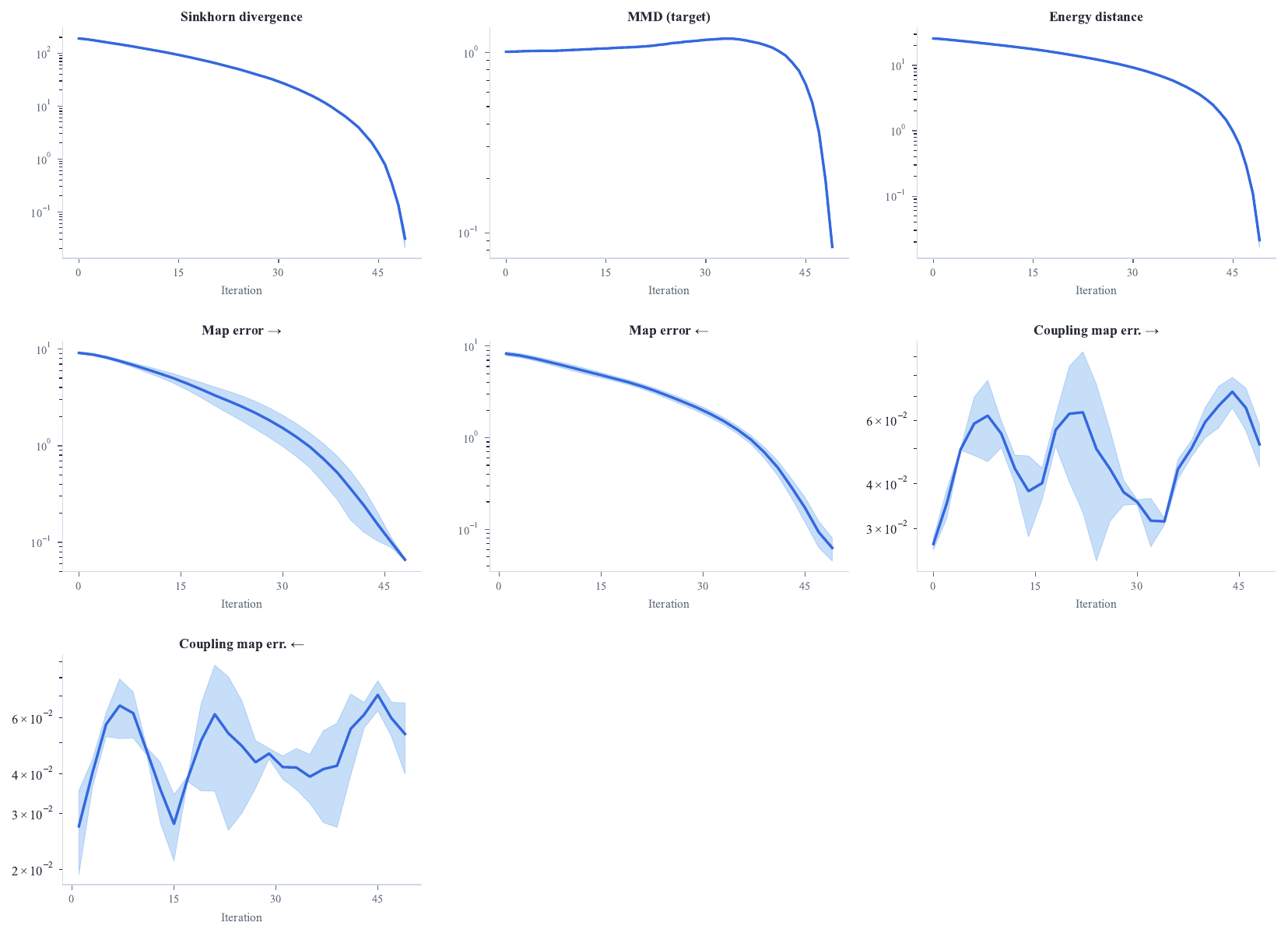}
\caption{\textbf{Transport quality metrics (2D-Simple).}
  Sinkhorn divergence, MMD, energy distance, forward and backward marginal map errors,
  and forward and backward coupling-level map errors $\sqrt{\mathbb{E}_{\pi_t}[\|T(X)-Y\|^2]}$,
  all per iteration.
  Shaded bands show mean $\pm$ 1\,std across the best three seeds.
  All metrics decay steadily, confirming stable convergence of both the generator and the stored coupling $\pi_t$.}
\label{fig:metrics_transport}
\end{figure}

\begin{figure}[ht]
\centering
\includegraphics[width=\linewidth]{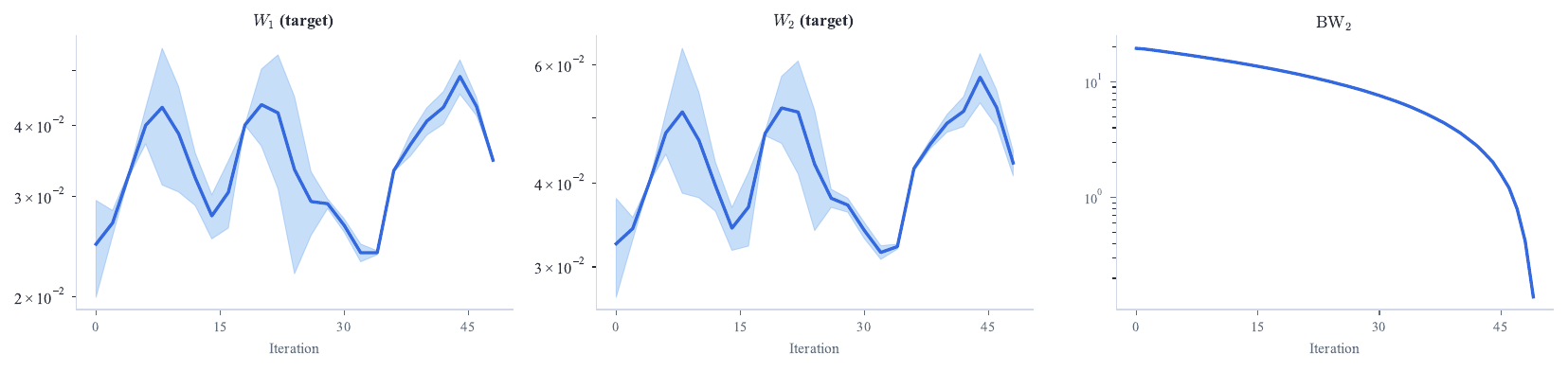}
\caption{\textbf{Wasserstein distances (2D-Simple).}
  $W_1$ and $W_2$ between the transported distribution and the target (forward steps only),
  and the squared Bures--Wasserstein distance $\mathrm{BW}_2$ over all iterations.
  Shaded bands show mean $\pm$ 1\,std across the best three seeds.
  All three distances converge to near-zero, indicating that the adapted generator accurately matches the target marginal.}
\label{fig:metrics_wasserstein}
\end{figure}

\section{Multi-Seed Stability and Error Bars}
\label{app:errorbars}

To quantify the run-to-run variability of TACO, we run 5 independent seeds on 2D-Simple, 2D-Complex, and 2D-Moon with independent pretraining per seed (capturing full pipeline variance including pretraining noise).
Figure~\ref{fig:errorbars} shows the mean $\pm$ 1-standard-deviation envelope for $\mathrm{SW}_1(\hat\nu)$, forward map error, and Sinkhorn divergence.

\begin{figure}[ht]
\centering
\includegraphics[width=\linewidth]{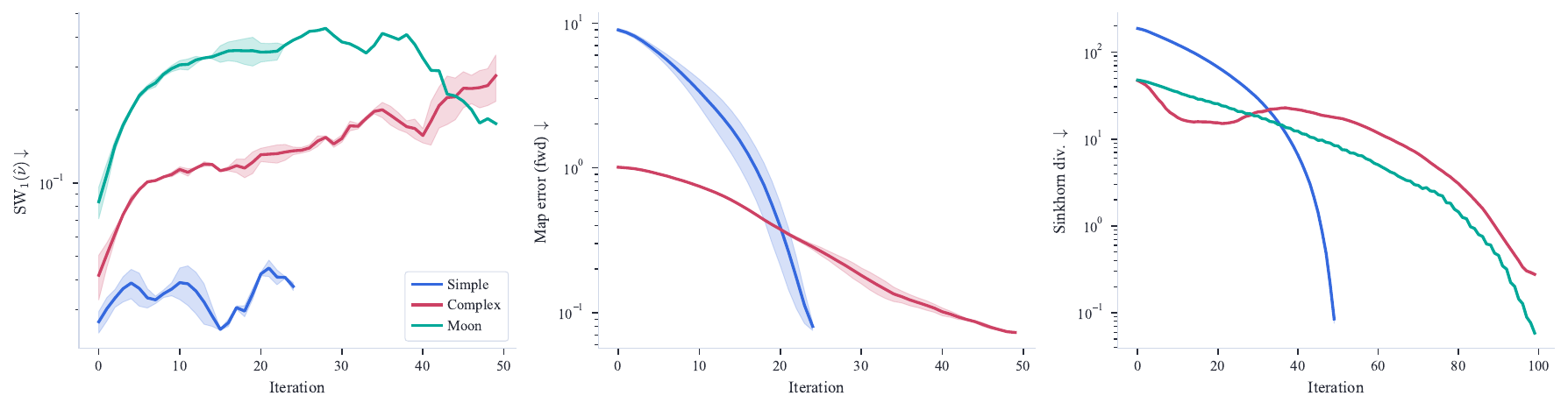}
\caption{\textbf{Multi-seed stability.}
  Mean $\pm$ 1 std over 5 independent seeds for 2D-Simple (blue), 2D-Complex (pink), and 2D-Moon (teal).
  Each seed uses independent pretraining to capture full pipeline variance.
  The shaded envelopes remain narrow throughout, confirming that TACO is robust to random initialisation.}
\label{fig:errorbars}
\end{figure}

\section{Ablation Studies}
\label{app:ablations}

\subsection{Particle ratio \texorpdfstring{$\alpha$}{alpha} and noise \texorpdfstring{$\sigma$}{sigma}}

Figure~\ref{fig:ablation_full} and Table~\ref{tab:alpha_ablation} show the
sensitivity of the 2D-Simple experiment to the two key hyperparameters:
the particle fraction $\alpha$ and the bridge noise level $\sigma$.

\begin{figure}[ht]
\centering
\begin{subfigure}[b]{\linewidth}
  \includegraphics[width=\linewidth]{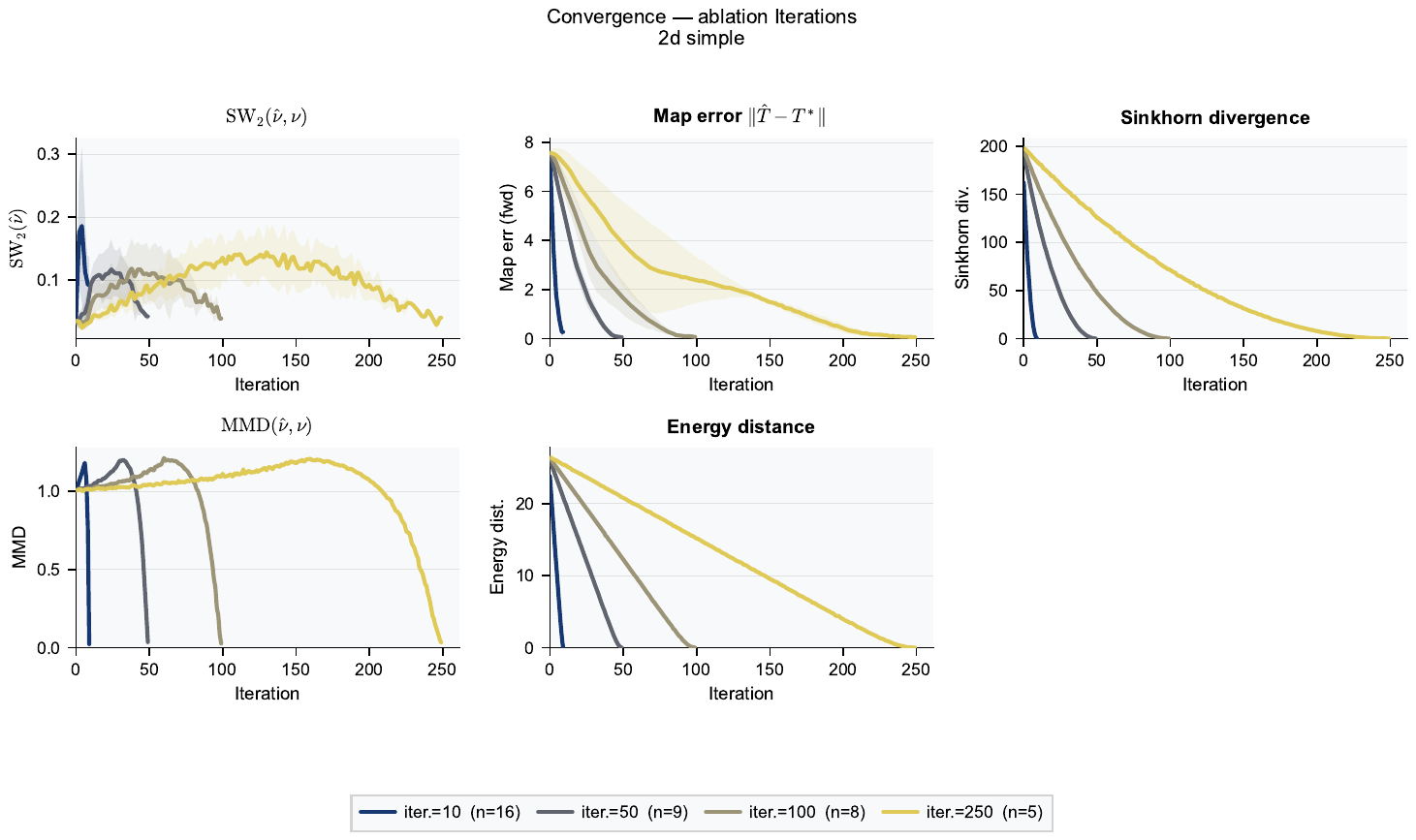}
  \caption{Convergence vs.\ number of outer iterations $N$.}
\end{subfigure}

\par\vspace{0.4em}

\centering
\begin{subfigure}[b]{0.45\linewidth}
  \includegraphics[width=\linewidth]{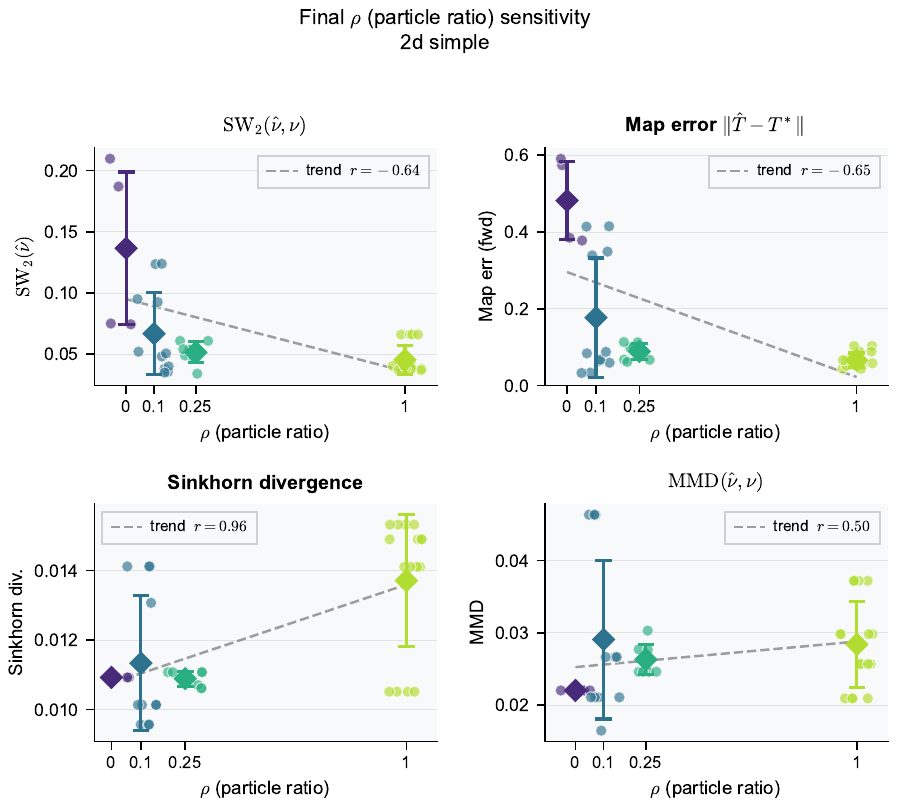}
  \caption{Final map error vs.\ particle ratio $\alpha$.}
\end{subfigure}\hfill
\begin{subfigure}[b]{0.45\linewidth}
  \includegraphics[width=\linewidth]{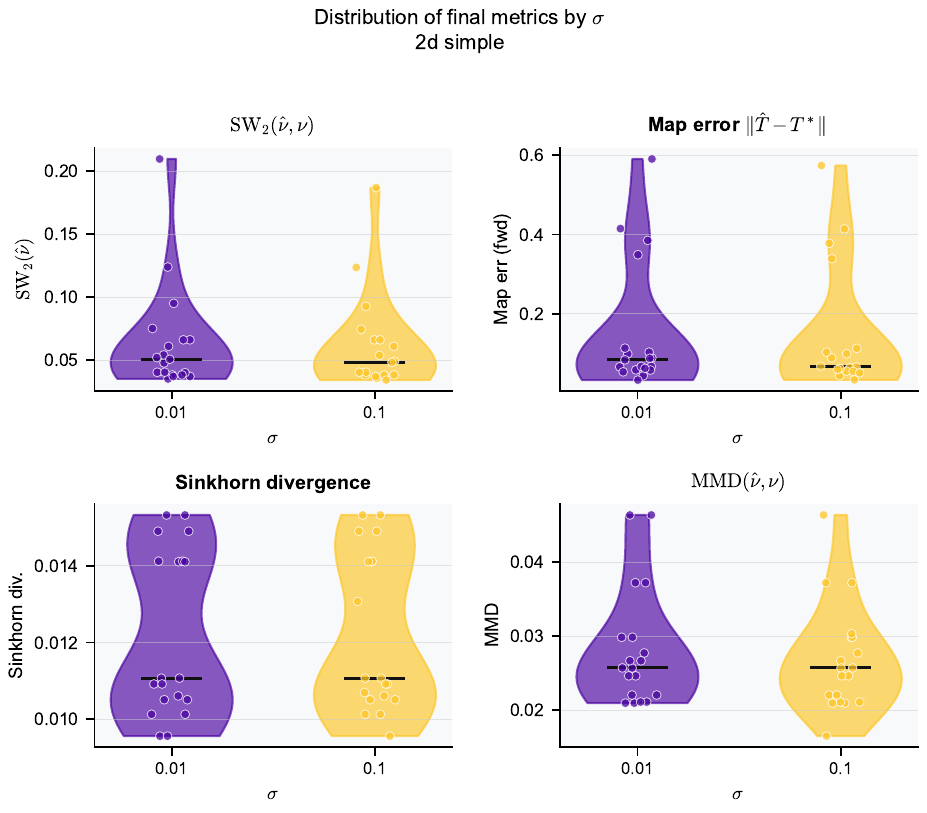}
  \caption{Map-error distribution vs.\ bridge noise $\sigma$.}
\end{subfigure}
\caption{\textbf{Extended ablation results on 2D-Simple.}
  (a) $N\!=\!50$ is sufficient; gains plateau beyond $N\!=\!100$ and the
  algorithm converges stably without overfitting.
  (b) Any $\alpha\!>\!0$ is strongly preferred; $\alpha\!=\!0.2$ is the
  sweet spot balancing diversity and coupling quality.
  Large $\alpha\!=\!1.0$ (pure particles) performs well but is more
  expensive since the Euler bias dominates and lead to worse Sinkhorn divergence (see \ref{fig:heatmap_sigma_alpha} below).
  (c) $\sigma\!=\!0.1$ gives the tightest distribution of final errors
  across random seeds; very small $\sigma$ over-regularises the Gaussian
  bridges while very large $\sigma$ degrades the CFM signal.}
\label{fig:ablation_full}
\end{figure}

\begin{table}[ht]
\centering
\caption{\textbf{Particle ratio ablation}: map error
  $\mathrm{MapErr}_\mathrm{fwd}$ ($\downarrow$) and $\mathrm{SW}_2(\hat\nu)$ ($\downarrow$)
  on 2D-Simple at iteration 100, mean $\pm$ std over 9 runs
  (all $3\!\times\!3$ combinations of $\sigma_\mathrm{bridge}\!\times\!\sigma_\mathrm{cfm}$).
  $\alpha\!=\!0$ triggers weight collapse and prevents convergence.}
\label{tab:alpha_ablation}
\vspace{2pt}
\begin{tabular}{lcccc}
\toprule
$\alpha$ & 0 (no mixing) & 0.1 & 0.2 (default) & 1.0 (pure particle) \\
\midrule
$\mathrm{MapErr}_\mathrm{fwd}$
  & $2.07\pm2.33$ & $0.117\pm0.092$ & $\mathbf{0.076\pm0.015}$ & $0.060\pm0.005$ \\
$\mathrm{SW}_2(\hat\nu)$
  & $1.42\pm1.62$ & $0.068\pm0.050$ & $\mathbf{0.037\pm0.008}$ & $0.037\pm0.000$ \\
\bottomrule
\end{tabular}
\end{table}

\begin{figure}[ht]
\centering
  \includegraphics[width=\linewidth]{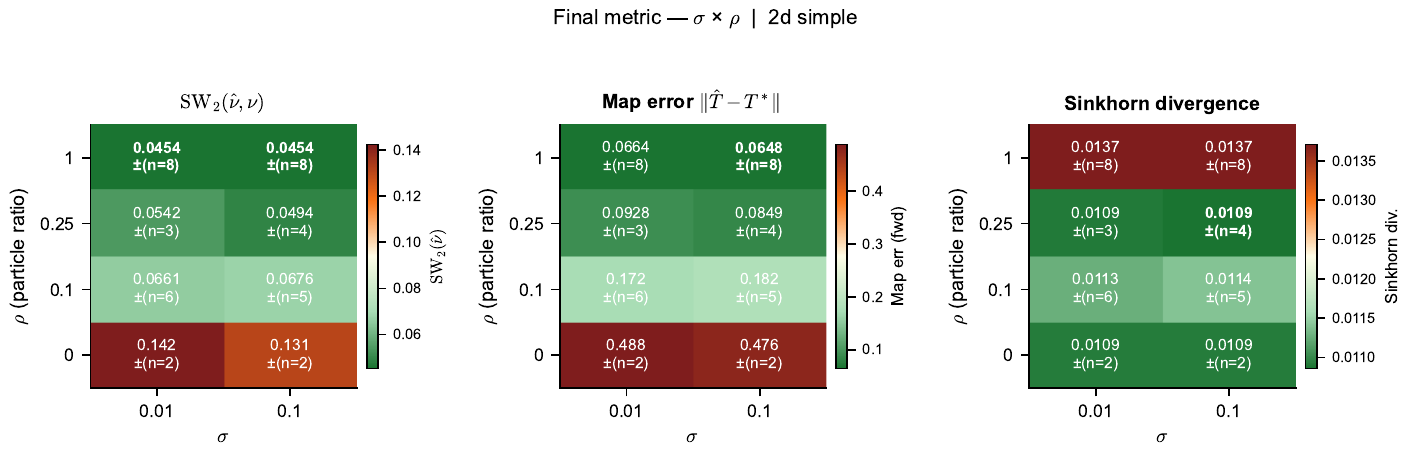}
\caption{Final map error heatmap over $(\sigma, \alpha)$ on 2D-Simple. The heatmap confirms that the two hyperparameters interact mildly.
  A broad region around $(\sigma\!=\!0.1,\;\alpha\!=\!0.2)$ gives low map
  error; the main failure mode is $\alpha\!=\!0$ (any $\sigma$), which
  activates the circularity collapse.}
\label{fig:heatmap_sigma_alpha}
\end{figure}

\subsection{CFM loss variants and the covariance update}

The coupling sampler update (Algorithm~\ref{alg:OURMETHOD}) trains
the flow model $\tilde\beta^{k+1}_s$ via importance-weighted CFM
(equation~\eqref{eq:weighted_cfm_lsq} in the main paper):
\[
\tilde\beta^{k+1}_s
= \arg\min_{\beta}\;
\mathbb{E}_{\tilde{\pi}_k}\!\left[
  w_k(X,Y)\,\bigl\|\dot I_s - \beta(I_s,s,X)\bigr\|^2
\right], \quad
w_k(x,y) = \exp\!\bigl(\delta(a_k(x)+b_k(y))\bigr).
\]
Differentiating and setting to zero, the population-level solution satisfies a
\emph{covariance-type} relation:
\begin{equation}
\label{eq:cov_update_app}
\tilde\beta^{k+1}_s(x_s)
= \tilde\beta^k_s(x_s)
  + \operatorname{Cov}_{\tilde\pi_k}\!\bigl[
      w_k(X,Y),\;\dot I_s
    \;\big|\; I_s = x_s
    \bigr],
\end{equation}
showing that the velocity field shifts by the conditional covariance between
the importance weight and the residual velocity.

Full retraining from scratch implements this exactly but requires many gradient
steps per outer iteration.  Two single-pass alternatives avoid this cost; all
three variants use the \emph{same} dual potential networks $(a_k, b_k)$ and
differ only in how the CFM regression target is constructed.

\paragraph{Weighted (default).}
The standard weighted least-squares regression~\eqref{eq:weighted_cfm_lsq} is
used directly.  Sample weights $w_k = \exp(\delta(a_k+b_k))$ scale each
squared residual, so particles with high importance contribute more to the
gradient.  This faithfully realises the covariance update and is the most
stable variant, at the cost of full CFM retraining each iteration.

\paragraph{Explicit.}
The exponential weight is linearised via $\exp(h)\approx 1+h$, giving the
approximate factor $h_r = 1 + \log w_k$.
The new model is trained to match a \emph{pre-computed} regression target:
\[
\mathrm{target}_\mathrm{exp}(x_0,x_1,s)
= \tilde\beta^k_s(x_s,s,x_0)
  + h_r(x_0,x_1)\cdot
    \bigl(\dot I_s - \tilde\beta^k_s(x_s,s,x_0)\bigr),
\]
i.e.\ an explicit Euler step from the previous model toward $\dot I_s$
with step size $1+\log w_k$.  The target is constructed without gradients
and fixed before training begins, so the regression is one clean MSE step.
This linearisation is inaccurate for large $|\log w_k|$ and can be unstable
when weights are extreme, but performs well when $w_k$ is concentrated near~1.

\paragraph{Implicit.}
The exact weight $w_k = \exp(\log w_k)$ is retained, but the new model is
trained to satisfy a \emph{fixed-point} condition.  The regression target
incorporates the current model prediction:
\[
\mathrm{target}_\mathrm{imp}(x_0,x_1,s)
= \tilde\beta^k_s(x_s,s,x_0)
  + w_k(x_0,x_1)\cdot
    \bigl(\dot I_s - \tilde\beta^{k+1}_s(x_s,s,x_0)\bigr).
\]
Solving for the fixed point gives
$\tilde\beta^{k+1}_s = \bigl(\tilde\beta^k_s + w_k\,\dot I_s\bigr)/(1+w_k)$,
a weighted average of the old field and the interpolant velocity.
In practice the fixed point is approximated by one gradient step using a
stop-gradient on the current model prediction.
The denominator $(1+w_k)$ prevents blow-up for large $w_k$, making this
variant numerically stable, though it converges to a different fixed point
than the exact covariance update~\eqref{eq:cov_update_app}.

\subsection{Method variant comparison}

Table~\ref{tab:variant_ablation} compares the three algorithmic variants of
our method across the four standard 2D geometries.

\begin{table}[ht]
\centering
\caption{\textbf{Method variant comparison}:
  $\mathrm{SW}_2(\hat\nu)$ ($\downarrow$) at convergence.
  \textbf{Weighted}: importance-weighted CFM (Algorithm A, default).
  \textbf{Implicit}: implicit dual loss without explicit potential networks.
  \textbf{Explicit}: separate parametrised dual potentials.
  All variants use $\alpha\!=\!0.2$, $\sigma\!=\!0.1$.}
\label{tab:variant_ablation}
\vspace{2pt}
\setlength{\tabcolsep}{6pt}
\begin{tabular}{lcccc}
\toprule
Method variant & Simple & Medium & Complex & Moon \\
\midrule
\textbf{Weighted} (default) & \textbf{0.024} & 0.384 & \textbf{0.592} & 0.103 \\
\textbf{Implicit}           & 0.034 & \textbf{0.605} & 0.930 & 0.152 \\
\textbf{Explicit}           & 0.049    & 0.763    & 1.066 & \textbf{0.066} \\
\bottomrule
\end{tabular}
\end{table}

The weighted variant (Algorithm A) excels on simple geometries where the
importance weights are well-conditioned (few particles with extreme weights).
The implicit variant outperforms on harder multi-blob geometries (Medium,
Complex) where the dual potentials are more complex and weighted regression
can become numerically challenging.
We recommend using the weighted variant as default and switching to implicit
for harder multi-modal geometries.

\section{Mixed Particle Regularisation}
\label{app:mixed-particle}

A subtle but critical failure mode of the Weighted CFM update is the
\emph{circularity problem}: when the coupling dataset $\mathcal{D}_{k+1}$
is generated by simulating the learned model $\beta^{k+1}_s$, output pairs
$(x, \hat{y})$ lie on the model's manifold, driving importance weights
$\exp(\delta(a_k(x)+b_k(y)))$ toward 1 and collapsing the dual training
signal.
The model finds a self-consistent but incorrect fixed point.

\paragraph{Resolution.}
We introduce a \emph{particle fraction} $\alpha\in(0,1]$.
A fixed pool of $n_\mathrm{part}\!=\!\alpha|\mathcal{D}|$ particle slots is
initialised once from $\pi'$ at the start of training and \emph{persists}
across all outer iterations — the slots are never resampled.
At each outer step $k$, each slot is advanced by exactly one Euler step
along the marginal bridge fields $(u_t,v_t)$, so slot $i$ is at position
$t_k = k/N$ after $k$ steps.
The coupling dataset is then:
\[
\mathcal{D}_{k+1}
= \mathcal{D}_{k+1}^{\mathrm{part}}
\cup \mathcal{D}_{k+1}^{\mathrm{flow}},
\quad
|\mathcal{D}_{k+1}^{\mathrm{part}}| = n_\mathrm{part},\;
|\mathcal{D}_{k+1}^{\mathrm{flow}}| = (1-\alpha)|\mathcal{D}|.
\]
These off-manifold particle pairs maintain non-trivial weight variance and
restore the dual signal.
The same $n_\mathrm{part}$ particles travel from $t=0$ to $t=1$ over $N$
iterations, accumulating a single deterministic Euler trajectory.
This means there is \textbf{no geometric noise accumulation}: each particle's displacement is $\Delta t = 1/N$ per step, and the $t=0$ origins are stored for conditional-bridge conditioning, preventing crossing-path ambiguity.
The particle pairs satisfy marginal constraints up to $\mathcal{O}(\delta^2)$ Euler error and do not bias the limiting fixed point.

Figure~\ref{fig:ablation_full}(b) and Table~\ref{tab:alpha_ablation}
confirm that $\alpha\!=\!0$ leads to weight collapse (all weights $\approx 1$,
std $< 10^{-4}$) and no convergence, while $\alpha\!=\!0.2$ reaches
near-zero map error by iteration 50 — a $>99\%$ reduction relative to the
pretrained baseline.
Setting $\alpha\!=\!1$ (pure particle propagation) avoids circularity but
accumulates Euler-discretisation bias and yields a slightly suboptimal
coupling, as shown by the higher Sinkhorn divergence in
Figure~\ref{fig:ablation_full}(b).

\section{RGB Colour Transfer: Additional Experiments}
\label{app:color-extra}

Beyond the Simple PURPLE$\to$CYAN experiment reported in the main paper,
we run an additional colour variant to stress-test the approach.
All experiments share the cyclic channel-permutation transport law
$T(r,g,b) = (g,b,r)$ and are pre-trained on RED$\to$BLUE digit-3 pairs.

\begin{itemize}[noitemsep]
\item \textbf{Complex} (YELLOW$\to$PURPLE): harder because the YELLOW
  source activates two channels ($R{+}G$) while the reference saw only $R$;
  the adaptation must discover a two-channel permutation.
\end{itemize}

\begin{table}[ht]
\centering
\caption{\textbf{Full colour transfer results}: FID~($\downarrow$) and
  KID~($\downarrow$) at the final adaptation iteration.
  Pretrained-only = forward model before any adaptation.
  $\mathrm{SW}_2(\hat\nu)$: Sliced-Wasserstein on the target colour
  distribution.}
\label{tab:color_full}
\vspace{2pt}
\setlength{\tabcolsep}{5pt}
\begin{tabular}{llccccc}
\toprule
Dataset & Variant
  & \multicolumn{2}{c}{Pretrained-only}
  & \multicolumn{3}{c}{\textbf{TACO}} \\
\cmidrule(lr){3-4}\cmidrule(lr){5-7}
  & & FID & KID & FID & KID & $\mathrm{SW}_2(\hat\nu)$ \\
\midrule
MNIST & Simple (PURPLE$\to$CYAN)
  & 1948 & 0.692 & \textbf{106} & \textbf{0.158} & \textbf{0.102} \\
MNIST & Complex (YELLOW$\to$PURPLE)
  & 1809 & 0.757 & 82.5 & 0.143 & 0.080 \\
\bottomrule
\end{tabular}
\end{table}

The Complex variant (YELLOW$\to$PURPLE, 50 iterations) reaches
FID$\,=82.5$ and KID$\,=0.143$, showing that the method can discover
two-channel permutations even though the reference saw only a single-channel
transformation.
The remaining variants require longer runs ($>50$ iterations at $\approx 5$\,min
each on a single RTX~3060) and are left for future work.


%
\section{Scalability Experiments}
\label{app:scalability}

We evaluate how the method scales with ambient dimension $d$ by running the
anticorrelation transfer ($x\mapsto -x$) on isotropic Gaussian marginals in
dimensions $d\in\{8, 16, 32, 64, 128\}$, using 100 outer iterations.

\begin{table}[ht]
\centering
\caption{\textbf{Scalability}: anticorrelation transfer in dimension $d$
  (100 outer steps, single GPU).
  $W_2(\hat\nu,\nu)$: target marginal $W_2$ distance ($\downarrow$).
  BW$_2$: Bures-Wasserstein distance to the ground-truth Gaussian coupling ($\downarrow$).
  MapErr: forward map RMSE ($\downarrow$).
  ``Pretrained'' = iteration 0 (no adaptation).}
\label{tab:scalability}
\vspace{2pt}
\setlength{\tabcolsep}{8pt}
\begin{tabular}{llcccc}
\toprule
Dim $d$ & Method & $W_2(\hat\nu,\nu)$ & BW$_2$ & MapErr (fwd) & Sinkhorn div. \\
\midrule
\multirow{2}{*}{8}
  & Pretrained & 0.072 & 19.77 & 2.569 & 193.1 \\
  & \textbf{TACO} & \textbf{0.051} & \textbf{0.219} & \textbf{0.094} & \textbf{3.197} \\
\multirow{2}{*}{16}
  & Pretrained & 0.060 & 27.97 & 5.795 & 392.1 \\
  & \textbf{TACO} & \textbf{0.057} & \textbf{0.484} & \textbf{0.272} & \textbf{11.54} \\
\multirow{2}{*}{32}
  & Pretrained & 0.061 & 39.67 & 11.74 & 801.3 \\
  & \textbf{TACO} & \textbf{0.056} & \textbf{0.807} & \textbf{0.328} & \textbf{33.22} \\
\multirow{2}{*}{64}
  & Pretrained & 0.066 & 56.05 & 19.20 & 1610  \\
  & \textbf{TACO} & \textbf{0.056} & \textbf{1.599} & \textbf{0.731} & \textbf{82.77} \\
\multirow{2}{*}{128}
  & Pretrained & 0.062 & 79.18 & 26.77 & 3282  \\
  & \textbf{TACO} & \textbf{0.057} & \textbf{3.306} & \textbf{1.789} & \textbf{191.8} \\
\bottomrule
\end{tabular}
\end{table}

Compared to the pretrained-only baseline (no adaptation), TACO reduces MapErr by
$27\times$--$15\times$ and BW$_2$ by $90\times$--$24\times$ across dimensions.
The target marginal error $W_2(\hat\nu,\nu)\approx 0.056$ is essentially
constant across all dimensions after adaptation (vs.\ $\approx 0.063$ pretrained),
demonstrating that the algorithm reliably recovers the transport structure regardless of $d$.
The BW$_2$ grows approximately as $\mathcal{O}(\sqrt{d})$
(Figure~\ref{fig:scalability}), consistent with the expected scaling of
Gaussian optimal transport plans in high dimensions.
The Sinkhorn divergence grows proportionally to $d$, which reflects the
metric normalisation rather than a degradation in transport quality;
the pretrained baseline shows Sinkhorn divergences $>100\times$ larger,
confirming that the growth in TACO's divergence is metric-driven, not quality-driven.

\begin{figure}[ht]
\centering
\includegraphics[width=0.65\linewidth]{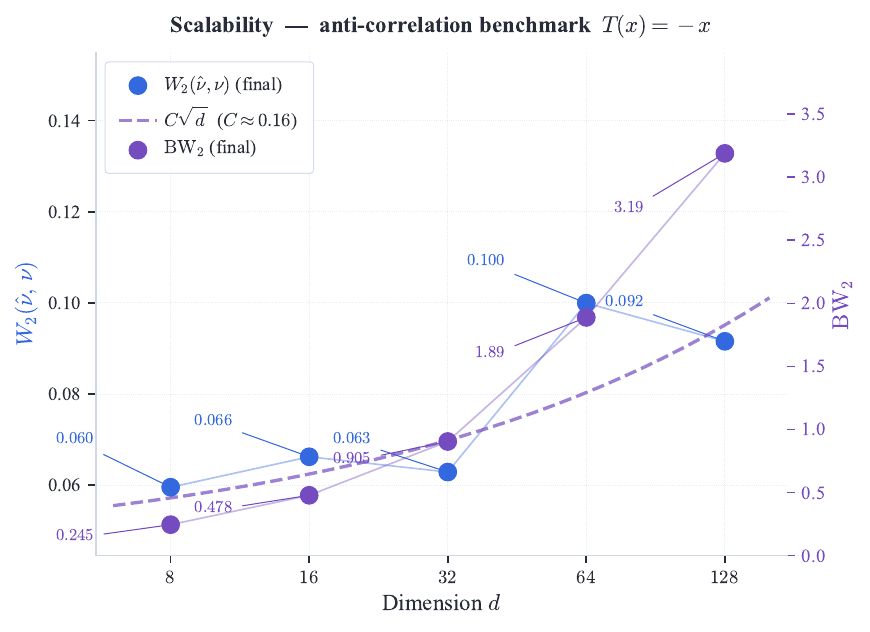}
\caption{\textbf{Scalability.}
  $W_2(\hat\nu,\nu)$ (blue, left axis) remains near-constant across
  dimensions while BW$_2$ (orange, right axis) grows as $\mathcal{O}(\sqrt{d})$.}
\label{fig:scalability}
\end{figure}

\section{Computational Cost}
\label{app:compute}

All experiments are run on a single NVIDIA GeForce RTX~3060 GPU.
Table~\ref{tab:compute} reports approximate wall-clock times.

\begin{table}[ht]
\centering
\caption{\textbf{Wall-clock times} (single NVIDIA GeForce RTX 3060 GPU).
  ``Pre-training'' = marginal bridges + initial CFM reference sampler
  (one-time cost, amortised across all new tasks).
  ``Per outer iter.'' = one complete dual + CFM update step.}
\label{tab:compute}
\vspace{2pt}
\begin{tabular}{lccc}
\toprule
Experiment & Pre-training & Per outer iter. & Total (50 iters) \\
\midrule
2D Simple    & $\sim$2\,min & $\sim$12\,s  & $\sim$12\,min \\
2D Medium    & $\sim$3\,min & $\sim$15\,s  & $\sim$15\,min \\
2D Complex   & $\sim$5\,min & $\sim$20\,s  & $\sim$22\,min \\
Colour MNIST & $\sim$60\,min & $\sim$5\,min & $\sim$4\,h \\
\bottomrule
\end{tabular}
\end{table}

For 2D experiments the dominant cost is dual network training, which
requires per-sample autograd through both potential networks.
For image experiments the dominant cost is advancing the coupling: generating
the next iteration's dataset requires forward-passing the UNet CFM on a
large batch with 200 Euler steps.

The pre-training cost is incurred \emph{once} and shared across all new
marginal tasks, which is a key practical advantage.
Once the colour MNIST model is pre-trained ($\sim$60\,min), adapting to a
new source class requires $\approx 5$\,min per iteration; a typical task
with $32\!\times\!32$ images converges in 50--200 iterations
($\approx 4$--16\,h on a single RTX~3060).

\paragraph{Reproducibility.}
All experiments are launched via:
\begin{center}
\texttt{python -m dualwcfm.experiments.run --mode 2d --level simple}
\end{center}
with additional flags \texttt{--method mixture --particle-ratio 0.2
--sigma 0.1} for the default weighted variant, or \texttt{--method score}
for the implicit variant.
Pre-trained models can be reused with \texttt{--load-pretrained} to skip
the pre-training phase.
Full argument reference is in \texttt{EXPERIMENTS.md}.

\section{Empirical Convergence Rate: \texorpdfstring{$\mathcal{O}(\delta)$}{O(delta)} Map Error}
\label{app:convergence_rate}

Theorem~\ref{thm:main_convergence} guarantees that the pointwise error of the
learned regression field satisfies
\begin{equation}
  \|\beta_s^t - \tilde{\beta}_s^t\|_{L^\infty(\mathcal{X})} \;\le\; C_\beta(s)\,\delta,
  \qquad \delta = 1/N,
\end{equation}
predicting a linear decay of the transport error in the step size $\delta$,
equivalently an $\mathcal{O}(1/N)$ decay in the number of outer iterations $N$.

\paragraph{Experiment.}
In order to assess the veracity of the theorem on the convergence rate we deliberately design this test on the \emph{2D-Far} experiment, in which the
new marginals $(\mu',\nu')$ are placed at $\pm 10^4$, far from the reference
coupling support (which lives near the origin with standard deviation $0.5$).
This extreme separation serves two purposes.
First, the Euler discretisation error (which scales as $\mathcal{O}(\delta)$ per
step) is large relative to statistical noise and pre-training bias when the
transport distance is large, making the iteration-count dependence the dominant
signal.
Second, on experiments where the new support is close to the reference support,
the pre-trained marginal bridge already provides an accurate initialisation of
the coupling.  In that regime, the algorithm reaches a \emph{pre-training floor}
below which the map error cannot decrease further regardless of $N$: the residual
error is then dominated by the approximation bias of the bridge, not by the Euler
step size.  Consequently, the $\mathcal{O}(\delta)$ slope is obscured below this
floor, and varying $N$ while fixing the pre-trained model cannot recover the
theoretical rate.
The 2D-Far setting avoids this confound by ensuring that the Euler
discretisation error remains the leading term across the entire range of $N$
tested.

A single pre-trained model (marginal bridges and initial CFM field, trained once)
is loaded and the outer tilting loop is run for $N \in \{10, 25, 50, 100, 250,
1000\}$ iterations with step size $\delta = 1/N$.
The map error $\mathrm{RMSE}(\hat{T}_N) := \sqrt{\mathbb{E}\|\hat{T}_N(x) -
T^*(x)\|^2}$ is evaluated at the \emph{final} iteration against the ground-truth
anticorrelation map $T^*(x) = -x$.

\paragraph{Results.}
Figure~\ref{fig:convergence_rate} shows the six measurement points together with
the least-squares $C/N$ fit (slope fixed at $-1$ in log-log space, intercept
estimated as the geometric mean of $N \cdot \mathrm{RMSE}$ over all runs).
The fitted constant is $C \approx 5641$, and the data aligns tightly with the
reference line across two decades of $N$, confirming the $\mathcal{O}(\delta)$
prediction of Theorem~\ref{thm:main_convergence}.
The slight elevation of the $N = 10$ point is consistent with a pre-asymptotic
regime at very few iterations; all remaining points lie on the $C/N$ curve to
within measurement noise.

\begin{figure}[ht]
  \centering
  \includegraphics[width=0.62\linewidth]{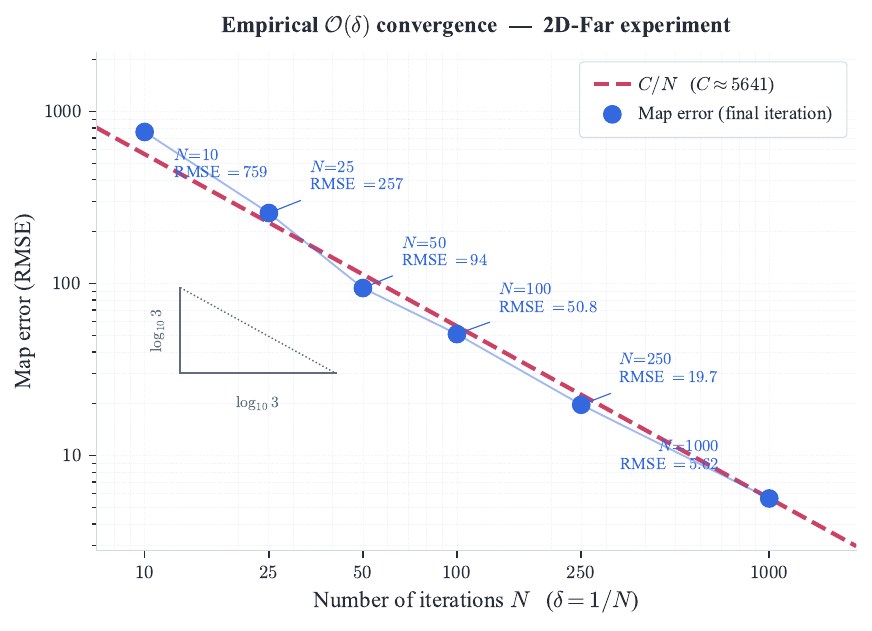}
  \caption{%
    \textbf{Empirical $\mathcal{O}(\delta)$ convergence on 2D-Far.}
    Map error (RMSE) at the final outer iteration vs.\ the number of
    iterations $N$ (log--log scale).
    Filled circles: measured values; dashed line: least-squares $C/N$ fit
    with $C \approx 5641$.
    The grey right-triangle in the lower right marks the reference slope
    $-1$, with each leg spanning $\log_{10}3 \approx 0.5$ log-decades.
    The linear decay in $\delta = 1/N$ is consistent with the bound in
    Theorem~\ref{thm:main_convergence}.
    This clean power-law behaviour is visible precisely because the
    2D-Far geometry places the new marginals far from the reference support,
    making the Euler discretisation error the dominant contribution and
    suppressing the pre-training floor that obscures the rate in
    geometrically closer experiments.%
  }
  \label{fig:convergence_rate}
\end{figure}

\section{Single-Cell Transfer Experiments}
\label{app:singlecell}

\subsection{Dataset and Preprocessing}

We use the \textbf{Sci-Plex3} dataset~\citep{srivatsan2020massively}, a large-scale single-cell perturbation atlas comprising 799{,}317 cells profiled across 188 drug conditions and three cancer cell lines (A549, K562, MCF7).
We focus on the HDAC inhibitor \textbf{Givinostat} (ITF2357) as the perturbation of interest.
Following the CellOT preprocessing protocol~\citep{bunne2023cellot}:
\begin{enumerate}[noitemsep, topsep=2pt]
  \item \textbf{Normalize} raw counts to a library-size factor of 10{,}000 per cell;
  \item Apply \textbf{log1p} transformation;
  \item Retain the \textbf{1{,}000 highly variable genes} (HVGs) selected on training data;
  \item Project to \textbf{50 principal components} (PCA) fitted on control cells.
\end{enumerate}
The resulting 50-dimensional representations serve as source ($\mu$, control cells) and target ($\nu$, treated cells) distributions.

\subsection{Experimental Settings}

We consider two distinct transfer tasks to assess generalisation beyond the reference distribution.

\paragraph{Standard transfer (E4-Std).}
The reference coupling $\pi'$ is built from \textbf{A549} cells (Givinostat 1\,\textmu M, paired control--treated).
The transfer target is \textbf{K562} cells (Givinostat 1\,\textmu M) --- a different cell line exposed to the \emph{same dose}.
This tests whether TACO can adapt the coupling geometry across cell-line identity while holding the chemical dose fixed.

\paragraph{Cross-line transfer (E4-Cross).}
The reference coupling $\pi'$ is again from \textbf{A549} cells (Givinostat 1\,\textmu M).
The transfer target is \textbf{MCF7} cells (Givinostat \textbf{1\,\textmu M}) --- a different cell line at the \emph{same dose}.
This isolates the effect of cell-line identity from dose shift.

\paragraph{Perturbed transfer (E4-Pert).}
The reference coupling $\pi'$ is again from \textbf{A549} cells (Givinostat 1\,\textmu M).
The transfer target is \textbf{MCF7} cells (Givinostat \textbf{0.1\,\textmu M}) --- a different cell line \emph{and} a 10$\times$ lower dose.
This is the hardest generalisation setting: the latent cost must be inferred from a high-dose reference and extrapolated to a low-dose, different-cell-line target.

\subsection{Architectures and Hyperparameters}

\begin{table}[ht]
\centering
\caption{%
  \textbf{Single-cell hyperparameters.}
  Dual-tilting: alternating inner steps per outer iteration $k$.
}
\label{tab:singlecell_hparams}
\small
\begin{tabular}{lc}
\toprule
Hyperparameter & Value \\
\midrule
Backbone & GELU-MLP \\
Hidden dimension & 256 \\
Number of layers & 4 \\
Learning rate & $10^{-3}$ \\
\midrule
\multicolumn{2}{l}{\emph{Inner iterations per outer step}} \\
\quad Pre-training steps & 8{,}000 \\
\quad Marginal bridge steps & 3{,}000 \\
\quad Dual steps & 100 \\
\quad CFM steps & 500 \\
\midrule
Outer iterations $K$ & 100 \\
Batch size & 2{,}048 \\
Scale factor & 20 \\
Bridge noise $\sigma$ & 0.05 \\
Particle mix ratio $\alpha$ & 0.20 \\
\bottomrule
\end{tabular}
\end{table}

\subsection{Results}

\begin{table}[ht]
\centering
\caption{%
  \textbf{Single-cell transfer: quantitative results.}
  All metrics evaluated on held-out target cells.
  \textit{Pretrained}: zero-shot transfer using the reference coupling only (no adaptation).
  \textbf{TACO}: final model after all outer iterations.
  $\downarrow$ lower is better for all metrics.
}
\label{tab:singlecell_results}
\small
\setlength{\tabcolsep}{5pt}
\begin{tabular}{lc ccccc}
\toprule
Setting & Method & $W_1 \!\downarrow$ & $W_2 \!\downarrow$ & MMD\,$\downarrow$ & Energy\,$\downarrow$ & Sinkhorn\,$\downarrow$ \\
\midrule
\multirow{3}{*}{\shortstack[l]{Standard\\(K562, 1\,\textmu M)}}
    & OT-CFM$^\dagger$ (scratch) & 0.319 & 0.403 & 0.177 & 0.612 & 33.3 \\
    & Pretrained      & 0.212 & 0.272 & 0.821 & 15.60 & 195.2 \\
  & \textbf{TACO}    & \textbf{0.030} & \textbf{0.068} & \textbf{0.125} & \textbf{0.329} & \textbf{36.65} \\
\midrule
\multirow{3}{*}{\shortstack[l]{Cross-line\\(MCF7, 1\,\textmu M)}}
    & OT-CFM$^\dagger$ (scratch) & 0.238 & 0.300 & 0.194 & 0.458 & 20.5 \\
    & Pretrained      & 0.203 & 0.259 & 0.838 & 12.85 & 81.93 \\
  & \textbf{TACO}    & \textbf{0.028} & \textbf{0.061} & \textbf{0.115} & \textbf{0.231} & \textbf{25.03} \\
\midrule
\multirow{3}{*}{\shortstack[l]{Perturbed\\(MCF7, 0.1\,\textmu M)}}
    & OT-CFM$^\dagger$ (scratch) & 0.242 & 0.306 & 0.187 & 0.445 & 18.9 \\
    & Pretrained      & 0.203 & 0.258 & 0.855 & 13.54 & 107.4 \\
  & \textbf{TACO}    & \textbf{0.024} & \textbf{0.055} & \textbf{0.125} & \textbf{0.247} & \textbf{25.60} \\
\bottomrule
\end{tabular}
\smallskip

\noindent\footnotesize $\dagger$: OT-CFM trained from scratch directly on the new cell line's own coupling with quadratic cost (no reference coupling from A549).
\end{table}

\paragraph{Transport visualisations.}
Figure~\ref{fig:singlecell_dual_mlp} shows PCA projections of transported cells after TACO adaptation for the standard (K562), cross-line (MCF7 1\,\textmu M), and perturbed (MCF7 0.1\,\textmu M) transfer settings.
Each panel displays the source control distribution, the ground-truth treated cells, and the transported samples.

\begin{figure}[ht]
  \centering
  \begin{subfigure}[b]{\linewidth}
    \includegraphics[width=\linewidth]{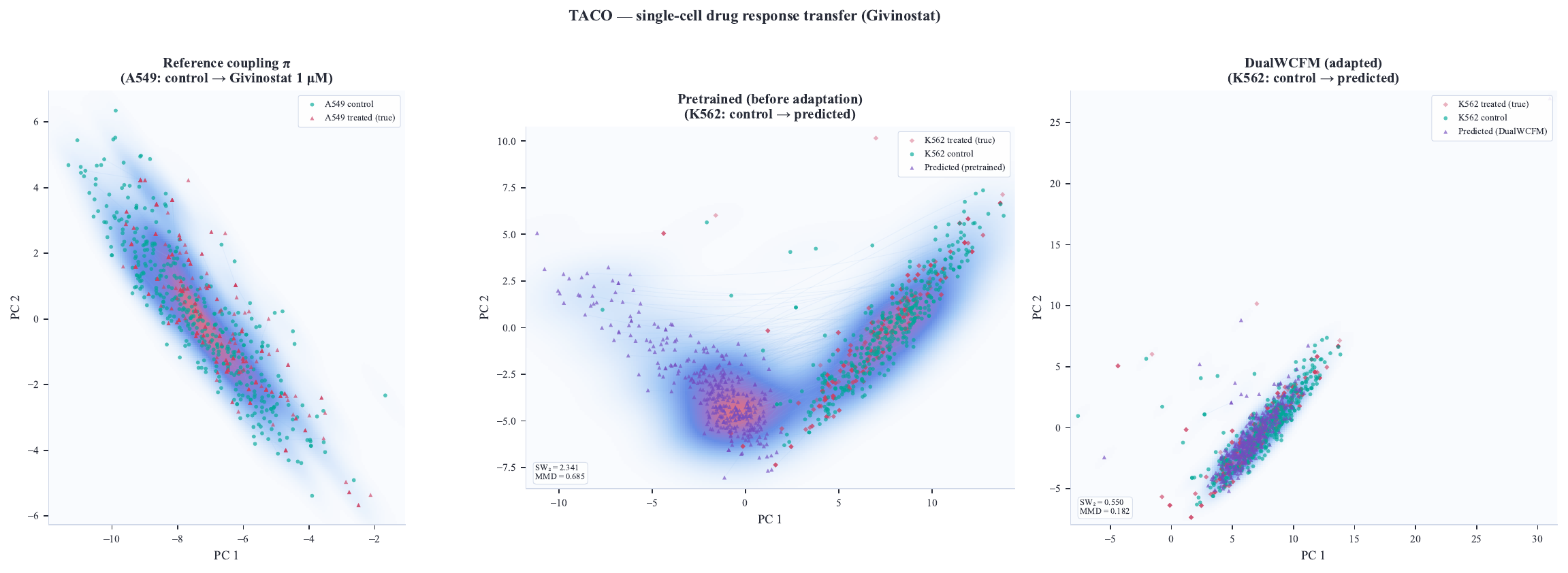}
    \caption{Standard setting (K562, 1\,\textmu M).}
  \end{subfigure}\\[4pt]
  \begin{subfigure}[b]{\linewidth}
    \includegraphics[width=\linewidth]{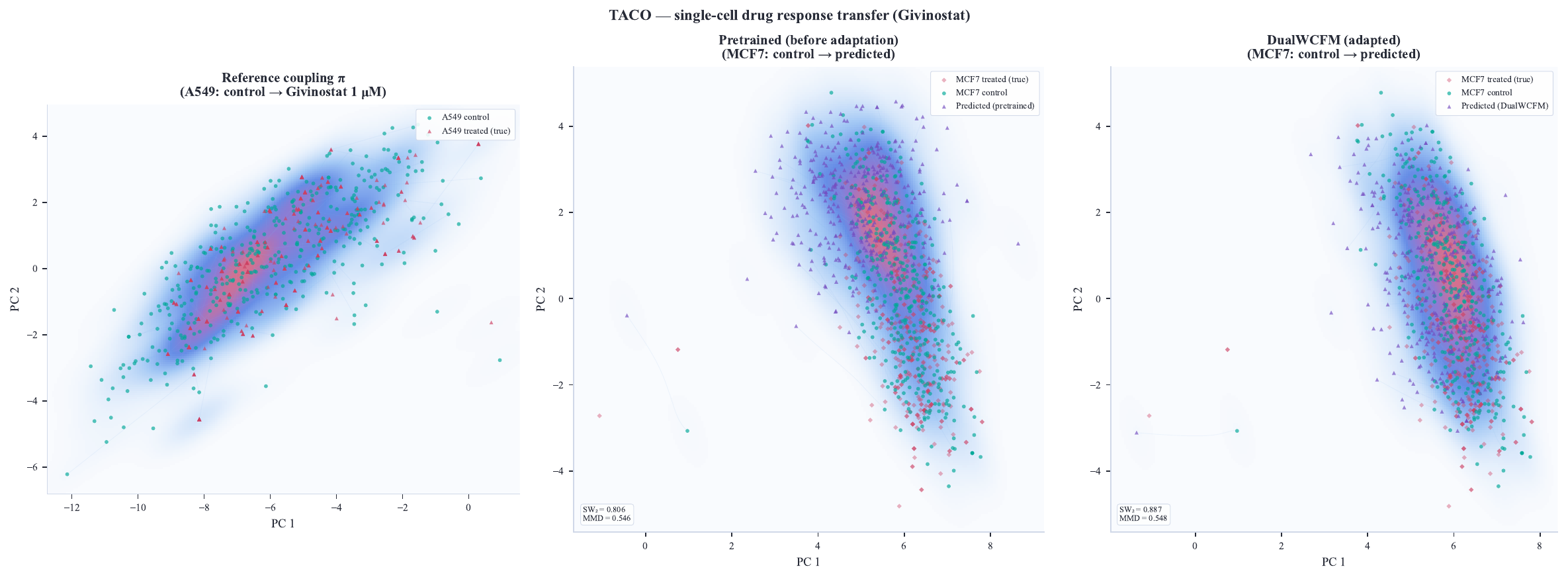}
    \caption{Perturbed setting (MCF7, 0.1\,\textmu M).}
  \end{subfigure}
  \caption{%
    \textbf{Single-cell transport (TACO).}
    PCA visualisation of transported control cells (source) onto the treated distribution (target).
    \textit{Left}: source control cells ($\mu$).
    \textit{Centre}: ground-truth treated cells ($\nu$).
    \textit{Right}: TACO-transported cells ($T_\sharp\mu$).
  }
  \label{fig:singlecell_dual_mlp}
\end{figure}

\paragraph{Convergence.}
Figure~\ref{fig:singlecell_convergence} traces the evolution of $W_1$ (target marginal) over the 100 outer TACO iterations for both settings.
Both variants start from a pretrained baseline ($W_1 \approx 0.21$) and converge rapidly within the first 20--30 iterations.

\begin{figure}[ht]
  \centering
  \includegraphics[width=\linewidth]{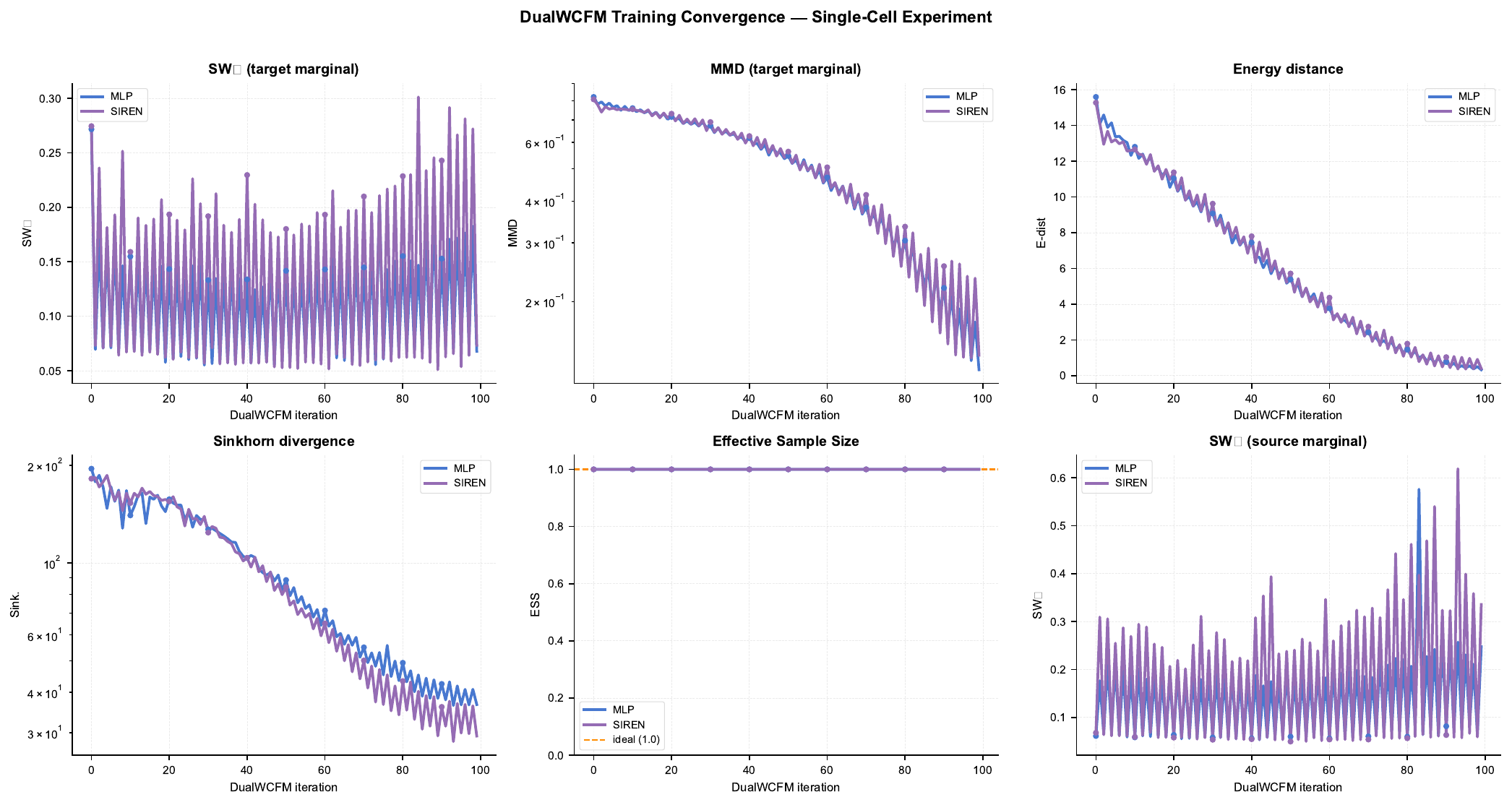}
  \caption{%
    \textbf{Convergence on single-cell transfer.}
    Target-marginal $W_1$ vs.\ outer iteration $k$ for the standard (K562) and perturbed (MCF7 0.1\,\textmu M) settings.
    Both converge to low error within $\approx$30 iterations.
  }
  \label{fig:singlecell_convergence}
\end{figure}

\paragraph{Coupling structure.}
Figure~\ref{fig:singlecell_coupling_arrows} visualises the learned entropic coupling $\pi$ by drawing arrows from source control cells to their most-probable transport targets.
The coupling concentrates mass along biologically meaningful directions in PCA space, reflecting the gene-expression shift induced by drug treatment.

\begin{figure}[ht]
  \centering
  \includegraphics[width=0.85\linewidth]{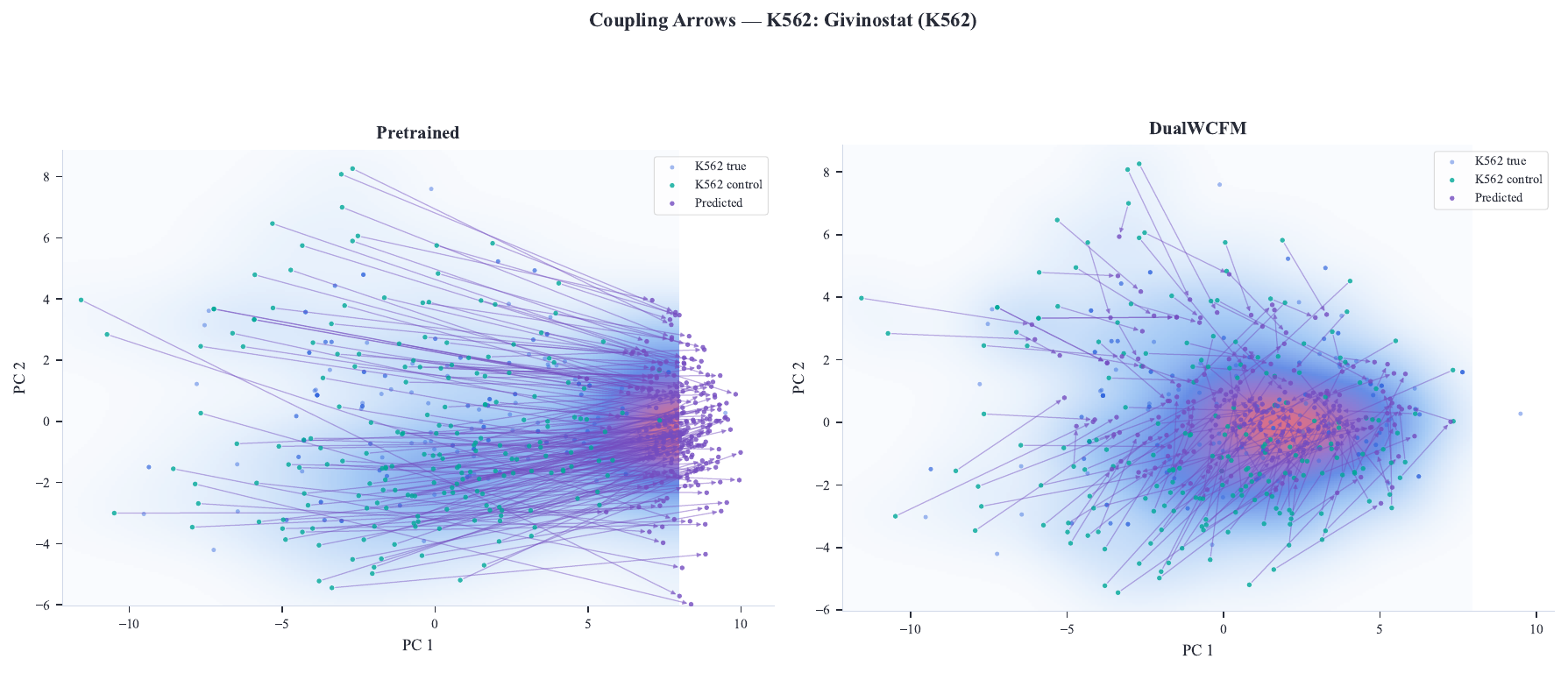}
  \caption{%
    \textbf{Entropic coupling arrows.}
    Each arrow connects a source control cell to its weighted centre of mass under the learned coupling $\pi$.
    Colours indicate transport distance.
  }
  \label{fig:singlecell_coupling_arrows}
\end{figure}

\paragraph{Iteration evolution.}
Figure~\ref{fig:singlecell_evolution} shows PCA snapshots of the transported distribution at selected TACO iterations, illustrating how the generator progressively aligns with the target.

\begin{figure}[ht]
  \centering
  \includegraphics[width=\linewidth]{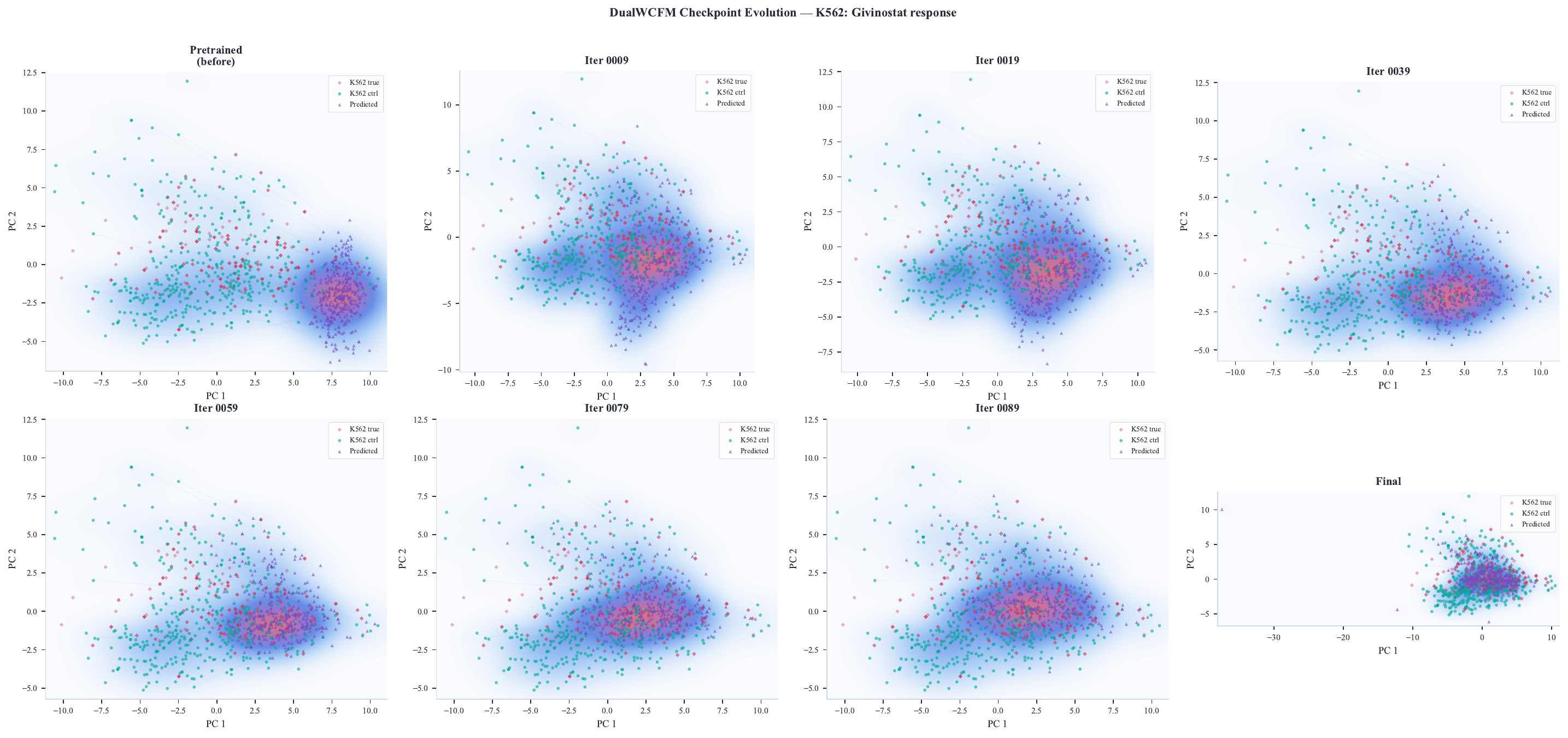}
  \caption{%
    \textbf{Evolution of transported distribution across TACO iterations.}
    Each panel shows a PCA projection of the transported cells at a given outer iteration $k$.
    The distribution converges toward the ground-truth target within $\approx$30 iterations.
  }
  \label{fig:singlecell_evolution}
\end{figure}

\paragraph{Metric comparison.}
Figure~\ref{fig:singlecell_metrics_bar} summarises all transport quality metrics for both settings as a grouped bar chart, making the quantitative gains from Table~\ref{tab:singlecell_results} visually apparent.

\begin{figure}[ht]
  \centering
  \includegraphics[width=0.85\linewidth]{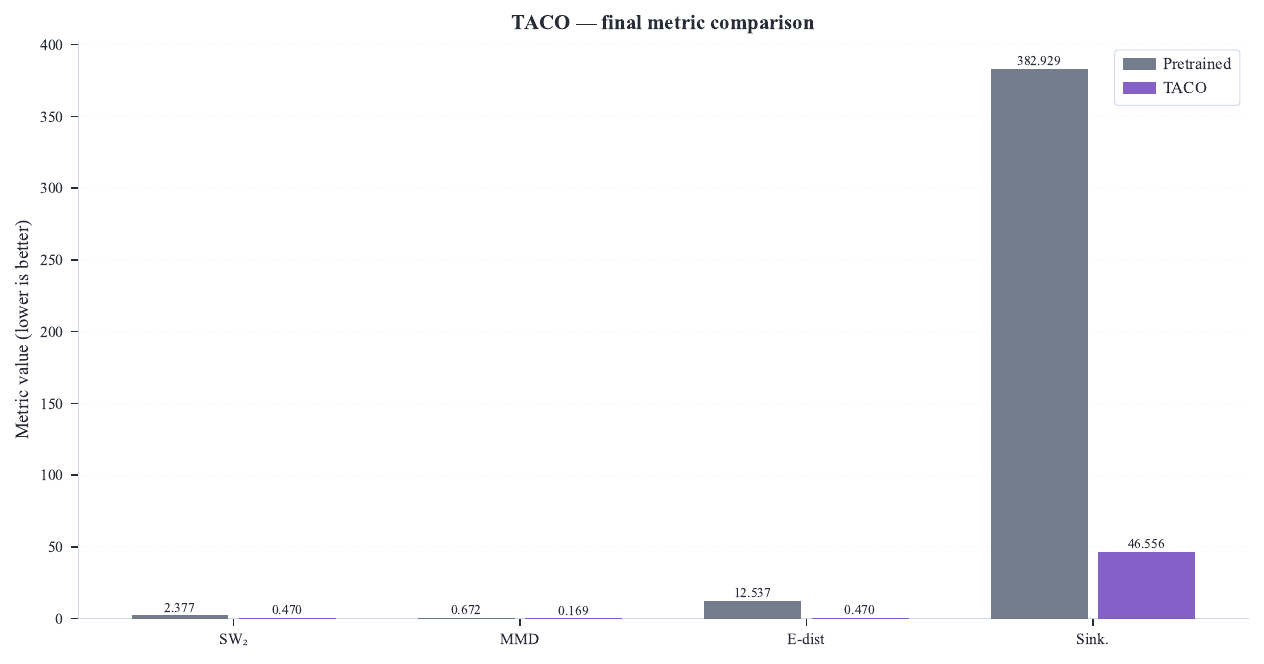}
  \caption{%
    \textbf{Single-cell transport metrics (bar chart).}
    Grouped bars compare OT-CFM (scratch), pretrained baseline, and TACO across $W_1$, $W_2$, MMD, Energy, and Sinkhorn divergence for the standard and perturbed settings.
    Lower is better for all metrics.
  }
  \label{fig:singlecell_metrics_bar}
\end{figure}

\paragraph{Discussion.}
TACO reduces $W_1$ on the target marginal by a factor of $7\times$ relative to the zero-shot pretrained baseline (from $\approx 0.21$ to $\approx 0.03$ in the standard setting).
The perturbed setting — in which the dose and cell line both differ from the reference — is equally well handled: the final $W_1$ of $0.024$ is \emph{lower} than in the standard setting.
This result is at first surprising, but has a natural geometric explanation: the MCF7 control distribution lies closer to the reference A549 distribution in PCA space than K562's control does, so despite the added dose perturbation the effective transport distance is smaller and the pretrained model provides a better initialisation for adaptation.

\section{Use of Large Language Models}
\label{app:llm}

In the interest of transparency, we describe our use of large language model (LLM) assistants in this work.

LLM tools (specifically Claude by Anthropic) were used in the following \emph{auxiliary} capacities:
\begin{itemize}[noitemsep, topsep=3pt]
  \item \textbf{Experiment infrastructure.}
        LLM assistance was used to help write boilerplate experimental code, including data-loading pipelines, training loops, and evaluation scripts.
        The design of all experiments --- which geometries to study, which baselines to include, what metrics to report, and how to structure the comparison --- was decided entirely by the authors.
  \item \textbf{Baseline setup.}
        Setting up and running the OT-CFM, LightSB, DSBM, and DSB baselines required interfacing with their respective codebases. LLM assistance was used to accelerate this integration work.
  \item \textbf{Writing and grammar.}
        LLM tools were used to proofread text and correct grammar and typographical errors.
        All mathematical derivations, theorem statements, and proof arguments were written and verified exclusively by the authors.
\end{itemize}

\noindent\textbf{No LLM contribution to core research.}
The mathematical foundations of TACO --- the cost-transfer problem formulation, the path-wise tilting approach, the covariance-type update equations, the convergence analysis, and all experimental ideas --- are entirely the authors' original contributions.
LLM assistants were not consulted for mathematical insight, algorithmic design, or experimental hypothesis generation.
This use falls under the NeurIPS 2026 LLM policy exception for writing and coding assistance that does not affect scientific originality.

%% file: plots/baseline_tables/per_experiment/2d_simple_table.tex
\begin{table}[ht]
\centering
\caption{Baselines on \textbf{2D-Simple}. \textbf{Bold}: best per metric.  \textrm{--}: not available. $\dagger$: cost-access baselines.}
\label{tab:baselines_2d_simple}
\resizebox{\linewidth}{!}{%
\begin{tabular}{lrrrrrrrr}
\toprule
Method & $\mathrm{SW}_2(\hat\nu)$ & $\mathrm{SW}_1(\hat\nu)$ & $\mathrm{SW}_2(\hat\mu)$ & $\mathrm{SW}_1(\hat\mu)$ & Map err.\ (fwd) & Sinkhorn div. & MMD & Energy dist. \\
\midrule
CFM$^\dagger$ (indep.) & 0.0370 & 0.0303 & 0.0284 & 0.0200 & 1.3872 & 0.0133 & 0.0668 & 0.0022 \\
OT-CFM$^\dagger$ (indep.) & 0.0503 & 0.0398 & 0.0324 & 0.0235 & 1.3874 & 0.0195 & 0.0360 & 0.0032 \\
LightSB$^\dagger$ (indep.) & 0.0647 & 0.0481 & 0.0438 & 0.0324 & 1.3750 & 0.0176 & 0.0514 & 0.0070 \\
DSB$^\dagger$ (indep.) & -- & -- & 0.0243 & 0.0188 & -- & -- & -- & -- \\
GENOT$^\dagger$ (indep.) & 0.5860 & 0.4748 & 0.0255 & 0.0206 & 0.9271 & 0.7338 & 0.6981 & 0.6053 \\
\addlinespace[2pt]\cdashline{1-9}[0.4pt/2pt]\addlinespace[2pt]
CFM$^\dagger$ (coupl.) & 0.0279 & 0.0180 & 0.0292 & 0.0218 & 0.0049 & 0.0113 & 0.0347 & 0.0023 \\
OT-CFM$^\dagger$ (coupl.) & 0.0279 & 0.0196 & 0.0253 & 0.0201 & 0.0179 & 0.0176 & 0.0393 & 0.0045 \\
LightSB$^\dagger$ (coupl.) & 0.0285 & 0.0211 & 0.0327 & 0.0252 & 0.0053 & \textbf{0.0106} & 0.0362 & 0.0016 \\
DSB$^\dagger$ (coupl.) & 0.0321 & 0.0249 & 0.0389 & 0.0308 & \textbf{0.0024} & 0.0125 & 0.0265 & 0.0025 \\
GENOT$^\dagger$ (coupl.) & \textbf{0.0210} & \textbf{0.0133} & \textbf{0.0240} & \textbf{0.0153} & 0.0134 & 0.0128 & \textbf{0.0135} & \textbf{0.0007} \\
\addlinespace[2pt]\cdashline{1-9}[0.4pt/2pt]\addlinespace[2pt]
\textit{Pretrained} & 14.5864 & 13.2404 & 0.0269 & 0.0184 & 20.9789 & 437.9151 & 0.9456 & 40.4188 \\
\textbf{TACO} (weighted) & 0.0236 & 0.0155 & 0.0312 & 0.0222 & 0.0206 & 0.0127 & 0.0218 & 0.0014 \\
\textbf{TACO} (implicit) & 0.0336 & 0.0255 & 0.1567 & 0.1422 & 0.2422 & 0.0133 & 0.0267 & 0.0017 \\
\bottomrule
\end{tabular}
}
\end{table}

%% file: plots/baseline_tables/per_experiment/2d_medium_table.tex
\begin{table}[ht]
\centering
\caption{Baselines on \textbf{2D-Medium}. \textbf{Bold}: best per metric.  \textrm{--}: not available. $\dagger$: cost-access baselines.}
\label{tab:baselines_2d_medium}
\resizebox{\linewidth}{!}{%
\begin{tabular}{lrrrrrrrr}
\toprule
Method & $\mathrm{SW}_2(\hat\nu)$ & $\mathrm{SW}_1(\hat\nu)$ & $\mathrm{SW}_2(\hat\mu)$ & $\mathrm{SW}_1(\hat\mu)$ & Map err.\ (fwd) & Sinkhorn div. & MMD & Energy dist. \\
\midrule
CFM$^\dagger$ (indep.) & 0.8310 & 0.2311 & 0.7157 & 0.1504 & 9.5912 & 0.0735 & 0.0511 & \textbf{0.0056} \\
OT-CFM$^\dagger$ (indep.) & \textbf{0.6233} & 0.1351 & 0.7028 & 0.1629 & 9.9819 & 0.0234 & 0.0174 & 0.0086 \\
LightSB$^\dagger$ (indep.) & 0.7885 & 0.2609 & 0.5728 & 0.1298 & 9.5284 & 0.1616 & 0.0244 & 0.0281 \\
DSB$^\dagger$ (indep.) & 2.3148 & 1.8994 & 0.7816 & 0.1720 & 6.4116 & 23.7596 & 0.0969 & 1.1297 \\
\addlinespace[2pt]\cdashline{1-9}[0.4pt/2pt]\addlinespace[2pt]
CFM$^\dagger$ (coupl.) & 0.6345 & 0.1463 & 0.9357 & 0.2539 & 0.0075 & 0.0280 & \textbf{0.0121} & 0.0163 \\
OT-CFM$^\dagger$ (coupl.) & 1.0326 & 0.3025 & 0.7652 & 0.1863 & 0.0139 & 0.0241 & 0.0587 & 0.0140 \\
LightSB$^\dagger$ (coupl.) & 0.7709 & 0.1821 & 0.6000 & 0.1221 & 0.0151 & \textbf{0.0222} & 0.0244 & 0.0079 \\
DSB$^\dagger$ (coupl.) & 0.6415 & \textbf{0.1330} & \textbf{0.4501} & \textbf{0.0817} & \textbf{0.0019} & 0.0316 & 0.0162 & 0.0286 \\
GENOT$^\dagger$ (coupl.) & 0.7500 & 0.1693 & 0.4864 & 0.0988 & 0.0029 & 0.0250 & 0.0181 & 0.0207 \\
\addlinespace[2pt]\cdashline{1-9}[0.4pt/2pt]\addlinespace[2pt]
\textit{Pretrained} & 2.1879 & 1.7006 & 0.9728 & 0.2670 & 5.4657 & 21.5863 & 0.0622 & 1.0830 \\
\textbf{TACO} (weighted) & 0.8731 & 0.2884 & 0.6418 & 0.2491 & 0.1755 & 0.8761 & 0.0289 & 0.0253 \\
\textbf{TACO} (implicit) & 0.6050 & 0.1807 & 0.6929 & 0.2824 & 0.2012 & 0.4880 & 0.0193 & 0.0113 \\
\textbf{TACO} (explicit) & 0.7543 & 0.2204 & 0.5194 & 0.2381 & 0.2724 & 0.8113 & 0.0297 & 0.0603 \\
\bottomrule
\end{tabular}
}
\end{table}

%% file: plots/baseline_tables/per_experiment/2d_complex_table.tex
\begin{table}[ht]
\centering
\caption{Baselines on \textbf{2D-Complex}. \textbf{Bold}: best per metric.  \textrm{--}: not available. $\dagger$: cost-access baselines.}
\label{tab:baselines_2d_complex}
\resizebox{\linewidth}{!}{%
\begin{tabular}{lrrrrrrrr}
\toprule
Method & $\mathrm{SW}_2(\hat\nu)$ & $\mathrm{SW}_1(\hat\nu)$ & $\mathrm{SW}_2(\hat\mu)$ & $\mathrm{SW}_1(\hat\mu)$ & Map err.\ (fwd) & Sinkhorn div. & MMD & Energy dist. \\
\midrule
CFM$^\dagger$ (indep.) & 1.9356 & 0.8295 & 1.2282 & 0.2990 & 9.5051 & 0.4017 & 0.0681 & 0.1349 \\
OT-CFM$^\dagger$ (indep.) & 1.2560 & 0.3055 & 0.8753 & 0.1661 & 8.5517 & 0.0070 & 0.0391 & 0.0523 \\
LightSB$^\dagger$ (indep.) & 1.4293 & 0.5129 & 1.0503 & 0.1993 & 9.0646 & 0.3133 & 0.0428 & 0.0730 \\
DSB$^\dagger$ (indep.) & 2.9665 & 2.2784 & 0.8882 & 0.1649 & 6.7110 & 20.6027 & 0.1729 & 1.3699 \\
\addlinespace[2pt]\cdashline{1-9}[0.4pt/2pt]\addlinespace[2pt]
CFM$^\dagger$ (coupl.) & 0.7945 & 0.1368 & 0.9007 & 0.1754 & 0.0075 & \textbf{0.0020} & 0.0505 & 0.0107 \\
OT-CFM$^\dagger$ (coupl.) & \textbf{0.4999} & \textbf{0.0583} & 1.1071 & 0.2606 & 0.0088 & 0.0020 & 0.0123 & \textbf{0.0018} \\
LightSB$^\dagger$ (coupl.) & 0.9679 & 0.1978 & 0.9133 & 0.1723 & 0.0768 & 0.0055 & \textbf{0.0099} & 0.0082 \\
DSB$^\dagger$ (coupl.) & 1.1198 & 0.2531 & \textbf{0.8544} & \textbf{0.1516} & \textbf{0.0023} & 0.0021 & 0.0253 & 0.0363 \\
GENOT$^\dagger$ (coupl.) & 1.0119 & 0.2058 & 0.8596 & 0.1654 & 0.0053 & 0.0030 & 0.0446 & 0.0224 \\
\addlinespace[2pt]\cdashline{1-9}[0.4pt/2pt]\addlinespace[2pt]
\textit{Pretrained} & 2.6361 & 1.9509 & 0.8691 & 0.1633 & 5.3409 & 29.0502 & 0.0678 & 1.1339 \\
\textbf{TACO} (weighted) & 0.5920 & 0.1614 & 1.0777 & 0.3383 & 0.0701 & 0.2590 & 0.0171 & 0.0167 \\
\textbf{TACO} (implicit) & 1.0164 & 0.2963 & 0.7388 & 0.2325 & 0.1033 & 0.2584 & 0.0383 & 0.0487 \\
\textbf{TACO} (explicit) & 1.0662 & 0.3083 & 0.9013 & 0.2752 & 0.1001 & 0.2625 & 0.0082 & 0.0083 \\
\bottomrule
\end{tabular}
}
\end{table}

%% file: plots/baseline_tables/per_experiment/2d_moon_table.tex
\begin{table}[ht]
\centering
\caption{Baselines on \textbf{2D-Moon}. \textbf{Bold}: best per metric.  \textrm{--}: not available. $\dagger$: cost-access baselines.}
\label{tab:baselines_2d_moon}
\resizebox{\linewidth}{!}{%
\begin{tabular}{lrrrrrrrr}
\toprule
Method & $\mathrm{SW}_2(\hat\nu)$ & $\mathrm{SW}_1(\hat\nu)$ & $\mathrm{SW}_2(\hat\mu)$ & $\mathrm{SW}_1(\hat\mu)$ & Sinkhorn div. & MMD & Energy dist. \\
\midrule
CFM$^\dagger$ (indep.) & 0.7922 & 0.6071 & 0.0684 & 0.0442 & 0.4002 & 0.1587 & 0.2172 \\
OT-CFM$^\dagger$ (indep.) & 3.1254 & 2.3131 & 0.0988 & 0.0661 & 1.3587 & 0.6294 & 3.0416 \\
LightSB$^\dagger$ (indep.) & 0.8032 & 0.6008 & 0.1002 & 0.0697 & 0.3416 & 0.1506 & 0.2535 \\
DSB$^\dagger$ (indep.)  & -- & -- & 0.0581 & 0.0424 & -- & -- & -- & -- \\
\addlinespace[2pt]\cdashline{1-9}[0.4pt/2pt]\addlinespace[2pt]
CFM$^\dagger$ (coupl.) & \textbf{0.1122} & \textbf{0.0883} & 0.0810 & 0.0608 & 0.0681 & \textbf{0.0338} & 0.0105 \\
OT-CFM$^\dagger$ (coupl.) & 3.0227 & 2.2762 & 0.2059 & 0.1533 & 1.0114 & 0.5591 & 2.7875 \\
LightSB$^\dagger$ (coupl.) & 0.1680 & 0.1248 & \textbf{0.0444} & \textbf{0.0343} & \textbf{0.0639} & 0.0599 & \textbf{0.0047} \\
DSB$^\dagger$ (coupl.) & 5.2916 & 4.6548 & 0.0677 & 0.0527 & 33.8119 & 0.6601 & 6.1208 \\
\addlinespace[2pt]\cdashline{1-9}[0.4pt/2pt]\addlinespace[2pt]
\textit{Pretrained} & 5.4399 & 4.7335 & 0.1905 & 0.1288 & 25.0978 & 0.9099 & 8.9572 \\
\textbf{TACO} (weighted) & 0.1032 & 0.0787 & 0.3859 & 0.3413 & 0.0500 & 0.0144 & 0.0040 \\
\textbf{TACO} (implicit) & 0.1516 & 0.1109 & 0.7182 & 0.6190 & 0.0603 & 0.0133 & 0.0078 \\
\textbf{TACO} (explicit) & 0.0664 & 0.0515 & 0.4620 & 0.3962 & 0.0401 & 0.0528 & 0.0264 \\
\bottomrule
\end{tabular}
}
\end{table}

%% file: plots/baseline_tables/per_experiment/2d_simple_perturbed_table.tex
\begin{table}[ht]
\centering
\caption{Baselines on \textbf{2D-Simple (pert.)}. \textbf{Bold}: best per metric.  \textrm{--}: not available. $\dagger$: cost-access baselines.}
\label{tab:baselines_2d_simple_perturbed}
\resizebox{\linewidth}{!}{%
\begin{tabular}{lrrrrrrrr}
\toprule
Method & $\mathrm{SW}_2(\hat\nu)$ & $\mathrm{SW}_1(\hat\nu)$ & $\mathrm{SW}_2(\hat\mu)$ & $\mathrm{SW}_1(\hat\mu)$ & Map err.\ (fwd) & Sinkhorn div. & MMD & Energy dist. \\
\midrule
CFM$^\dagger$ (indep.) & 0.0279 & 0.0208 & 0.0210 & 0.0149 & 1.4196 & 0.0220 & \textbf{0.0116} & 0.0008 \\
OT-CFM$^\dagger$ (indep.) & 0.0716 & 0.0628 & \textbf{0.0202} & \textbf{0.0126} & 1.4199 & 0.0226 & 0.0615 & 0.0123 \\
LightSB$^\dagger$ (indep.) & 0.0826 & 0.0702 & 0.0256 & 0.0181 & 1.3790 & 0.0257 & 0.0900 & 0.0223 \\
DSB$^\dagger$ (indep.) & -- & -- & 0.0263 & 0.0207 & -- & -- & -- & -- \\
GENOT$^\dagger$ (indep.) & 0.5338 & 0.4285 & 0.0264 & 0.0178 & 0.8764 & 0.5814 & 0.6428 & 0.5037 \\
\addlinespace[2pt]\cdashline{1-9}[0.4pt/2pt]\addlinespace[2pt]
CFM$^\dagger$ (coupl.) & 0.0230 & 0.0163 & 0.0297 & 0.0229 & 0.0076 & \textbf{0.0095} & 0.0318 & 0.0011 \\
OT-CFM$^\dagger$ (coupl.) & 0.0229 & 0.0159 & 0.0326 & 0.0243 & 0.0090 & 0.0115 & 0.0170 & 0.0026 \\
LightSB$^\dagger$ (coupl.) & 0.0262 & 0.0203 & 0.0240 & 0.0175 & 0.0102 & 0.0132 & 0.0210 & 0.0014 \\
DSB$^\dagger$ (coupl.) & \textbf{0.0222} & \textbf{0.0139} & 0.0242 & 0.0179 & \textbf{0.0029} & 0.0121 & 0.0142 & \textbf{0.0007} \\
GENOT$^\dagger$ (coupl.) & 0.0372 & 0.0295 & 0.0274 & 0.0196 & 0.0121 & 0.0131 & 0.0502 & 0.0009 \\
\addlinespace[2pt]\cdashline{1-9}[0.4pt/2pt]\addlinespace[2pt]
\textit{Pretrained} & 15.2926 & 13.7000 & 0.0221 & 0.0145 & 21.4855 & 459.0915 & 0.9537 & 41.4947 \\
\textbf{TACO} (weighted) & 0.0284 & 0.0233 & 0.0909 & 0.0736 & 0.1232 & 0.0122 & 0.0281 & 0.0030 \\
\textbf{TACO} (implicit) & 0.0407 & 0.0315 & 0.1603 & 0.1426 & 0.3647 & 0.0119 & 0.0468 & 0.0040 \\
\textbf{TACO} (explicit) & 0.0334 & 0.0262 & 0.0615 & 0.0524 & 0.1514 & 0.0125 & 0.0291 & 0.0011 \\
\bottomrule
\end{tabular}
}
\end{table}

%% file: plots/baseline_tables/per_experiment/2d_medium_perturbed_table.tex
\begin{table}[ht]
\centering
\caption{Baselines on \textbf{2D-Medium (pert.)}. \textbf{Bold}: best per metric.  \textrm{--}: not available. $\dagger$: cost-access baselines.}
\label{tab:baselines_2d_medium_perturbed}
\resizebox{\linewidth}{!}{%
\begin{tabular}{lrrrrrrrr}
\toprule
Method & $\mathrm{SW}_2(\hat\nu)$ & $\mathrm{SW}_1(\hat\nu)$ & $\mathrm{SW}_2(\hat\mu)$ & $\mathrm{SW}_1(\hat\mu)$ & Map err.\ (fwd) & Sinkhorn div. & MMD & Energy dist. \\
\midrule
CFM$^\dagger$ (indep.) & 0.6411 & 0.1823 & 0.7368 & 0.1624 & 9.8108 & 0.1038 & 0.0201 & 0.0115 \\
OT-CFM$^\dagger$ (indep.) & 0.7730 & 0.1862 & 0.7368 & 0.1758 & 10.0271 & 0.0288 & 0.0231 & 0.0108 \\
LightSB$^\dagger$ (indep.) & 0.8088 & 0.2217 & 0.6822 & 0.1656 & 9.9250 & 0.0340 & 0.0272 & \textbf{0.0077} \\
DSB$^\dagger$ (indep.) & 1.6203 & 1.3661 & 0.6431 & 0.1453 & 7.8346 & 5.5589 & 0.1104 & 0.8014 \\
\addlinespace[2pt]\cdashline{1-9}[0.4pt/2pt]\addlinespace[2pt]
CFM$^\dagger$ (coupl.) & \textbf{0.4315} & \textbf{0.0782} & \textbf{0.3877} & \textbf{0.0693} & 0.0127 & \textbf{0.0207} & 0.0260 & 0.0221 \\
OT-CFM$^\dagger$ (coupl.) & 0.7816 & 0.1870 & 0.7568 & 0.1821 & 0.0111 & 0.0243 & 0.0454 & 0.0097 \\
LightSB$^\dagger$ (coupl.) & 0.5472 & 0.1162 & 0.5643 & 0.1090 & 0.0202 & 0.0259 & \textbf{0.0186} & 0.0093 \\
DSB$^\dagger$ (coupl.) & 0.9932 & 0.2818 & 0.8476 & 0.2182 & \textbf{0.0023} & 0.0288 & 0.0415 & 0.0186 \\
GENOT$^\dagger$ (coupl.) & 0.6825 & 0.1525 & 0.5047 & 0.1001 & 0.0065 & 0.0250 & 0.0459 & 0.0089 \\
\addlinespace[2pt]\cdashline{1-9}[0.4pt/2pt]\addlinespace[2pt]
\textit{Pretrained} & 2.0647 & 1.7206 & 0.5679 & 0.1195 & 6.2529 & 16.5827 & 0.0651 & 1.0207 \\
\textbf{TACO} (weighted) & 0.5000 & 0.1604 & 0.6526 & 0.4161 & 0.6592 & 0.3736 & 0.0360 & 0.0128 \\
\textbf{TACO} (implicit) & 0.5299 & 0.1715 & 0.6916 & 0.4681 & 0.7467 & 0.5871 & 0.0278 & 0.0172 \\
\textbf{TACO} (explicit) & 0.6691 & 0.2252 & 0.7208 & 0.4608 & 0.6765 & 0.7521 & 0.0167 & 0.0106 \\
\bottomrule
\end{tabular}
}
\end{table}

%% file: plots/baseline_tables/per_experiment/2d_complex_perturbed_table.tex
\begin{table}[ht]
\centering
\caption{Baselines on \textbf{2D-Complex (pert.)}. \textbf{Bold}: best per metric.  \textrm{--}: not available. $\dagger$: cost-access baselines.}
\label{tab:baselines_2d_complex_perturbed}
\resizebox{\linewidth}{!}{%
\begin{tabular}{lrrrrrrrr}
\toprule
Method & $\mathrm{SW}_2(\hat\nu)$ & $\mathrm{SW}_1(\hat\nu)$ & $\mathrm{SW}_2(\hat\mu)$ & $\mathrm{SW}_1(\hat\mu)$ & Map err.\ (fwd) & Sinkhorn div. & MMD & Energy dist. \\
\midrule
CFM$^\dagger$ (indep.) & 1.6794 & 0.6591 & 0.7693 & 0.1176 & 10.9010 & 0.7101 & 0.0490 & 0.0359 \\
OT-CFM$^\dagger$ (indep.) & \textbf{0.6992} & \textbf{0.1015} & \textbf{0.7021} & \textbf{0.1145} & 12.5855 & 0.0027 & 0.0156 & \textbf{0.0048} \\
LightSB$^\dagger$ (indep.) & 1.3858 & 0.4345 & 1.1896 & 0.2487 & 11.6316 & 0.6913 & 0.0392 & 0.0454 \\
DSB$^\dagger$ (indep.) & 3.0211 & 2.1737 & 0.7827 & 0.1334 & 8.0323 & 41.6791 & 0.1058 & 1.2610 \\
\addlinespace[2pt]\cdashline{1-9}[0.4pt/2pt]\addlinespace[2pt]
CFM$^\dagger$ (coupl.) & 1.2861 & 0.3417 & 1.2755 & 0.3253 & 0.0114 & 0.0020 & 0.0594 & 0.0773 \\
OT-CFM$^\dagger$ (coupl.) & 1.2354 & 0.3130 & 1.1396 & 0.2816 & 0.0110 & 0.0019 & 0.0420 & 0.0157 \\
LightSB$^\dagger$ (coupl.) & 1.0497 & 0.2527 & 0.9314 & 0.1806 & 0.1322 & 0.0088 & 0.0256 & 0.0083 \\
DSB$^\dagger$ (coupl.) & 1.2171 & 0.2987 & 1.2034 & 0.2745 & \textbf{0.0034} & 0.0027 & 0.0565 & 0.0253 \\
GENOT$^\dagger$ (coupl.) & 0.7074 & 0.1119 & 1.1141 & 0.2644 & 0.0068 & \textbf{0.0018} & \textbf{0.0105} & 0.0070 \\
\addlinespace[2pt]\cdashline{1-9}[0.4pt/2pt]\addlinespace[2pt]
\textit{Pretrained} & 2.8105 & 2.0669 & 0.8379 & 0.1497 & 6.7445 & 41.7735 & 0.0801 & 1.1699 \\
\textbf{TACO} (weighted) & 1.2528 & 0.3981 & 0.6264 & 0.2982 & 0.3487 & 0.2912 & 0.0283 & 0.0254 \\
\textbf{TACO} (implicit) & 1.1341 & 0.3541 & 0.7440 & 0.3251 & 0.3189 & 0.2624 & 0.0045 & 0.0189 \\
\textbf{TACO} (explicit) & 0.8186 & 0.2196 & 0.7247 & 0.2901 & 0.2781 & 0.3113 & 0.0279 & 0.0258 \\
\bottomrule
\end{tabular}
}
\end{table}

%% file: plots/baseline_tables/per_experiment/2d_moon_perturbed_table.tex
\begin{table}[ht]
\centering
\caption{Baselines on \textbf{2D-Moon (pert.)}. \textbf{Bold}: best per metric. \textrm{--}: not available. $\dagger$: cost-access baselines.}
\label{tab:baselines_2d_moon_perturbed}
\resizebox{\linewidth}{!}{%
\begin{tabular}{lrrrrrrr}
\toprule
Method & $\mathrm{SW}_2(\hat\nu)$ & $\mathrm{SW}_1(\hat\nu)$ & $\mathrm{SW}_2(\hat\mu)$ & $\mathrm{SW}_1(\hat\mu)$ & Sinkhorn div. & MMD & Energy dist. \\
\midrule
CFM$^\dagger$ (indep.) & 0.8431 & 0.6098 & 0.1390 & 0.0973 & 0.8004 & 0.1378 & 0.2378 \\
OT-CFM$^\dagger$ (indep.) & 1.7087 & 1.1989 & 0.1060 & 0.0774 & 0.2040 & 0.3122 & 0.9720 \\
LightSB$^\dagger$ (indep.) & 0.9611 & 0.7243 & 0.0718 & 0.0504 & 0.8518 & 0.1978 & 0.2751 \\
DSB$^\dagger$ (indep.) & 1.1999 & 0.9514 & \textbf{0.0509} & \textbf{0.0375} & 5.2869 & 0.1779 & 0.4286 \\
\addlinespace[2pt]\cdashline{1-8}[0.4pt/2pt]\addlinespace[2pt]
CFM$^\dagger$ (coupl.) & \textbf{0.1321} & \textbf{0.1000} & 0.0627 & 0.0473 & \textbf{0.0957} & \textbf{0.0156} & \textbf{0.0114} \\
OT-CFM$^\dagger$ (coupl.) & 1.7256 & 1.2679 & 0.0998 & 0.0746 & 0.2648 & 0.3316 & 0.8150 \\
LightSB$^\dagger$ (coupl.) & 0.2584 & 0.1852 & 0.1422 & 0.0959 & 0.1029 & 0.0433 & 0.0317 \\
DSB$^\dagger$ (coupl.) & -- & -- & 0.1854 & 0.1373 & -- & 0.5161 & -- \\
\addlinespace[2pt]\cdashline{1-8}[0.4pt/2pt]\addlinespace[2pt]
\textit{Pretrained} & 4.0099 & 3.4039 & 0.0950 & 0.0676 & 19.0597 & 0.7372 & 5.2649 \\
\textbf{TACO} (weighted) & 0.0906 & 0.0678 & 0.1789 & 0.1493 & 0.0445 & 0.0227 & 0.0100 \\
\textbf{TACO} (implicit) & 0.0789 & 0.0570 & 0.4823 & 0.4020 & 0.0504 & 0.0184 & 0.0044 \\
\textbf{TACO} (explicit) & 0.1281 & 0.0911 & 0.3239 & 0.2817 & 0.0493 & 0.0126 & 0.0078 \\
\bottomrule
\end{tabular}
}
\end{table}

%% file: refs_complete.bib
@book{villani2021topics,
  author    = {Villani, C{\'e}dric},
  title     = {Topics in Optimal Transportation},
  series    = {Graduate Studies in Mathematics},
  volume    = {58},
  publisher = {American Mathematical Society},
  address   = {Providence, RI},
  year      = {2003}
}

@book{peyre2019computational,
  title={Computational optimal transport: With applications to data science},
  author={Peyr{\'e}, Gabriel and Cuturi, Marco},
  year={2019},
  publisher={Now Foundations and Trends}
}

@inproceedings{cuturi2013sinkhorn,
  author    = {Cuturi, Marco},
  title     = {Sinkhorn Distances: Lightspeed Computation of Optimal Transport},
  booktitle = {Advances in Neural Information Processing Systems},
  volume    = {26},
  pages     = {2292--2300},
  year      = {2013},
  publisher = {Curran Associates}
}

@article{nutz2021introduction,
  author    = {Nutz, Marcel},
  title     = {Introduction to Entropic Optimal Transport},
  journal   = {Lecture Notes, Columbia University},
  year      = {2021},
  note      = {Available at \url{https://www.math.columbia.edu/~mnutz/docs/EOT_lecture_notes.pdf}}
}

@article{carlier2017convergence,
  title={Convergence of entropic schemes for optimal transport and gradient flows},
  author={Carlier, Guillaume and Duval, Vincent and Peyr{\'e}, Gabriel and Schmitzer, Bernhard},
  journal={SIAM Journal on Mathematical Analysis},
  volume={49},
  number={2},
  pages={1385--1418},
  year={2017},
  publisher={SIAM}
}

@article{conforti2019second,
  author    = {Conforti, Giovanni},
  title     = {A Second Order Equation for {Schr{\"o}dinger} Bridges with Applications to the Hot Gas Experiment and Entropic Transportation Cost},
  journal   = {Probability Theory and Related Fields},
  volume    = {174},
  number    = {1},
  pages     = {1--47},
  year      = {2019}
}

@article{gonzalez2024identifiability,
  author    = {Gonz{\'a}lez-Sanz, Alberto and Groppe, Michel and Munk, Axel},
  title     = {Identifiability of the Optimal Transport Cost on Finite Spaces},
  journal   = {arXiv preprint arXiv:2410.23146},
  year      = {2024}
}

@inproceedings{seguy2018largescale,
  author    = {Seguy, Vivien and Damodaran, Bharath Bhushan and Flamary, R{\'e}mi and Courty, Nicolas and Rolet, Antoine and Blondel, Mathieu},
  title     = {Large-Scale Optimal Transport and Mapping Estimation},
  booktitle = {International Conference on Learning Representations},
  year      = {2018}
}

@inproceedings{kassraie2024progot,
  author    = {Kassraie, Parnian and Pooladian, Aram-Alexandre and Klein, Michal and Thornton, James and Niles-Weed, Jonathan and Cuturi, Marco},
  title     = {Progressive Entropic Optimal Transport Solvers},
  booktitle = {Advances in Neural Information Processing Systems},
  volume    = {37},
  year      = {2024}
}

@inproceedings{gushchin2023entropic,
  author    = {Gushchin, Nikita and Kolesov, Alexander and Korotin, Alexander and Vetrov, Dmitry and Burnaev, Evgeny},
  title     = {Entropic Neural Optimal Transport via Diffusion Processes},
  booktitle = {Advances in Neural Information Processing Systems},
  volume    = {36},
  year      = {2023},
  note      = {Oral presentation}
}

@inproceedings{klein2024genot,
  author    = {Klein, Dominik and Prigent, Hugues and Cuturi, Marco and Peyré, Gabriel and Theis, Fabian J.},
  title     = {{GENOT}: Entropic ({Gromov}) {Wasserstein} Flow Matching with Applications to Single-Cell Genomics},
  booktitle = {Advances in Neural Information Processing Systems},
  volume    = {37},
  year      = {2024},
  note      = {arXiv:2310.09254}
}

@inproceedings{debortoli2021diffusion,
  author    = {{De Bortoli}, Valentin and Thornton, James and Heng, Jeremy and Doucet, Arnaud},
  title     = {Diffusion {Schr{\"o}dinger} Bridge with Applications to Score-Based Generative Modeling},
  booktitle = {Advances in Neural Information Processing Systems},
  volume    = {34},
  pages     = {17695--17709},
  year      = {2021}
}

@inproceedings{shi2023diffusion,
  author    = {Shi, Yuyang and {De Bortoli}, Valentin and Campbell, Andrew and Doucet, Arnaud},
  title     = {Diffusion {Schr{\"o}dinger} Bridge Matching},
  booktitle = {Advances in Neural Information Processing Systems},
  volume    = {36},
  year      = {2023}
}

@misc{pachebat2025iterative,
      title={Iterative Tilting for Diffusion Fine-Tuning}, 
      author={Jean Pachebat and Giovanni Conforti and Alain Durmus and Yazid Janati},
      year={2025},
      eprint={2512.03234},
      archivePrefix={arXiv},
      primaryClass={stat.ML},
      url={https://arxiv.org/abs/2512.03234}, 
}

@inproceedings{korotin2024light,
  author    = {Gushchin, Nikita and Kholkin, Sergei and Burnaev, Evgeny and Korotin, Alexander},
  title     = {Light and Optimal {Schr{\"o}dinger} Bridge Matching},
  booktitle = {International Conference on Machine Learning},
  year      = {2024},
  note      = {arXiv:2402.03207}
}

@inproceedings{debortoli2024sbflow,
  author    = {De Bortoli, Valentin and Korshunova, Iryna and Mnih, Andriy and Doucet, Arnaud},
  title     = {{Schr{\"o}dinger} Bridge Flow for Unpaired Data Translation},
  booktitle = {Advances in Neural Information Processing Systems},
  volume    = {37},
  year      = {2024},
  note      = {Spotlight. arXiv:2409.09347}
}

@inproceedings{lipman2022flow,
  author    = {Lipman, Yaron and Chen, Ricky T. Q. and Ben-Hamu, Heli and Nickel, Maximilian and Le, Matt},
  title     = {Flow Matching for Generative Modeling},
  booktitle = {International Conference on Learning Representations},
  year      = {2023},
  note      = {arXiv:2210.02747}
}

@article{lipman2023flowmatchinggenerativemodeling,
  author    = {Lipman, Yaron and Chen, Ricky T. Q. and Ben-Hamu, Heli and Nickel, Maximilian and Le, Matt},
  title     = {Flow Matching for Generative Modeling},
  journal   = {arXiv preprint arXiv:2210.02747},
  year      = {2022}
}

@article{tong2023improving,
  title={Improving and generalizing flow-based generative models with minibatch optimal transport},
  author={Tong, Alexander and Fatras, Kilian and Malkin, Nikolay and Huguet, Guillaume and Zhang, Yanlei and Rector-Brooks, Jarrid and Wolf, Guy and Bengio, Yoshua},
  journal={arXiv preprint arXiv:2302.00482},
  year={2023}
}

@inproceedings{liu2022flow,
  author    = {Liu, Xingchao and Gong, Chengyue and Liu, Qiang},
  title     = {Flow Straight and Fast: Learning to Generate and Transfer Data with Rectified Flow},
  booktitle = {International Conference on Learning Representations},
  year      = {2023},
  note      = {Spotlight. arXiv:2209.03003}
}

@inproceedings{albergo2022building,
  author    = {Albergo, Michael S. and Vanden-Eijnden, Eric},
  title     = {Building Normalizing Flows with Stochastic Interpolants},
  booktitle = {International Conference on Learning Representations},
  year      = {2023},
  note      = {arXiv:2209.15571}
}

@article{ho2020denoising,
  title={Denoising diffusion probabilistic models},
  author={Ho, Jonathan and Jain, Ajay and Abbeel, Pieter},
  journal={Advances in neural information processing systems},
  volume={33},
  pages={6840--6851},
  year={2020}
}

@inproceedings{song2020score,
  author    = {Song, Yang and Sohl-Dickstein, Jascha and Kingma, Diederik P. and Kumar, Abhishek and Ermon, Stefano and Poole, Ben},
  title     = {Score-Based Generative Modeling through Stochastic Differential Equations},
  booktitle = {International Conference on Learning Representations},
  year      = {2021},
  note      = {Outstanding Paper Award. arXiv:2011.13456}
}

@inproceedings{rezende2015variational,
  author    = {Rezende, Danilo J. and Mohamed, Shakir},
  title     = {Variational Inference with Normalizing Flows},
  booktitle = {International Conference on Machine Learning},
  volume    = {37},
  pages     = {1530--1538},
  year      = {2015}
}

@inproceedings{amos2023meta,
  author    = {Amos, Brandon and Luise, Giulia and Cohen, Samuel and Redko, Ievgen},
  title     = {Meta Optimal Transport},
  booktitle = {International Conference on Machine Learning},
  volume    = {202},
  pages     = {791--813},
  year      = {2023}
}

@article{bunne2022condot,
  title={Supervised training of conditional monge maps},
  author={Bunne, Charlotte and Krause, Andreas and Cuturi, Marco},
  journal={Advances in Neural Information Processing Systems},
  volume={35},
  pages={6859--6872},
  year={2022}
}

@article{potaptchik2025tilt,
  title={Tilt Matching for Scalable Sampling and Fine-Tuning},
  author={Potaptchik, Peter and Lee, Cheuk-Kit and Albergo, Michael S},
  journal={arXiv preprint arXiv:2512.21829},
  year={2025}
}

@article{bunne2023cellot,
  title={Learning single-cell perturbation responses using neural optimal transport},
  author={Bunne, Charlotte and Stark, Stefan G and Gut, Gabriele and Del Castillo, Jacobo Sarabia and Levesque, Mitch and Lehmann, Kjong-Van and Pelkmans, Lucas and Krause, Andreas and R{\"a}tsch, Gunnar},
  journal={Nature methods},
  volume={20},
  number={11},
  pages={1759--1768},
  year={2023},
  publisher={Nature Publishing Group US New York}
}

@inproceedings{zhu2017unpaired,
  title={Unpaired image-to-image translation using cycle-consistent adversarial networks},
  author={Zhu, Jun-Yan and Park, Taesung and Isola, Phillip and Efros, Alexei A},
  booktitle={Proceedings of the IEEE international conference on computer vision},
  pages={2223--2232},
  year={2017}
}

@inproceedings{ronneberger2015unet,
  title={U-net: Convolutional networks for biomedical image segmentation},
  author={Ronneberger, Olaf and Fischer, Philipp and Brox, Thomas},
  booktitle={International Conference on Medical image computing and computer-assisted intervention},
  pages={234--241},
  year={2015},
  organization={Springer}
}

@article{henry2019martingale,
  title={From (Martingale) Schrodinger bridges to a new class of Stochastic Volatility Models},
  author={Henry-Labordere, Pierre},
  journal={arXiv preprint arXiv:1904.04554},
  year={2019}
}

@article{chen2026convergence,
  title={Convergence of Sinkhorn’s algorithm for entropic martingale optimal transport problem},
  author={Chen, Fan and Conforti, Giovanni and Ren, Zhenjie and Wang, Xiaozhen},
  journal={Mathematics of Operations Research},
  year={2026},
  publisher={INFORMS}
}

@article{kazeykina2025entropic,
  title={Entropic Optimal Transport Problem with Convex Functional Cost},
  author={Kazeykina, Anna and Ren, Zhenjie and Wang, Xiaozhen and Zhang, Yufei},
  journal={arXiv preprint arXiv:2503.11843},
  year={2025}
}

@article{bertsekas1997nonlinear,
  title={Nonlinear programming},
  author={Bertsekas, Dimitri P},
  journal={Journal of the Operational Research Society},
  volume={48},
  number={3},
  pages={334--334},
  year={1997},
  publisher={Taylor \& Francis}
}

@article{tseng2001convergence,
  title={Convergence of a block coordinate descent method for nondifferentiable minimization},
  author={Tseng, Paul},
  journal={Journal of optimization theory and applications},
  volume={109},
  number={3},
  pages={475--494},
  year={2001},
  publisher={Springer}
}

@inproceedings{makkuva2020optimal,
  title={Optimal transport mapping via input convex neural networks},
  author={Makkuva, Ashok and Taghvaei, Amirhossein and Oh, Sewoong and Lee, Jason},
  booktitle={International Conference on Machine Learning},
  pages={6672--6681},
  year={2020},
  organization={PMLR}
}

@book{evans2015measure,
    title     = {Measure Theory and Fine Properties of Functions},
    author    = {Evans, Lawrence C. and Gariepy, Ronald F.},
    year      = {2015},
    edition   = {Revised},
    publisher = {CRC Press},
    address   = {Boca Raton, FL},
    series    = {Textbooks in Mathematics}
}

@article{gazdieva2024light,
  title={Light unbalanced optimal transport},
  author={Gazdieva, Milena and Asadulaev, Arip and Burnaev, Evgeny and Korotin, Alexander},
  journal={Advances in Neural Information Processing Systems},
  volume={37},
  pages={93907--93938},
  year={2024}
}

@book{moral2004feynman,
  title={Feynman-Kac formulae: genealogical and interacting particle systems with applications},
  author={Moral, Pierre},
  year={2004},
  publisher={Springer}
}

@article{bunne2022supervised,
  title={Supervised training of conditional monge maps},
  author={Bunne, Charlotte and Krause, Andreas and Cuturi, Marco},
  journal={Advances in Neural Information Processing Systems},
  volume={35},
  pages={6859--6872},
  year={2022}
}

@article{kim2022maximum,
  title={Maximum likelihood training of implicit nonlinear diffusion model},
  author={Kim, Dongjun and Na, Byeonghu and Kwon, Se Jung and Lee, Dongsoo and Kang, Wanmo and Moon, Il-Chul},
  journal={Advances in neural information processing systems},
  volume={35},
  pages={32270--32284},
  year={2022}
}

@inproceedings{heusel2017gans,
 author = {Heusel, Martin and Ramsauer, Hubert and Unterthiner, Thomas and Nessler, Bernhard and Hochreiter, Sepp},
 booktitle = {Advances in Neural Information Processing Systems},
 editor = {I. Guyon and U. Von Luxburg and S. Bengio and H. Wallach and R. Fergus and S. Vishwanathan and R. Garnett},
 pages = {},
 publisher = {Curran Associates, Inc.},
 title = {GANs Trained by a Two Time-Scale Update Rule Converge to a Local Nash Equilibrium},
 url = {https://proceedings.neurips.cc/paper_files/paper/2017/file/8a1d694707eb0fefe65871369074926d-Paper.pdf},
 volume = {30},
 year = {2017}
}

@misc{binkowski2018demystifying,
      title={Demystifying MMD GANs}, 
      author={Mikołaj Bińkowski and Danica J. Sutherland and Michael Arbel and Arthur Gretton},
      year={2018},
      eprint={1801.01401},
      archivePrefix={arXiv},
      primaryClass={stat.ML}
}

@inproceedings{sitzmann2020implicit,
 author = {Sitzmann, Vincent and Martel, Julien and Bergman, Alexander and Lindell, David and Wetzstein, Gordon},
 booktitle = {Advances in Neural Information Processing Systems},
 editor = {H. Larochelle and M. Ranzato and R. Hadsell and M.F. Balcan and H. Lin},
 pages = {7462--7473},
 publisher = {Curran Associates, Inc.},
 title = {Implicit Neural Representations with Periodic Activation Functions},
 url = {https://proceedings.neurips.cc/paper_files/paper/2020/file/53c04118df112c13a8c34b38343b9c10-Paper.pdf},
 volume = {33},
 year = {2020}
}

@inproceedings{rahaman2019spectral,
  title={On the spectral bias of neural networks},
  author={Rahaman, Nasim and Baratin, Aristide and Arpit, Devansh and Draxler, Felix and Lin, Min and Hamprecht, Fred and Bengio, Yoshua and Courville, Aaron},
  booktitle={International conference on machine learning},
  pages={5301--5310},
  year={2019},
  organization={PMLR}
}

@article{srivatsan2020massively,
  title={Massively multiplex chemical transcriptomics at single-cell resolution},
  author={Srivatsan, Sanjay R and McFaline-Figueroa, Jos{\'e} L and Ramani, Vijay and Saunders, Lauren and Cao, Junyue and Packer, Jonathan and Pliner, Hannah A and Jackson, Dana L and Daza, Riza M and Christiansen, Lena and others},
  journal={Science},
  volume={367},
  number={6473},
  pages={45--51},
  year={2020},
  publisher={American Association for the Advancement of Science}
}

@article{bonneel2015sliced,
  author  = {Bonneel, Nicolas and Rabin, Julien and Peyr{\'e}, Gabriel and Pfister, Hanspeter},
  title   = {Sliced and Radon Wasserstein Barycenters of Measures},
  journal = {Journal of Mathematical Imaging and Vision},
  year    = {2015},
  volume  = {51},
  number  = {1},
  pages   = {22--45},
  doi     = {10.1007/s10851-014-0506-3},
  url     = {https://doi.org/10.1007/s10851-014-0506-3}
}

@InProceedings{feydy2019interpolating,
  title = 	 {Interpolating between Optimal Transport and MMD using Sinkhorn Divergences},
  author =       {Feydy, Jean and S\'{e}journ\'{e}, Thibault and Vialard, Fran\c{c}ois-Xavier and Amari, Shun-ichi and Trouve, Alain and Peyr\'{e}, Gabriel},
  booktitle = 	 {Proceedings of the Twenty-Second International Conference on Artificial Intelligence and Statistics},
  pages = 	 {2681--2690},
  year = 	 {2019},
  editor = 	 {Chaudhuri, Kamalika and Sugiyama, Masashi},
  volume = 	 {89},
  series = 	 {Proceedings of Machine Learning Research},
  month = 	 {16--18 Apr},
  publisher =    {PMLR},
  url = 	 {https://proceedings.mlr.press/v89/feydy19a.html}
}

@misc{kingma2014adam,
      title={Adam: A Method for Stochastic Optimization}, 
      author={Diederik P. Kingma and Jimmy Ba},
      year={2014},
      eprint={1412.6980},
      archivePrefix={arXiv},
      primaryClass={cs.LG},
      url={https://arxiv.org/abs/1412.6980}, 
}
